\RequirePackage{fix-cm}
\documentclass{svjour3}                     
\smartqed  
\usepackage[margin=1in]{geometry}
\usepackage{graphicx}
\usepackage{amsmath} 
\usepackage{url}
\usepackage{authblk} 
\usepackage{amssymb}
\usepackage{amsfonts}
\usepackage{verbatim}
\usepackage{mathrsfs}
\usepackage{subfigure}
\usepackage{mathtools}
\usepackage{nicefrac}
\usepackage{xcolor}

\newcommand{\rr}{\mathbb{R}}

\newcommand{\cc}{\mathbb{C}}

\newcommand{\nn}{\mathbb{N}}

\newcommand{\TT}[1]{\mathcal{T} \left( #1 \right)}

\newcommand{\overbar}[1]{\mkern 1.5mu\overline{\mkern-1.5mu#1\mkern-1.5mu}\mkern 1.5mu}

\usepackage{todonotes}

\newcommand{\norm}[1]{\left|\left| #1 \right|\right|}
\newcommand{\abs}[1]{\left| #1 \right|}

\newcommand{\setof}[1]{\left\{#1\right\}}

\newtheorem{prop}[theorem]{Proposition}
\begin{document}

\title{Homoclinic dynamics in a restricted four body problem
\thanks{The second author was partially supported by NSF grant 
DMS-1813501.
Both authors were partially supported by NSF grant DMS-1700154 
and by  the Alfred P. Sloan Foundation grant G-2016-7320}
}
%

\subtitle{transverse connections for the saddle-focus equilibrium solution set}


\author{Shane Kepley         \and
        J.D. Mireles James 
}


\institute{S. Kepley \at
              Rutgers University, Department of Mathematics \\
              \email{sk2011@math.rutgers.edu}           
           \and
           J.D. Mireles James \at
              Florida Atlantic University, Department of Mathematical Sciences \\
              \email{jmirelesjames@fau.edu}
}

\date{Received: date / Accepted: date}

\maketitle

\begin{abstract}
We describe a method for computing an atlas for the stable or unstable 
manifold attached to an equilibrium point, and implement the method
for the saddle-focus libration points of the planar equilateral restricted 
four body problem.  
We employ the method at the maximally symmetric case of equal masses,
where we compute atlases for both the stable and unstable manifolds. 
The resulting atlases are comprised of thousands of 
individual chart maps, with each chart represented by a two variable 
Taylor polynomial. 
Post-processing the atlas data yields approximate intersections of the 
invariant manifolds, which we refine via a shooting method for an appropriate 
two point boundary value problem.  Finally, we apply numerical continuation to 
some of the BVP problems.  
This breaks the symmetries and leads to connecting orbits  
for some non-equal values of the primary masses.  
\end{abstract}



%
%
%

\section{Introduction} \label{intro}
Illuminating studies by Darwin, Str\"{o}mgren, and Moulton in the first decades of the Twentieth 
Century established the importance 
of numerical calculations in the qualitative theory of Hamiltonian systems  
\cite{MR1554890,stromgrenRef,moultonBook}.  In particular their work 
gave new insights into the orbit structure of 
the circular restricted three body problem (CRTBP), 
a problem already immortalized by Poincar\'{e}.
Interest in the CRTBP was reinvigorated in the 1960's with the inauguration of the space race and 
a number of authors including Szebehely, Nacozy, and Flandern
\cite{szebehelyOnStromgren,szebehelyTriangularPoints} 
harnessed the newly available power of digital computing to 
settle some questions raised by Str\"{o}mgren.
The interested reader will find a delightful retelling of this story
with many additional references in the book of Szebhely \cite{theoryOfOrbits}.

Motivated by the works just mentioned,  in 1973 Henrard proved a theorem settling a 
conjecture of Str\"{o}mgren about the role of asymptotic orbits. 
 More precisely, Henrard showed that the existence of a transverse homoclinic for a
saddle-focus equilibrium in a two degrees of freedom Hamiltonian system implies the existence of a  
tube of periodic orbits parameterized by energy and accumulating to the homoclinic \cite{MR0365628}. 
In the same paper he showed that the period of the orbits in the family goes to infinity
and their stability changes infinitely often as they accumulate to the homoclinic.  
This phenomena was called the \textit{blue sky catastrophe} by Abraham
\cite{MR813508} and has been studied by a number of authors including 
L.P. Shilnikov, A.L. Shilnikov, and Turaev \cite{MR3253906}, Devaney
\cite{MR0431274}. 

In 1976 it was further shown by Devaney that such a transverse homoclinic --
again for a saddle-focus in a two degrees of freedom Hamiltonian system -- implies the existence 
of  chaotic dynamics in the energy level of the equilibrium \cite{MR0442990}.  
See also the works of Lerman \cite{MR998368,MR1135905}.
Such theorems should be thought of as Hamiltonian versions of the 
homoclinic bifurcations studied by Shil\'{n}ikov \cite{MR0259275,0025-5734-10-1-A07,MR0210987}.
Taken together the results cited so far paint a vivid picture of the rich dynamics 
near a transverse homoclinic connection in a two degree of freedom Hamiltonian system.

The present study concerns asymptotic orbits in the planar equilateral restricted four 
body problem, henceforth referred to as the circular restricted four body problem
(CRFBP).  The problem has a rich literature dating at least back to the work 
of Pedersen \cite{pedersen1,pedersen2}.
Detailed numerical studies of the equilibrium set, as well as
the planar and spatial Hill's regions are found in Sim\'{o} \cite{MR510556}, in 
 Baltagiannis and Papadakis \cite{MR2845212}, and in
\'{A}lvarez-Ram\'{i}erz and Vidal \cite{MR2596303}.
Mathematically rigorous theorems about the equilibrium set and 
its bifurcations are  proven by Leandro and Barros in 
\cite{MR2232439,MR2784870,MR3176322} (with computer assistance).
They show that for any value of the masses there are either 8, 9, 
or 10 equilibrium solutions with 6 outside the equilateral triangle 
formed by the primary bodies (see Figure \ref{rotatingframe}).

Fundamental families of periodic orbits are considered by 
Papadakis in \cite{MR3571218,MR3500916},
and by Burgos-Garc\'{i}a, Bengochea,  and Delgado in 
\cite{burgosTwoEqualMasses,MR3715396}.
A study by Burgos-Garc\'{i}a, Lessard, and Mireles James proves the existence 
of some spatial periodic orbits for the CRFBP \cite{jpJaimeAndMe}
(again with computer assistance).
An associated Hill's problem is derived and its periodic orbits are studied 
by Burgos-Garc\'{i}a and Gidea in \cite{MR3554377,MR3346723}.

Regularization of collisions are studied by \'{A}lvarez-Ram\'{i}rez, Delgado, and Vidal in 
\cite{MR3239345}.  Chaotic motions were studied numerically by Gidea and Burgos in \cite{MR2013214}, 
and by \'{A}lvarez-Ram\'{i}rez and Barrab\'{e}s in \cite{MR3304062}.
Perturbative proofs of the existence of chaotic motions are
found in the work of She, Cheng and Li
\cite{MR3626383,MR3158025,MR3038224},
and also in the work of 
Alvarez-Ram\'\i rez,  Garc\'i{a},  Palaci\'{a}n,  and Yanguas
\cite{chaosCRFBP}.
Blue sky catastrophes in the CRFBP were previously studied by 
Burgos-Garc\'{i}a and Delgado in \cite{MR3105958}, and by 
Kepley and Mireles James in \cite{shaneAndJay}. This last reference 
develops (computer assisted) methods of proof for 
verifying the hypotheses of the theorems of Hernard and Devaney.

The main goal of the present work is to study
orbits which are homoclinic to a saddle-focus equilibrium solution
in the equilateral restricted four body problem.
We apply the parameterization method of Cabr\'{e}, Fontich, and de la Llave
to compute a chart for the stable or unstable manifold in a neighborhood of the equilibrium
\cite{MR1976079,MR1976080,MR2177465}.
Then, we implement the analytic continuation scheme for local invariant manifolds
developed by Kalies, Kepley, and Mireles James in 
\cite{manifoldPaper1}, where it was applied to some two dimensional 
manifolds in the Lorenz system. 
We adapt this scheme for the CRFBP, and compute atlases for the
local stable/unstable manifolds attached to a saddle-focus equilibrium. 
By an atlas we mean a collection of analytic maps or {\em charts} of the form, $P \colon [-1, 1]^2 \to \mathbb{R}^4$, where the image of $P$ lies in the stable or unstable manifold. The union of these charts is a piecewise approximation for a large portion of the manifold away from the equilibrium. For a more formal definition see any standard text on differential geometry. The charts are computed using high order polynomial approximations with algorithms that
exploit automatic manipulations of formal series.  

After computing the stable/unstable manifold atlases we post-process 
to find approximate intersections.
Once a potential intersection is 
located we refine the approximation using a Newton scheme for a
two point boundary value problem as in the classical work of 
Doedel, Friedman, and Kunin \cite{MR1007358,MR1456497}. In the case of the CRFBP, our algorithm identifies a large collection of connecting orbits which are naturally ordered by connection time. We focus on the maximally symmetric case of equal masses, which we refer to as \textit{the triple Copenhagen problem}. We prove that a rotational symmetry in this case reduces the complexity of the atlas computations by a factor of $3$.

The algorithm for producing the atlases utilizes an adaptive subdivision routine to carefully control errors. 
This results in a large number of charts, on the order of tens of thousands, in only a few minutes of computation time. 
These computations are expensive in terms of memory usage, and it is impractical to recompute the atlases for a large number of parameter values, at least given the resources of the present study; namely laptop/desktop computers running single threads. 
Instead, after computing an ensemble of connecting orbits
for the triple Copenhagen problem, we apply numerical continuation 
to the boundary value problem describing the homoclinics.  
That is, we use the connections 
found for the equal mass case as a jumping off point for 
exploring nearby -- but non-symmetric -- mass parameters. 
Continuation of the connecting orbits is much more efficient than continuing the 
entire invariant manifold atlas.

As is well known, the bifurcation structure of the homoclinic continuation problem 
in the Hamiltonian setting is rich. We do not attempt 
automatic tracking of new branches, nor do we follow folds.
A more systematic study of the branching would make an excellent topic for future 
study, perhaps by combining our invariant manifold atlas data with powerful continuation software such as AUTO \cite{MR1404122}.

We emphasize that our restriction to the equal masses case is due to convenience and is not a technical restriction 
on the method itself.  Our atlas algorithm applies to any choice of parameters or even to other Hamiltonian systems.  Thus, even though we abandon the branch whenever the homoclinic continuation algorithm fails, we always have the ability to dig deeper into the cause of failure by running the full atlas computation from scratch.

We remark that our method is deployed in the full phase space,
and does not require choosing a fixed surface of section 
in which to study intersections of the invariant manifolds.
This is advantageous as many problems do not admit
a single section for which the return map
is topologically conjugate to the true dynamics.
Considering the intersections of the stable/unstable manifolds
in a particular section may not reveal all the connecting orbits.
Moreover, the first intersections to appear in phase space may 
not be the first to appear in a given section.  Indeed, projecting to a 
section can introduce discontinuities which make it impossible to 
precisely formulate notions like ``first intersection''.  
The great virtue of a surface of section (restricted to an 
energy level) is that it leads -- at least in the case of a two degree of freedom 
Hamiltonian -- to a two dimensional representation of the dynamics.
We remark that the methods of the present work generalize to 
systems with three or more degrees of freedom, where considering
surfaces of section is less fruitful.

\begin{figure}[!t]
\centering
\includegraphics[width=3.5in]{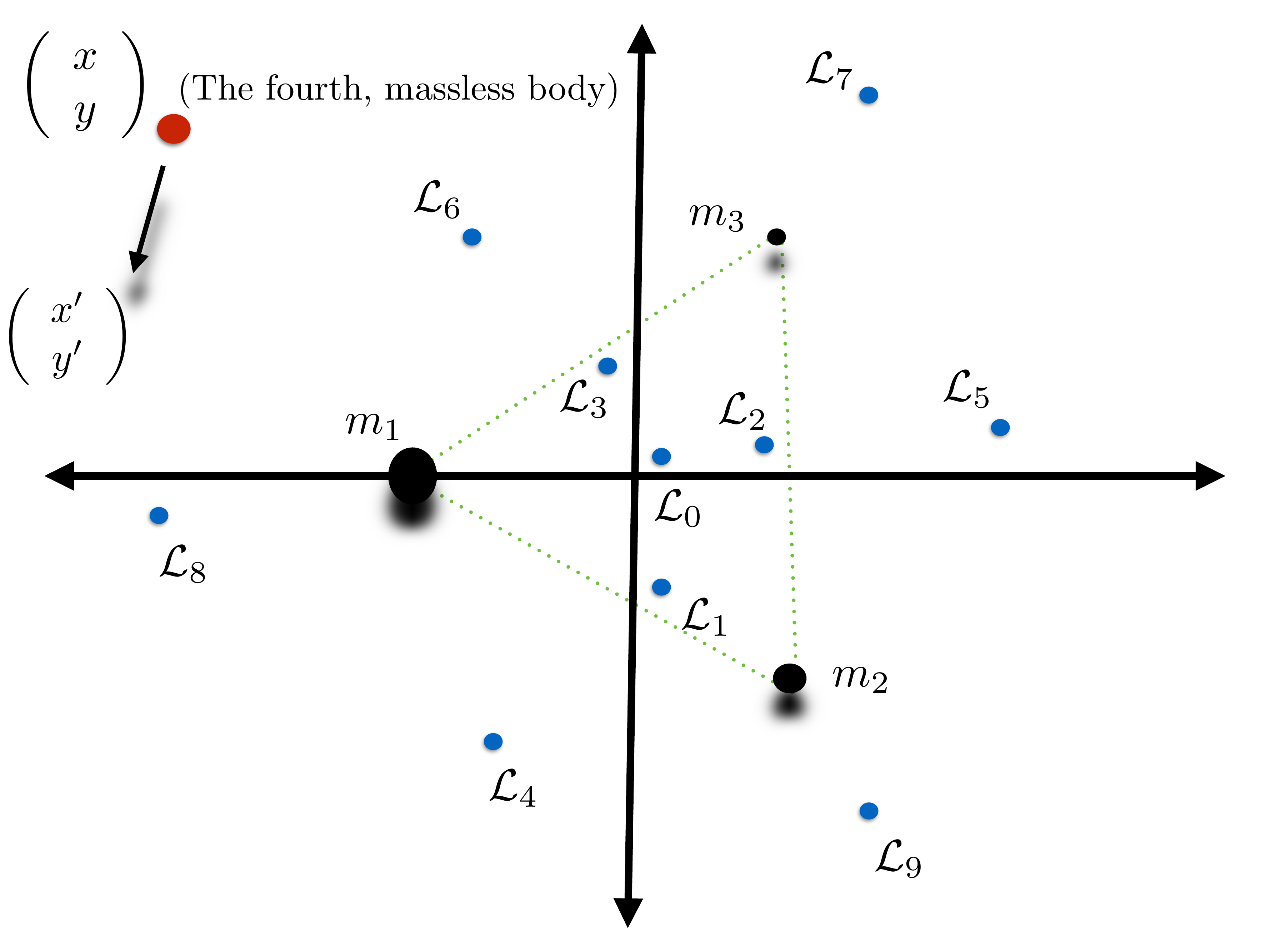}
\caption{Configuration space for the CRFBP: The three 
primary bodies with masses $m_1,m_2,$ and $m_3$ are 
arranged in an equilateral triangle configuration of 
Lagrange, which is  a relative equilibrium solution of the 
three body problem.    After transforming to a
co-rotating frame, we consider the motion of a fourth
massless body.  The equations of motion have 
$8$, $9$, or $10$ equilibrium solutions
(libration points) denoted by $\mathcal{L}_j$
for $0 \leq j \leq 9$.  The number of libration points,
and their stability, varies depending on $m_1$, $m_2$, and $m_3$. In this work we study the points, $\mathcal{L}_{0, 4,5,6}$, which are the only libration points which can have saddle-focus stability.
}\label{rotatingframe}
\end{figure}

\begin{figure}[!t]
\centering
\includegraphics[width=6in]{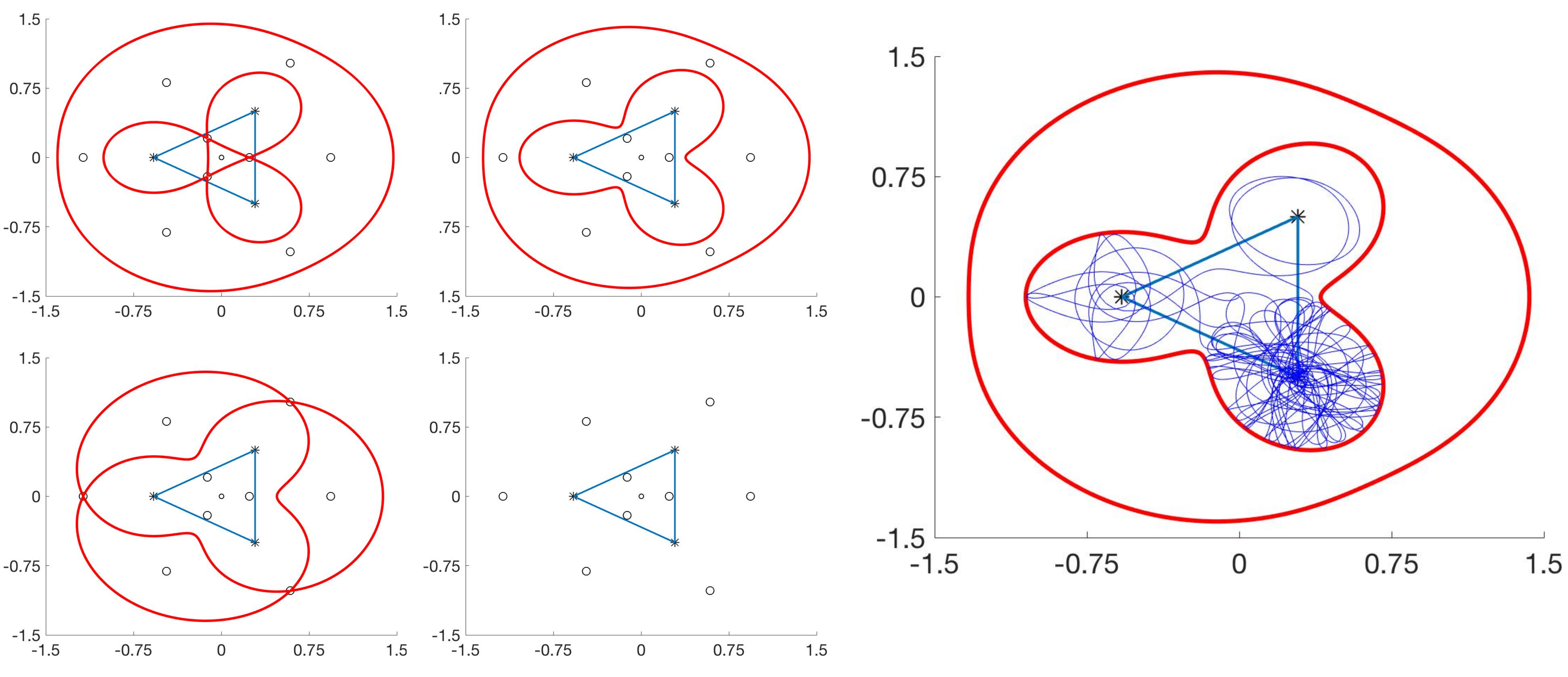}
\caption{Zero velocity curves for the triple Copenhagen problem:
Fixing a value of the Jacobi constant and setting velocity equal in 
Equation  \eqref{eq:CRFB_energy} implicitly defines the
zero velocity curves in the phase space of the CRFBP.  
An orbit which reaches one of these curves arrives with zero velocity, 
and hence turns around immediately.  These define 
natural boundaries which which orbits at a given energy level may not
cross.  Left: the zero velocity curves associated with the energy levels
of $\mathcal{L}_{1,2,3}$ (top left) $\mathcal{L}_0$ (top right), 
$\mathcal{L}_{4,5,6}$ (bottom left), and $\mathcal{L}_{7,8,9}$.
Right: a typical orbit in the $\mathcal{L}_0$ energy 
level confined by the zero velocity curves.  
}\label{fig:zvCurves}
\end{figure}

\begin{figure}[!t]
\centering
\includegraphics[width=5.5in]{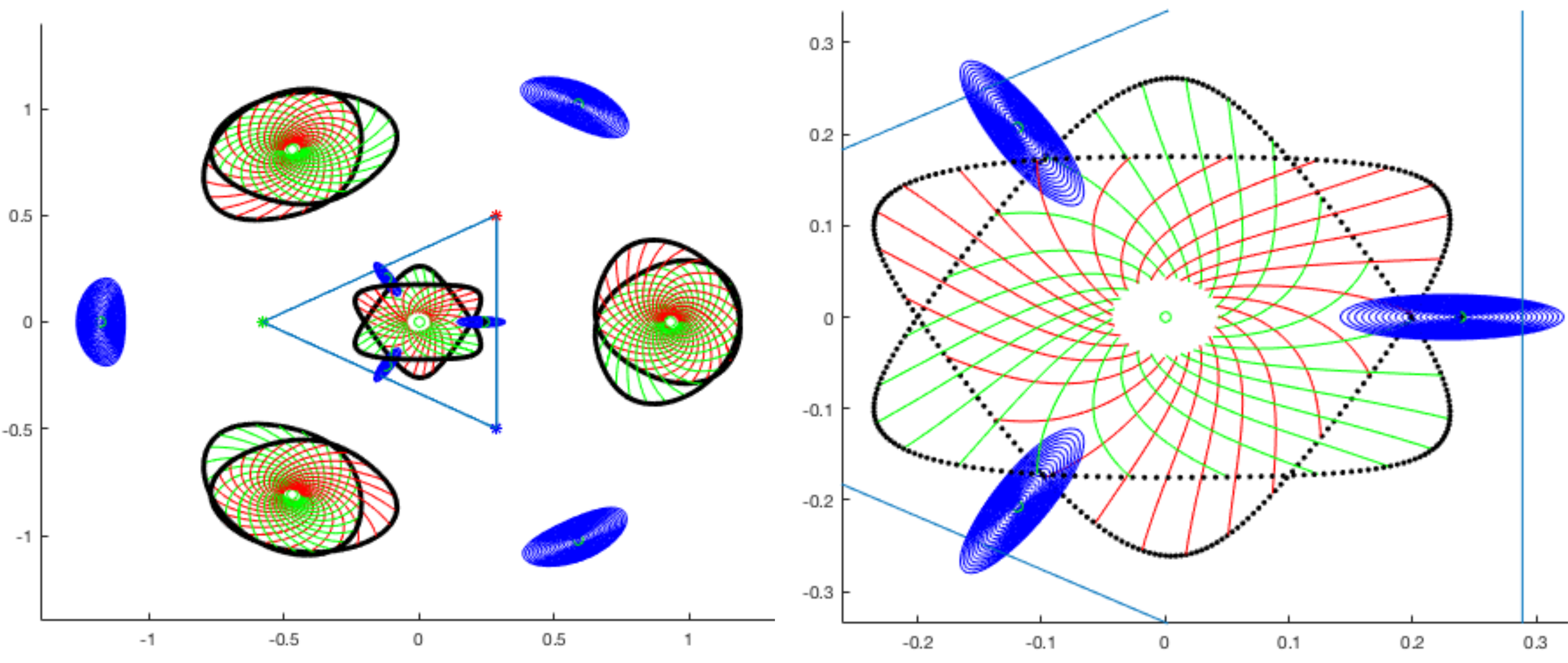}
\caption{Two dimensional local invariant manifolds in the triple Copenhagen problem
(CRFBP with equal masses):
Left: all two dimensional attached invariant manifolds for libration points in the 
equal mass case (one dimensional manifolds not shown).  
In the case of equal masses, the libration points 
$\mathcal{L}_{0,4,5,6}$ have saddle-focus stability.  
Orbits are shown accumulating to the libration points in forward/backward 
time (green/red respectively).  The libration points $\mathcal{L}_{4,5,6,7,8,9}$ 
on the other hand have saddle $\times$
center stability.  In this case each libration point has an attached center manifold
foliated by periodic orbits -- the so called planar Lyapunov orbits.  
We make no systematic study the Lyapunov orbits in the present 
work, and only remark that they appear to organize some of the  
homoclinic orbits in the discussion to follow.    
Right: closeup on the inner libration points and their invariant manifolds.
All references to color refer to the online version.
}\label{fig:localManifolds}
\end{figure}


\section{Saddle-focus equilibrium solutions of the equilateral CRFBP} \label{sec:saddleFocii}
In this section we review well known results about the set of equilibrium solutions in the 
CRFBP, focusing on material which informs the calculations carried out 
in the remainder of the work.  We are especially interested in 
the number and location of saddle-foci, and in how these depend on the mass ratios.   
First we recall the mathematical formulation of the problem, and some of its 
elementary properties.


\subsection{The planar equilateral circular restricted four body problem} \label{sec:RTBP}

Consider three particles with masses $0 < m_3 \leq m_2 \leq m_1 < 1$,
normalized so that 
\[
m_1 + m_2 + m_3 = 1.
\]
These massive particles are referred to as the ``primaries''.
Suppose that the primaries are located at the vertices of a planar equilateral triangle, 
rotating with constant angular velocity.  That is, we assume that the three massive 
bodies are in the triangular configuration of Lagrange.   
We choose a co-rotating coordinate frame which puts the triangle in the $xy$-plane and 
fixes the center of mass at the origin.  We orient the triangle so that 
the first primary is on the negative $x$-axis, the second 
body is in the lower right quadrant, and the smallest body is in the upper right quadrant.
Once in co-rotating coordinates, we 
are interested in the dynamics of a fourth, massless particle with coordinates $(x,y)$, moving in the gravitational field of the primaries.   The situation is illustrated in Figure \ref{rotatingframe}.

We write $(x_1, y_1)$, $(x_2, y_2)$ and $(x_3, y_3)$
to denote the 
locations of the primary masses.  
Let
\[
K = m_2(m_3 - m_2) + m_1(m_2 + 2 m_3).
\]
Taking into account the normalizations discussed above, 
the precise positions of the primary bodies are given by 
the formulas    
\[
x_1 =   \frac{-|K| \sqrt{m_2^2 + m_2 m_3 + m_3^2}}{K},  
\quad \quad \quad \quad 
y_1 =   0,  
\]
\[
x_2 =  \frac{|K|\left[(m_2 - m_3) m_3 + m_1 (2 m_2 + m_3)  \right]}{
2 K \sqrt{m_2^2 + m_2 m_3 + m_3^2} }  
\quad \quad 
y_2  =  \frac{-\sqrt{3} m_3}{2 m_2^{3/2}} \sqrt{\frac{m_2^3}{m_2^2 + m_2 m_3 + m_3^2}}
\]
\[
x_3 =  \frac{|K|}{2 \sqrt{m_2^2 + m_2 m_3 + m_3^2}},
\quad \quad  \quad 
y_3 =  \frac{\sqrt{3}}{2 \sqrt{m_2}} \sqrt{\frac{m_2^3}{m_2^2 + m_2 m_3 + m_3^2}}.
\]
Define the potential function 
\begin{equation} \label{eq:CRFBP_potential}
\Omega(x,y) :=
\frac{1}{2} (x^2 + y^2) + \frac{m_1}{r_1(x,y)} + \frac{m_2}{r_2(x,y)} + \frac{m_3}{r_3(x,y)}, 
\end{equation}
where 
\begin{equation} \label{eq:def_r}
r_j(x,y) := \sqrt{(x-x_j)^2 + (y-y_j)^2},  \quad \quad \quad j = 1,2,3,
\end{equation}
and let $\mathbf{x} = (x, \dot x, y, \dot y) \in \mathbb{R}^4$ denote the
state of the system.   
The equations of motion in the rotating frame are 
\[
\mathbf{x}' = f(\mathbf{x}),
\]
where
\begin{equation}\label{eq:SCRFBP}
 f(x, \dot x, y, \dot y) := 
\left(
\begin{array}{c}
\dot x \\
2 \dot y + \Omega_x(x, y) \\
\dot y \\
-2 \dot x + \Omega_y(x, y) \\
\end{array}
\right).
\end{equation}
The system conserves the quantity   
\begin{eqnarray} \label{eq:CRFB_energy}
E(x, \dot x, y, \dot y) &= 
-\left( {\dot x}^2 + {\dot y}^2 \right) + 2\Omega(x,y),
 \label{eq:energy}
\end{eqnarray}
which is called the \textit{Jacobi integral}.   
Note that $E$ is smooth  -- in fact real analytic -- away from the 
primaries.  The zero velocity curves are defined by fixing a 
value of the energy and setting $\dot x, \dot y$ to zero.
These curves are useful for understanding the structure of the 
phase space and are illustrated in Figure \ref{fig:zvCurves}.

As mentioned in the introduction,
the CRFBP has exactly $8$, $9$ or $10$ equilibrium solutions, depending 
on the values of the mass parameters $m_1, m_2,$ and $m_3$.   The
equilibria are referred to as \textit{libration points} in the dynamical astronomy
literature, and we denote them
by $\mathcal{L}_j$ for $0 \leq j \leq 9$.  A typical 
configuration of these libration points is illustrated in 
 Figure \ref{rotatingframe}, which also illustrates out naming 
 convention.  In the present work we are interested in the linear
 stability of the libration points. We are especially interested in 
determining the mass ratios where $\mathcal{L}_j$ 
 with $j = 0, 4, 5, 6$ are saddle-focus -- as opposed to 
 real saddle or center $\times$ center -- equilibria.
 This question is considered from a numerical point of view
in Section \ref{sec:whereAreTheSaddleFoci?}.
  
We note that for all values of the masses, $\mathcal{L}_j$
with $j = 1, 2, 3, 7, 8, 9$ have either saddle $\times$ center,
or center $\times$ center stability depending on the values of 
the masses
The local two dimensional invariant manifolds attached to all ten libration points 
are illustrated in Figure \ref{fig:localManifolds}, for the case 
of equal masses.

\begin{figure}[!t]
\centering
\includegraphics[width=4.5in,height=4.5in]{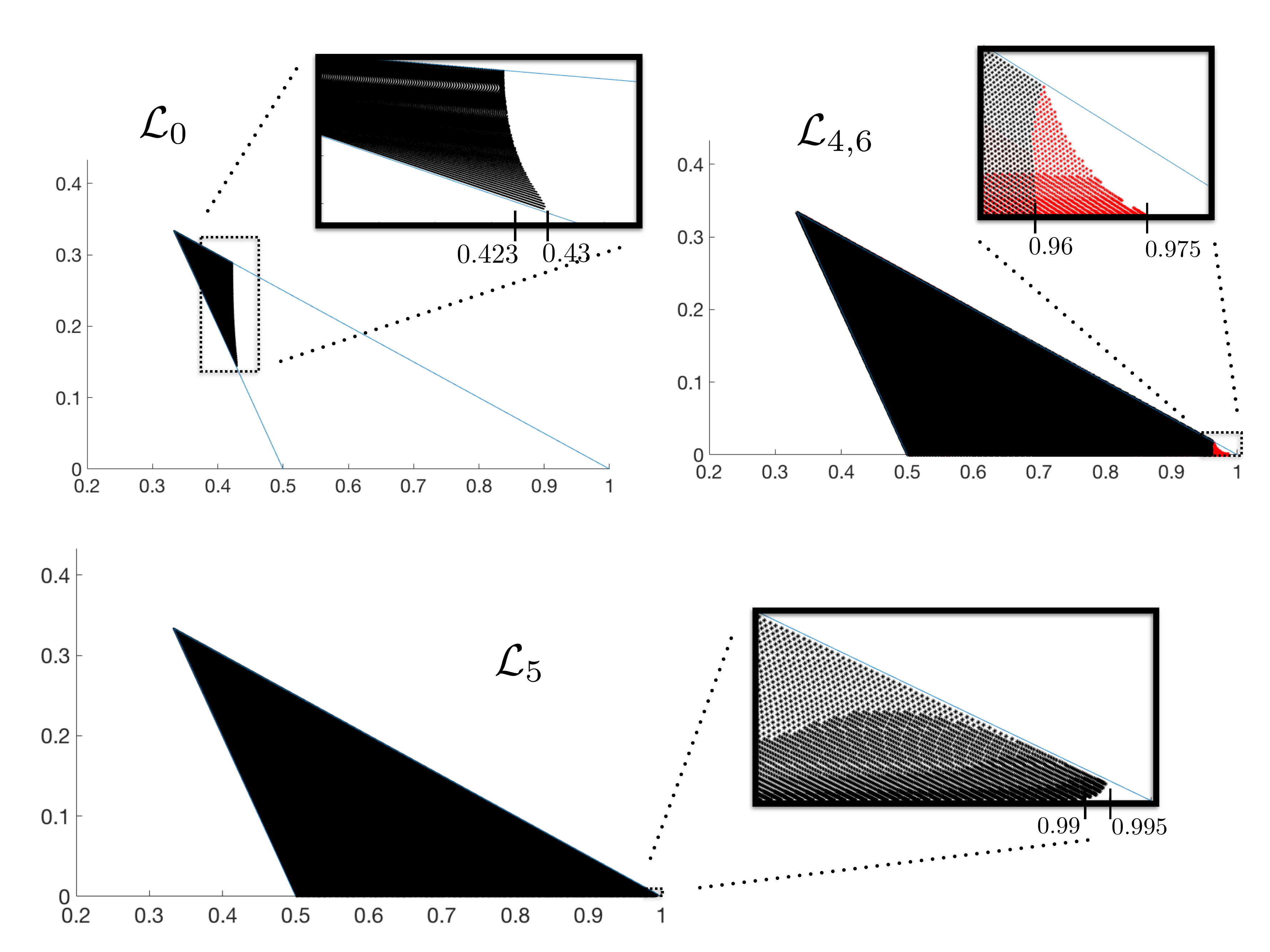}
\caption{Mass values and saddle-focus stability: 
results of a numerical search of the parameter space. Values of $m_1$ are on the 
horizontal axis and values of $m_3$ are on the vertical axis.  These determine 
the remaining mass parameter through the relation $m_2 = 1 - m_1 - m_3$.
In each frame a parameter pair is marked with a black or red dot
 if the libration point $\mathcal{L}_{0,4,5,6}$
has saddle-focus stability.  The top left figure reports the results for $\mathcal{L}_0$, 
the top right for $\mathcal{L}_{4,6}$ and the bottom frame is $\mathcal{L}_5$.
In each case the inlay zooms in on the 
Routh-Gascheau bifurcation curve.
Note that these bifurcation curves are nonlinear, and that 
in the top right results for $\mathcal{L}_4$  are black and results for 
$\mathcal{L}_6$ are red. We remark that the changes in the dot pattern in the bottom right inlay
is due to the use of an adaptive step size in our continuation algorithm.}\label{fig:sadFocParms}
\end{figure}

\subsection{Saddle foci in parameter space} \label{sec:whereAreTheSaddleFoci?}
The CRFBP admits as many as four and 
as few as zero saddle-focus equilibrium points, depending on the 
mass ratios.  We now consider briefly what happens in between these extremes
as the masses are varied.   
The problem is normalized so that $m_1 + m_2 + m_3 = 1$, 
with $m_3 \leq m_2 \leq m_1$, so we have that 
$m_1 \in [1/3, 1]$, $m_2 \in [0, 1/2]$ and $m_3 \in [0, 1/3]$.
Considering the 2-simplex in $\mathbb{R}^3$ satisfying these
constraints, we see that when 
$m_1 \in [1/3, 1/2]$ we have 
\[
m_3 \in \left[-2m_1 + 1, \frac{-1}{2}m_1 +\frac{1}{2} \right],
\]
while for $m_1 \in [1/2, 1]$ we have 
\[
m_3 \in \left[0, \frac{-1}{2}m_1 +\frac{1}{2} \right].
\]
In either case, once we choose $m_1$ and $m_3$, the value of $m_2$ is 
determined by   
\[
m_2 = 1 - m_1 - m_3.
\]

The question is, how does the stability of the libration points depend on the 
mass ratios?  We address the question for each of the points, $\mathcal{L}_{0,4,5,6}$, as follows. Beginning with 
the case of equal masses, $m_1 = m_2 = m_3 = 1/3$, 
we numerically continue each equilibrium to the 
opposite boundary of the parameter simplex at $m_3 = 0$.
Throughout the computation we track the stability of each libration point and label
a parameter point with a black dot whenever the 
stability is of saddle-focus type. The results 
are summarized in Figure \ref{fig:sadFocParms}.
We refer to the curve in the parameter simplex where the stability 
changes as the Routh-Gascheau curve.

Roughly speaking,  we see that when $\nicefrac{1}{3} \leq m_1 \leq 0.42$ the libration point
$\mathcal{L}_0$ is a saddle-focus for all allowable values of $m_2$, $m_3$.  
When $m_1 > 0.43$, the libration point $\mathcal{L}_0$ is no
longer a saddle, no matter the values of $m_2$, $m_3$. 
The points $\mathcal{L}_{4, 6}$ on the other hand 
have saddle-focus stability for most parameter values, and only bifurcate 
after $m_1 > 0.95$ (with $\mathcal{L}_6$ a little more robust than 
$\mathcal{L}_4$ except when $m_2 = m_3$).  The libration point
$\mathcal{L}_5$ is the most robust.  It maintains saddle-focus stability 
until $m_1 \approx 0.99$.  For $m_1 > 0.995$ there are no more saddle 
foci at all.  By reading parameter values off of the frames in Figure 
\ref{fig:sadFocParms} we can arrange that the CRFBP has
 $1, 2, 3$ or $4$ saddle-focus equilibria.  In the sequel we are interested 
 in homoclinic connections for such parameters.


\subsection{Two ways to formulate a connecting orbit: phase space geometry and 
boundary value problems} \label{sec:twoViewsOfConnections}
There are two standard ways to think about connecting 
orbits and -- while they are completely equivalent from a mathematical point of view --
in practice they have different advantages and disadvantages.  
In the following let $f \colon \mathbb{R}^n \to \mathbb{R}^n$ denote
a smooth vector field and let $\mathbf{x}_0 \in \mathbb{R}^n$ be an equilibrium 
solution for $f$.  We write $W^s(\mathbf{x}_0)$ and $W^u(\mathbf{x}_0)$ to denote respectively 
the stable and unstable manifolds attached to $\mathbf{x}_0$.
\begin{itemize}
\item \textbf{Analytic definition:} If $\mathbf{x} \colon \rr \to \mathbb{R}^n$
satisfies
\[
\frac{d}{dt} \mathbf{x}(t) = f(\mathbf{x}(t)),
\]
for all $t \in \mathbb{R}$, and satisfies the asymptotic boundary conditions 
\[
\lim_{t \to \pm \infty} \mathbf{x}(t) = \mathbf{x}_0,
\]
then we say that $\mathbf{x}$ is a homoclinic connecting orbit
for $\mathbf{x}_0$.  
\item \textbf{Geometric definition:} If 
\[
\hat x \in W^s(\mathbf{x}_0) \cap W^u(\mathbf{x}_0),
\]
and  $\mathbf{x} = \mbox{orbit}(\hat x)$ denotes the orbit which passes through $\hat x$, 
then $\mathbf x$ is a homoclinic connecting orbit for $\mathbf{x}_0$.  If the intersection 
of the manifolds is transverse then 
we say that $\mathbf{x}$ is a transverse homoclinic connection.
\end{itemize}

The analytic definition is recast as a finite time boundary value problem 
by projecting the boundary conditions onto local stable/unstable manifolds.
If $P, Q$ are parameterizations of the local unstable and stable manifolds 
respectively, then we look for $T > 0$ and $\mathbf{x} \colon [0, T] \to \mathbb{R}^n$,
so that $\mathbf{x}$ solves the differential equation subject to the boundary 
conditions 
\[
\mathbf{x}(0) \in \mbox{image}(P), 
\qquad \mbox{and} \qquad 
\mathbf{x}(T) \in \mbox{image}(Q).
\] 
In applications one frequently replaces $P$ and $Q$ by their linear approximations.
In Section \ref{sec:invariantManifolds} we review an approach called 
\textit{the parameterization method} for computing high order 
polynomial approximations of the local charts $P, Q$. 

%
%

\begin{remark}[Relative strengths and weaknesses] \label{rem:connectionFormulations}
One great advantage of the analytic formulation is that, since it is equivalent to a two 
point boundary value problem, we can utilize the Newton-Method to find very 
accurate solutions -- often on the order of machine precision.  
The formulation as a boundary value problem also lends itself to numerical continuation
schemes, which are very useful for exploring the parameter space.
The disadvantages are twofold.  First, in this formulation it is necessary to begin the Newton 
iteration with a fairly good approximate solution and  this raises the question: \textit{Where do the approximate
solutions come from?} Second, it is difficult to rule out solutions using the BVP approach.

In the geometric approach there is no need to make a guess.  Instead, one 
moves along the stable and unstable manifolds, and identifies connections by locating intersections in phase space.  At the same time, the geometric approach allows one to rule out connecting orbits by showing that a particular region of phase space does not contain any intersections.  The difficulty with the geometric perspective is that it provides information only as good as 
our knowledge of the embeddings of the stable/unstable manifolds.  Computing embeddings
of invariant manifolds is challenging, and methods tend to decrease in accuracy the farther from the 
equilibrium they are applied.

The important point, from the perspective of the present work, is that 
these two approaches complement one another.  The geometric formulation 
is good for locating and ruling out connections while the analytic formulation is 
good for refining approximations and for continuation with respect to
parameters.  This suggests the approach of the present work: namely that we use
the two formulations in concert, playing the strengths of one against the 
weaknesses of the other as appropriate. 

We remark that in many applications it is convenient to examine the intersections of
the invariant manifolds in an intermediate surface of section. 
This is especially true for two degree of freedom systems as the section intersected with 
the energy level leads to a two dimensional image which is easy to visualize.  
Often an appropriate section is suggested by the geometry of the problem, or by 
the goals of a particular space mission.  We refer the interested reader to the works
\cite{MR1765636,MR2163533,MR2557453} for examples and fuller discussion.
\end{remark}

%
%

\section{Numerical computation of the stable/unstable manifolds} 
\label{sec:invariantManifolds}
The results of Section \ref{sec:whereAreTheSaddleFoci?} show that 
\textit{for most parameter values}, the CRFBP has either three or four 
saddle-focus equilibria -- though for some parameters it may have only two,
or one, or none.
For a given saddle-focus equilibrium with fixed values of the mass parameters, 
we compute the invariant manifolds in two steps.  
First we find a high order expansion of an initial local chart 
containing the equilibrium solution.  Then we use a high 
order Taylor integration scheme to advect 
the boundary of the initial chart one sub-arc at a time. 
The second step is repeated until a certain integration time has been reached, 
or until some error tolerance has been exceeded.  
Along the way it is sometimes necessary to subdivide boundary 
arcs in order to manage the truncation errors.  

Our computation of the initial chart employs the parameterization method, 
which is reviewed in Section \ref{sec:parmMethod}.
Advection of the boundary uses a Taylor integration scheme 
similar to the one developed in \cite{manifoldPaper1},
but adapted to the problem at hand.  
Both procedures exploit differential-algebraic manipulations 
of formal power series, and these manipulations
are delicate due to the presence of the 
minus two thirds power in the nonlinearity of the CRFBP vector field.

One technique for manipulating power series of several complex variables 
involves \textit{automatic differentiation} combined with the radial-gradient.
This procedure is developed in 
 \cite{mamotreto}, and is reviewed in Appendix \ref{sec:formalSeries}.
 Another technique involves appending additional variables and equations to 
 the problem, so that the enlarged field is polynomial and equivalent to 
 the original CRFBP on a certain sub-manifold.  
 This option is discussed at length for the CRFBP in \cite{shaneAndJay} which also includes a more precise definition of what ``equivalent'' means here.
 See also  \cite{fourierAutomaticDiff}, and \cite{MR0135589}.

\subsection{Parameterization method for the local invariant manifold} \label{sec:parmMethod}
We now review the parameterization method adapted to the 
needs of the present work, namely for a stable/unstable manifold 
attached to  a saddle-focus equilibrium in $\mathbb{R}^4$.  
Much more general treatment of the parameterization method is found in 
\cite{MR1976079,MR1976080,MR2177465}.
See also the book on this topic \cite{mamotreto}.

Let $\mathbf{x}_0 \in \mathbb{R}^4$ denote a saddle-focus equilibrium point. Specifically, we suppose $f(\mathbf{x}_0) = 0$, 
\[
\lambda_{1,2} = -\alpha \pm i \beta, 
\]
with $\alpha, \beta > 0$ denotes the stable eigenvalues for $Df(\mathbf{x}_0)$, 
and $\xi_{1,2} \in \mathbb{C}^4$ denotes a choice of associated
complex conjugate eigenvectors.   

Since the eigenvalues are complex, it is convenient to 
look for a complex parameterization of a local stable manifold.  
Let 
\[
D^2 = \left\{ (z_1, z_2) \in \mathbb{C}^2 \, : \, |z_j| < 1, \ j = 1,2 \right\}
\]
denote the unit complex polydisc. 
We look for a parameterization $P \colon D^2 \to \mathbb{C}^4$
satisfying the infinitesimal 
conjugacy given by 
\begin{equation} \label{eq:invEq}
D P(\mathbf{z}) \Lambda \mathbf{z} = f(P(\mathbf{z})), 
\end{equation}
where $\mathbf{z} = (z_1, z_2)^{\mbox{T}}$, and 
\[
\Lambda = \left(
\begin{array}{cc}
\lambda_1 & 0 \\
0 & \lambda_2
\end{array}
\right).
\]
Equation 
\eqref{eq:invEq} is 
subject to the first order constraints 
\begin{equation} \label{eq:constraints}
P(0, 0) = \mathbf{x}_0, 
\quad \quad \quad \mbox{and} \quad \quad \quad 
\frac{\partial}{\partial z_{1,2}} P(0, 0) = \xi_{1,2}.
\end{equation}
Note that 
\[
DP(\mathbf{z}) \Lambda \mathbf{z} = 
\lambda_1 z_1 \frac{\partial}{\partial z_1} P(z_1, z_2) +
\lambda_2 z_2 \frac{\partial}{\partial z_2} P(z_1, z_2),
\]
is the push forward of the linear vector field by $P$.
The geometric meaning of Equation \eqref{eq:invEq} is illustrated in 
Figure \ref{fig:parmSchematic}.


Let 
$\Phi$ denote the flow generated by $f$.
Any $P$ satisfying Equation \eqref{eq:invEq}
on $D^2$ also satisfies the flow conjugacy
\begin{equation} \label{eq:flowConj}
\Phi(P(z_1, z_2), t) = P(e^{\lambda_1 t} z_1, e^{\lambda_2 t} z_2), 
\quad \quad \quad (z_1, z_2) \in D^2.
\end{equation}
In particular, if $P$ satisfies both Equation \eqref{eq:invEq}
and the constraints of Equation \eqref{eq:constraints}, then 
for any $(z_1, z_2) \in D^2$ it follows that 
\begin{align*}
\lim_{t \to \infty} \Phi(P(z_1, z_2), t)  &=  \lim_{t \to \infty} P(e^{\lambda_1t} z_1, e^{\lambda_2 t} z_2) \\
&= P(0, 0) \\
&= \mathbf{x}_0,
\end{align*}
so that $P(D^2) \subset W^s(\mathbf{x}_0)$.  
Combining this with the fact that the image of $P$ contains 
$\mathbf{x}_0$ and is tangent to the stable eigenspace at $\mathbf{x}_0$ we 
see that $P$ parameterizes a local stable manifold for $\mathbf{x}_0$.  
Moreover we recover the dynamics on the manifold through the conjugacy.  

When the vector field $f$ is analytic near $\mathbf{x}_0$, then $W^{s}(\mathbf{x}_0)$ is an analytic 
manifold, and it makes sense to look for an analytic chart of the form
\[
P(z_1, z_2) = \sum_{m=0}^\infty \sum_{n=0}^\infty p_{m,n} z_1^m z_2^n,
\]
with $p_{m,n} \in \mathbb{C}^4$ for all $m,n \in \mathbb{N}$.
Since we are interested in the real image of the chart
we look for a solution of Equation 
\eqref{eq:invEq}  with 
\[
P(z, \bar z) \in \mathbb{R}^4,
\]
for all $|z| < 1$.  This is achieved whenever the power series 
coefficients of the solution satisfy
\begin{equation} \label{eq:complexConj}
p_{n,m} = \overline{p_{m,n}},
\end{equation}
for all $(m,n) \in \mathbb{N}^2$.  
The real parameterization $\tilde{P} \colon B \to \mathbb{R}^4$ is 
recovered using complex conjugate variables 
\[
\tilde{P}(\sigma_1, \sigma_2) = P(\sigma_1 + i \sigma_2, \sigma_1 - i \sigma_2).
\]
Elementary proofs of the facts discussed in this 
section are found for example in \cite{shaneAndJay}.

\begin{figure}[!t]
\centering
\includegraphics[width=4in]{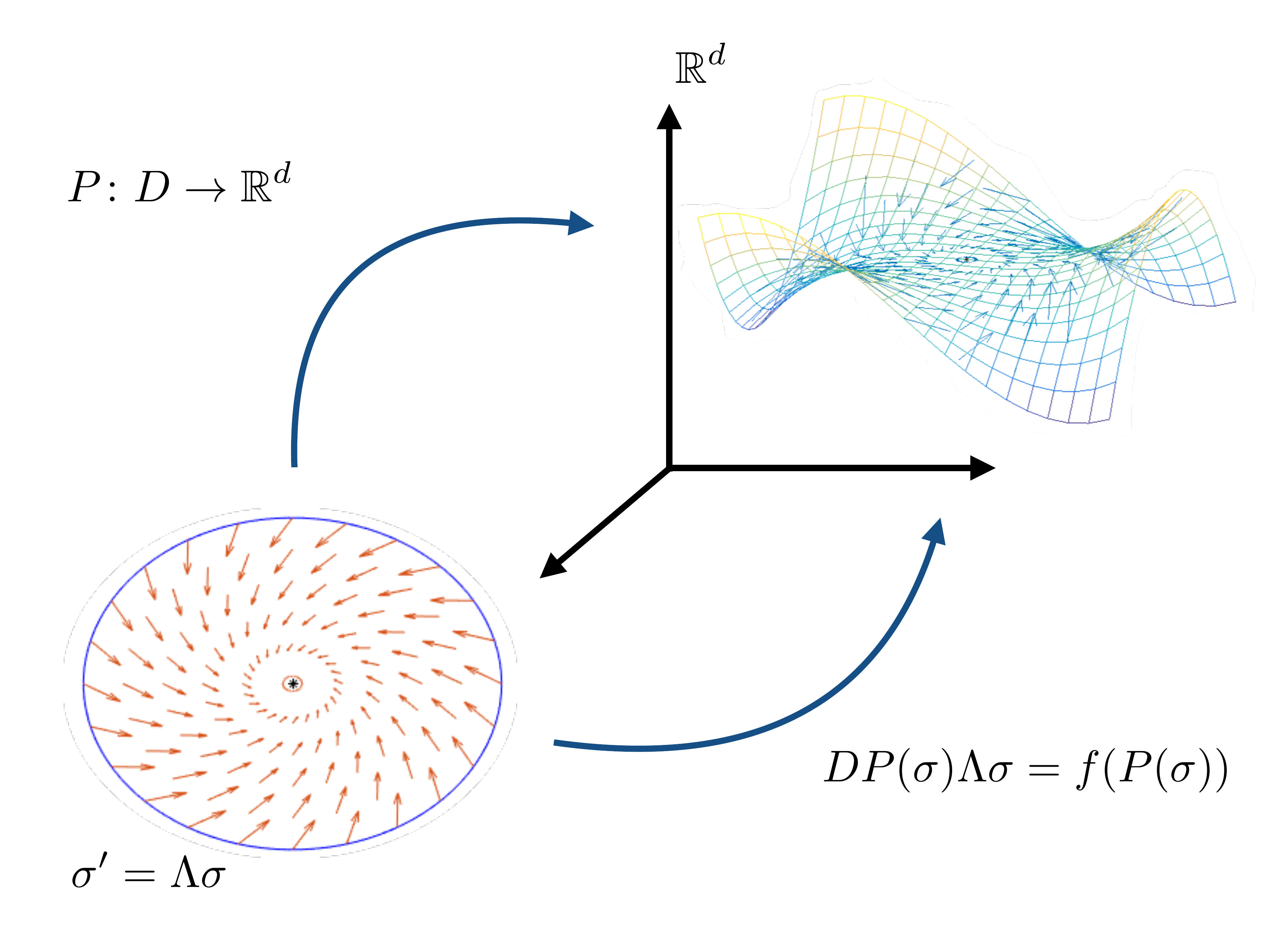}
\caption{\textbf{Geometric interpretation of the parameterization method for differential equations:} 
Equation \eqref{eq:invEq} requires that the push forward of 
the vector field $\Lambda$ by $P$ matches the vector field $f$ 
on the image of $P$. A function satisfying this equation is a 
parameterization of a local stable manifold.
}\label{fig:parmSchematic}
\end{figure}

\subsection{Power series solution of Equation \eqref{eq:invEq}} \label{sec:solveInvEq}
We describe three methods for computing the power series 
coefficients of an analytic solution of the invariance equation
given in Section \ref{sec:parmMethod}.  
Combining these methods leads to very efficient numerical methods.  

\subsubsection{Solution by power matching} \label{sec:powerMatch} Plugging the unknown power series 
expansion for $P$ into Equation \eqref{eq:invEq} leads to 
\[
\sum_{m=0}^\infty \sum_{n=0}^\infty (m \lambda_1 + n \lambda_2) p_{m,n} z_1^m z_2^n
= \sum_{m=0}^\infty \sum_{n=0}^\infty [f \circ P]_{m,n}  z_1^m z_2^n.
\]
It is shown in \cite{MR1976079} (see also the discussion in \cite{mamotreto})
that when we match like powers and isolate $p_{m,n}$ we
are led to an expression of the form 
\begin{align*}
(m \lambda_1 + n \lambda_2) p_{m,n} &= [f \circ P]_{m,n} \\
&= Df(p_{0,0}) p_{m,n} + R(P)_{m,n},
\end{align*}
where  $R(P)_{m,n}$ depends in a nonlinear way on   
coefficients $p_{j,k}$ with $ 0  \leq j+k < m+n$.
Isolating the variable $p_{m,n}$ on the left 
leads to the \textit{homological equations} 
\begin{equation}\label{eq:homEq}
\left[ Df(\mathbf{x}_0) - (m \lambda_1 + n \lambda_2) \mbox{Id} \right] p_{m,n} = -R(P)_{m,n}.
\end{equation}

\begin{remark}[The formal solution is well defined] \label{rem:wellDefined}. 
Observe that Equation \eqref{eq:homEq} is linear in $p_{m,n}$ and has a 
unique solution as long as $m \lambda_1 + n \lambda_2$ is not an 
eigenvalue of $Df(\mathbf{x}_0)$.  But $\lambda_2 = \overline{\lambda_1}$, and since any 
remaining eigenvalues are assumed to be unstable, 
we have that $m \lambda_1 + n \lambda_2$ is never an 
eigenvalue of $Df(\mathbf{x}_0)$.  Hence
 the matrix on the left hand side of the homological equation \eqref{eq:homEq}
is invertible for all $m + n \geq 2$.  

Given any first order data
as in the constraint Equations \eqref{eq:constraints},
the homological equations are uniquely solvable to all orders and
the corresponding formal series solution of Equation \eqref{eq:invEq} is well defined.
Since each Taylor coefficient is uniquely 
determined by the homological equations \eqref{eq:homEq}, 
it follows that   
the formal series solution is unique up to the choice of the scalings of the eigenvectors
in Equation \eqref{eq:constraints}.  
Solving the homological equations recursively 
to order $N \geq 2$ provides a polynomial chart $P^N$ which 
approximately parameterizes the local stable manifold.
\end{remark}    
\begin{remark}[Reality of the parameterization]  \label{rem:realityOfParm}
Taking complex conjugates in  
the homological equations \eqref{eq:homEq} shows that
the coefficients $p_{m,n}$ have the  
symmetry of Equation \eqref{eq:complexConj}.
\end{remark}

\subsubsection{A Newton scheme} \label{sec:ParmNewton} 
A quadratic convergence scheme for Equation \eqref{eq:invEq}
is obtained as follows. Define the nonlinear operator 
\[
\Psi[P](\sigma) = DP(\sigma) \Lambda \sigma - f(P(\sigma)),
\]
where $f$ is the CRFBP vector field, and note that a zero of $\Psi$
is a solution of Equation \eqref{eq:invEq}. Moreover we note
that, at least formally, the Fr\'{e}chet derivative is given by  
\[
D\Psi[P] H (\sigma) = DH(\sigma) \Lambda \sigma - Df(P(\sigma)) H(\sigma). 
\]
In fact this is the correct Fr\'{e}chet derivative of $\Psi$ when for example
we consider $\Psi$ defined on a Banach space of analytic functions, see
\cite{MR1976079,MR2177465,MR2966749}.

Choose $P_0$ an approximate zero of $\Psi$, and define the sequence
\[
P_{n+1} = P_n + \Delta_n,
\]
where $\Delta_n$ is the formal series solution of the linear equation 
\begin{equation} \label{eq:linear}
D\Psi[P] \Delta = - \Psi[P].
\end{equation}
If $P_0$ is a good enough approximate solution of Equation \eqref{eq:invEq}
we expect $P_n$ to converge quadratically to a zero of $\Psi$.
The linear operator $D \Psi[P]$ non-constant coefficient, 
and Equation \eqref{eq:linear} may be solved recursively via the following power matching 
scheme.
Define
\[
\Delta(\sigma_1, \sigma_2) = \sum_{m=0}^\infty \sum_{n=0}^\infty \Delta_{m,n} \sigma_1^m \sigma_2^n, 
\]
\[
Df(P(\sigma)) = \sum_{m=0}^\infty \sum_{n=0}^\infty A_{m,n} \sigma_1^m \sigma_2^n,
\]
and 
\[
-\Psi(P(\sigma)) =  \sum_{m=0}^\infty \sum_{n=0}^\infty q_{m,n} \sigma_1^m \sigma_2^n.
\]
Here $\Delta_{m,n}, q_{m,n} \in \mathbb{C}^4$, and $A_{mn}$ are $4 \times 4$ complex valued 
matrices for all $(m,n) \in \mathbb{N}^2$.
Plugging these series expansions into Equation \eqref{eq:linear} leads to 
\[
\sum_{m+n \geq 2} \left( (m \lambda_1 + n \lambda_2) \Delta_{m,n} -
 \sum_{j=0}^m \sum_{k=0}^n A_{m-j, n-k} \Delta_{j,k} \right) \sigma_1^m \sigma_2^n = 
 \sum_{m+n \geq 2} q_{m,n} \sigma_1^m \sigma_2^n,
\]
or, upon matching like powers, 
\[
(m \lambda_1 + n \lambda_2) \Delta_{m,n} -
 \sum_{j=0}^m \sum_{k=0}^n A_{m-j, n-k} \Delta_{j,k} = q_{m,n},
\]
for all $m+n \geq 2$.
We note that the sum contains one term of order $\Delta_{mn}$, 
appearing when $j = m$ and $k = n$.  That is
\[
 \sum_{j=0}^m \sum_{k=0}^n A_{m-j, n-k} \Delta_{j,k} = 
 A_{00} \Delta_{mn} + \mbox{``lower order terms of } \Delta \mbox{''}.
\]
Let 
\[
\tilde{\delta}_{j,k}^{m,n}  = \begin{cases}
1 & j < m \mbox{ or } k < n \\
0& j = m \mbox{ and } k = n
\end{cases}.
\]
Then we use $\tilde{\delta}_{j,k}^{m,n}$ to extract terms of order $(m,n)$ from the 
sum and write the equation for $\Delta_{mn}$ as 
\[
(m \lambda_1 + n \lambda_2) \Delta_{m,n} -
A_{0,0} \Delta_{m,n} -  \sum_{j=0}^m \sum_{k=0}^n \tilde \delta_{j,k}^{m,n}
A_{m-j, n-k} \Delta_{j,k} = q_{m,n}.
\]
Recall that $A_{0,0} = Dg(0) = Df(\mathbf{x}_0)$, so that rearranging terms leads to
the linear equations 
\begin{equation}\label{eq:homEq2}
\left(Df(\mathbf{x}_0) - (m \lambda_1 + n \lambda_2)\mbox{Id} \right) \Delta_{m,n} = 
- q_{m,n} - \sum_{j=0}^m \sum_{k=0}^n \tilde \delta_{j,k}^{m,n}
A_{m-j, n-k} \Delta_{j,k},
\end{equation}
for $m + n \geq 2$. Since the right hand side of  
Equation \eqref{eq:homEq2} is exactly the right hand side appearing in the homological 
equations \eqref{eq:homEq} of Section \ref{sec:powerMatch},
arguing as in Remarks \ref{rem:wellDefined} and \ref{rem:realityOfParm}
shows that the Equations of  \eqref{eq:homEq2} are uniquely solvable 
for all $m + n \geq 2$ just as before, 
and that the resulting power series coefficients have the desired symmetry.
Then this Newton scheme 
is well defined on the space of formal power series.

\subsubsection{A pseudo-Newton scheme} \label{sec:pseudoNewton} 
While the Newton scheme of the previous section converges rapidly (in the 
sense of the number of necessary iterations), 
solving the required non-constant coefficient linear equations are expensive.
In this case the overall computation may be slow just because of the cost
of computing the individual corrections.  
The iterations can be speed up as follows. 

First, we note that 
\[
D\Psi[P] \Delta (\sigma) = D \Delta (\sigma) \Lambda \sigma 
- Df(\mathbf{x}_0) \Delta(\sigma) + \mbox{``higher order terms''},
\]
and we define a new iterative scheme 
\[
P_{k+1}(\sigma) = P_k(\sigma) + \tilde{\Delta}_k(\sigma),
\]
where $ \tilde{\Delta}_k$ is a solution of the \textit{constant coefficient linear
equation} 
\[
D  \tilde{\Delta}_k(\sigma) \Lambda \sigma -  Df(\mathbf{x}_0)  \tilde{\Delta}_k(\sigma) 
= - \Psi(H_k).  
\]
On the level of power series, this equation becomes
\[
\sum_{m=0}^\infty \sum_{n=0}^\infty \left[(m \lambda_1 + n \lambda_2) \mbox{Id} - Df(\mathbf{x}_0)\right]
 \tilde{\Delta}_{m,n} \sigma_1^m \sigma_2^n = \sum_{m=0}^\infty \sum_{n=0}^\infty q_{m,n} \sigma_1^m \sigma_2^n,
\]
and matching like powers yields the linear equations 
\[
\left[Df(\mathbf{x}_0) - (m \lambda_1 + n \lambda_2) \mbox{Id} \right]
 \tilde{\Delta}_{m,n} =  -q_{m,n}.
\]
These homological equations uniquely determines the coefficients $\tilde{\Delta}_{m,n}$,
and have the virtue of being ``diagonal'' in Taylor coefficient space.  In practice we find that the pseudo-Newton scheme requires more iterates than the Newton method to converge. However, a single iteration step is much faster and for 
reasonable values of $N$ the pseudo-Newton method is faster overall. We discuss this further below.

\begin{remark}
In practice the linear approximation of $P$ by the eigenvectors provides 
a good initial guess for the Newton and pseudo-Newton schemes, especially 
when computations are started ``from scratch''.  However, within the context of 
calculations based on parameter continuation, 
we will take $P_0$ as the high order parameterization from the previous mass values.

Indeed, it seems that the best 
results are obtained by a ``hybrid'' approach.  That is, we compute an initial guess 
$P_0$ by recursively solving Equation \eqref{eq:homEq}
to some fixed order, $N_0$. Then, we refine this approximation via the Newton or pseudo-Newton scheme to obtain a polynomial approximation to order, $N > N_0$.  The runtime performance for this hybrid approach is recorded in Table \ref{tab:2}.
\end{remark}

\begin{remark}[Quantifying the errors] \label{rem:errors}
Suppose that the polynomial
\[
P^N(z_1, z_2) = \sum_{0 \leq m+n \leq N} p_{m,n} z_1^m z_2^n, 
\]
is an approximation solution of Equation \eqref{eq:invEq}.  One way to measure the 
quality of the approximation is to measure the {\em defect} associated with $P^N$ defined by the quantity 
\[
\mbox{defect}(P^N) = 
\sup_{\mathbf{z} \in D^2} \norm{DP^N(\mathbf{z}) \Lambda \mathbf{z} - f(P(\mathbf{z}))}_{\mathbb{C}^4}.
\]
This quantity could be approximated by evaluating on a mesh of points in $D$.
On the other hand, we can use the fact that for power series 
on the unit disk we have the bound
\[
\sup_{\mathbf{z} \in D^2} \| g(\mathbf{z}) \|_{\mathbb{C}^4} \leq 
\sum_{m+n = 0}^\infty \| a_{m,n} \|_{\mathbb{C}^4},
\]
where the infinite sum can be approximated by a finite sum.
Then another useful a-posteriori indicator is obtained by choosing an 
$N' > N$ and computing the quantity 
\[
\varepsilon_{\mbox{a-posteriori}} = 
\sum_{0 \leq m+n \leq N'} 
\left\| (m \lambda_1 + n \lambda_2) p^N_{m,n} - [f \circ P^N]_{m,n} \right\|_{\mathbb{C}^4},
\]
where $p_{m,n}^N$ are the power series coefficients of $P^N$, 
and $[f \circ P^N]_{m,n}$ are the coefficients of $f(P^N(\mathbf{z}))$.
Of course this bounds also the real image of $P^N$.

If $f$ is a polynomial of order $K$ then we take $N' = K N$.
If $f$ is not a polynomial, then the power series for 
$f\circ P^N$ has infinitely many terms even though $P^N$
is polynomial.  Then we choose $N' > N$ somewhat arbitrarily.
Note that $p_{m,n}^N$ are zero when $m + n > N$, so that eventually the sum
involves only the coefficients of the composition.   

Yet another useful error indicator is obtained by considering the dynamical 
conjugacy of Equation \eqref{eq:flowConj}.  Since the true solution 
satisfies the dynamical conjugacy exactly we consider also the quantity
defined by 
\[
\mbox{conjugacyDefect}(P^N) = 
\sup_{\mathbf{z} \in D^2} \sup_{t > 0} \norm{
\Phi(P(z_1, z_2), t) - P(e^{\lambda_1 t} z_1, e^{\lambda_2 t} z_2)}_{\mathbb{C}^4}.
\]
To approximate this quantity we fix $\tau > 0$ and 
let $\Phi_{\mbox{num}}$ denote a numerical 
integrator and $z_{k}$, $1 \leq k \leq K$
be a mesh of the complex circle so $|z_k| = 1$.
Define the indicator
\[
\varepsilon_{\mbox{conjugacy}} = \max_{1 \leq k \leq K}
\left\| \Phi_{\mbox{num}}(P^N(z_k, \overline{z_k}), \tau) - 
P^N(e^{\lambda_1 \tau} z_k, e^{\lambda_2 \tau} \overline{z_k})
\right\|_{\mathbb{C}^4}.
\]
Error bounds for a number of example computations are recorded in 
Table \ref{tab:1}.
\end{remark}

\begin{table}[t!]
\caption{\textbf{Taylor order, scaling, and error bounds for the 
parameterization method:} table reports the numerical defect and numerical 
conjugacy error associated with the local stable/unstable manifold parameterization
for a number of example computations, 
as functions of the polynomial order and eigenvector scalings.
The first column records the polynomial order $N$ of the approximation $P^N$.
The second column is the magnitude $\|\xi\|$ of the eigenvector (the scaling of the 
parameterization).  These are the inputs which must be specified by the 
user in any computation using the parameterization method.  
The third column reports the corresponding 
bound on magnitude of the highest order Taylor coefficients.   
The fourth and fifth columns record the numerical 
defect and conjugacy error respectively.  The sixth column 
reports the worst of these two quantities measured in multiples of machine epsilon.}
\label{tab:1}       
\begin{center}
\begin{tabular}{llllll}
\hline\noalign{\smallskip}
$N$ & $\|\xi\|$ &   $\max_{N} \|p_{m,n}\|$ & defect &  conj error  & $\#$ machine eps \\
\noalign{\smallskip}\hline\noalign{\smallskip}
 $1$ & $10^{-8}$ & $10^{-8}$ & $2.3 \times 10^{-15}$& $3.9 \times 10^{-15}$ & 18\\
 $1$ & $10^{-6}$  & $10^{-6}$ & $1.4 \times 10^{-11}$& $1.3 \times 10^{-11}$ &5,855 \\
  $1$ & $10^{-4}$  & $10^{-4}$ & $1.39 \times 10^{-7}$ & $1.3 \times 10^{-7}$ & $5.9\times 10^8$ \\
   $1$ & $1$ & $1$ & $13.9$& $4.1$ & $2.3 \times 10^{16}$ \\
$2$ & $10^{-6}$  & $1.4 \times 10^{-12}$ & $9.2 \times 10^{-16}$& $1.5 \times 10^{-14}$ & 68 \\
$3$ & $10^{-4}$  & $2.9 \times 10^{-12}$ & $6.5 \times 10^{-15}$& $1.6 \times 10^{-14}$ & 73 \\
$4$ & $10^{-3}$ &  $3.9 \times 10^{-12}$ & $7.9 \times 10^{-14}$& $2.8 \times 10^{-14}$ & 356 \\
$5$ & $10^{-3}$ & $2.5 \times 10^{-15}$ & $1.3 \times 10^{-15}$& $1.6 \times 10^{-14}$ & 73 \\
$7$ & $10^{-2}$  & $3.6 \times 10^{-13}$ & $2.9 \times 10^{-13}$& $8.6 \times 10^{-14}$ & 1036 \\
$10$ & $10^{-2}$ &  $1.2 \times 10^{-18}$ & $9.3 \times 10^{-16}$& $1.6 \times 10^{-14}$ & 73 \\
$15$ & $10^{-1}$  & $6.2 \times 10^{-12}$ & $3.9 \times 10^{-10}$& $5.3 \times 10^{-11}$ & $1.7 \times 10^6$ \\
$20$ & $10^{-1}$ &  $4.1 \times 10^{-15}$ & $3.9 \times 10^{-13}$& $4.9 \times 10^{-14}$ & $1,756$ \\
$25$ & $10^{-1}$ &  $2.9 \times 10^{-18}$ & $1.6 \times 10^{-15}$& $5.6 \times 10^{-14}$ & $253$ \\
$35$ & $1.5 \times 10^{-1}$  & $2.2 \times 10^{-18}$ & $2.4 \times 10^{-15}$& $1.1 \times 10^{-13}$ & $495$ \\
$45$ & $2 \times 10^{-1}$   & $3.1 \times 10^{-17}$ & $3.3 \times 10^{-14}$& $1.0 \times 10^{-13}$ & $450$ \\
$65$ & $2.5 \times 10^{-1}$ & $2.6 \times 10^{-17}$ & $7.7 \times 10^{-14}$& $1.1 \times 10^{-13}$ & $495$ \\
\noalign{\smallskip}\hline
\end{tabular}
\end{center}
\end{table}

\begin{remark}[Eigenvector scaling and coefficient decay] \label{rem:eigenvector_scaling}
Solutions of Equation \eqref{eq:invEq} are only unique up to the choice of the 
scalings of the eigenvectors and this freedom is exploited in our numerical algorithms.
Indeed, this is the reason we can always take our domain to be the unit disk.
The results in Table \ref{tab:1} describe the dependence of the  
numerical errors on the approximation order and the eigenvector 
scalings.
These numerical experiments lead to the following heuristic. If we scale the eigenvectors
so that the final coefficients -- that is the $N$-th order coefficients of $P^N$ -- 
are on the order of machine epsilon, then we obtain  
a-posteriori errors on the order of machine epsilon.  
\end{remark}

\begin{table}
\caption{\textbf{Runtime data for the parameterization method:} here the manifolds are first 
computed to order $N_0$ in order to measure the exponential decay rate associated with 
the Taylor coefficients. This data is used to determine the optimal eigenvector scaling, 
and then the coefficients are computed to order $N$ in a ``production run''.  The initial computation 
is always computed to order $N_0$ by recursion.  Then the production run is computed either by 
recursion, by Newton, or by the pseudo-Newton method. The computations were performed on a
MacBook Air with a $1.8GHz$ Intel Core i5 processor and $8$GB of 1600 MHz ram
running the version 10.12.6 of the Sierra  operating system with MATLAB version 
$2017b$. The same computations run about twenty percent faster on a 
Mac Pro desktop with a $3.7$GHz quad-core Intel Xenon E5 processor
and the same version of MATLAB.}
\label{tab:2}       

\begin{center}
\begin{tabular}{lllllll}
\hline\noalign{\smallskip}
$N$ & $N_0$ & $\tau$ & $\max_{m+n = N} \|p_{m,n}\|$  & recursion & Newton & pseudo-Newton  \\
\noalign{\smallskip}\hline\noalign{\smallskip}
 $10$ & $5$ & $0.024$ &  $8.4 \times 10^{-15}$ & $3.1$(sec) & $0.49$(sec) & $0.45$ (sec) \\
$20$ & $5$ & $0.13$ &  $2.9 \times 10^{-12}$ & $3.3$(sec) & $0.94$(sec) & $0.73$ (sec) \\
$20$ & $10$ & $0.09$ &  $9.9 \times 10^{-16}$ & $3$(sec) & $0.7$(sec) & $0.62$ (sec) \\
$30$ & $10$ & $0.15$ &  $2.9 \times 10^{-15}$ & $3.5$(sec) & $2.1$(sec) & $1.2$ (sec) \\
$30$ & $15$ & $0.15$ &  $1.2 \times 10^{-15}$ & $3.8$(sec) & $1.6$(sec) & $1.1$ (sec) \\
$30$ & $20$ & $0.15$ &  $1.2 \times 10^{-15}$ & $4.4$(sec) & $1.7$(sec) & $0.93$ (sec) \\
$40$ & $20$ & $0.21$ &  $4.5 \times 10^{-15}$ & $4.6$(sec) & $4.0$(sec) & $2.01$ (sec) \\
$70$ & $30$ & $0.27$ &  $4.1 \times 10^{-15}$ & $9.9$(sec) & $28.1$(sec) & $14.9$ (sec) \\
$70$ & $50$ & $0.27$ &  $1.2 \times 10^{-15}$ & $10.9$(sec) & $30.8$(sec) & $12.2$ (sec) \\
\noalign{\smallskip}\hline
\end{tabular}
\end{center}
\end{table}

\subsection{Integration of analytic arcs} \label{sec:advection}
In Section \ref{sec:atlas} we present a scheme for computing an atlas for the stable/unstable manifolds 
which relies on integrating analytic arcs of initial conditions by the flow generated by $f$. We describe 
this integrator in terms of power series expansions. Let us assume that 
$\gamma \colon (-1, 1) \to \mathbb{R}^4$ is an analytic arc with power series expansion 
\[
\gamma(s) = \sum_{n = 0}^\infty \gamma_n s^n \qquad \gamma_n \in \rr^4.
\]
Denote the formal series expansion 
\[
\Gamma(s, t) = \Phi(\gamma(s), t) = \sum_{m=0}^\infty \sum_{n=0}^\infty a_{m,n} s^n t^m \qquad a_{m,n} \in \rr^4.
\] 
Here, we use the variables $(s,t)$ in place of $(z_1,z_2)$ to emphasize the intuition that $s$ corresponds to the ``spatial'' parameterization along the initial data, and $t$ corresponds to the ``time'' parameterization along the flow. In other words, we consider $\Gamma$ as the solution of the parameterized 
family of initial value problems 
\[
\frac{d}{dt} \Gamma(s, t) = f(\Gamma(s, t)), 
\qquad \Gamma(s, 0) = \gamma(s), \qquad s \in (-1,1).
\]
Substituting the formal series into this IVP and matching like powers leads to the recursion relations
\[
a_{m+1, n} = \frac{1}{m+1} [f \circ \Gamma]_{m,n},
\qquad  a_{0, n} = \gamma_n,
\]
which allow us to compute the coefficients of $\Gamma$ to arbitrary order using the same methods described in Section \ref{sec:solveInvEq}. 
We also note that the precision of these formal series computations depend on convergence and domain decomposition of these series expansions which has not been addressed and will also be taken up in the following section.

\section{Building an atlas for the local stable/unstable manifold} \label{sec:atlas}
In this section let $W^*(\mathbf{x}_0)$ denote an invariant stable/unstable manifold for a saddle-focus equilibrium, $\mathbf{x}_0$. Our goal is to describe an algorithm for producing an atlas of chart maps which parameterizes a large portion of the invariant manifold. 
The union of the images of these maps is a piecewise parameterization 
of a $2$-dimensional subset of $W^*(\mathbf{x_0})$. 
Our procedure is iterative and at each step outputs a 
(strictly) larger piecewise parameterization. 

It is important to emphasize that our computations are carried out only to finite order. In particular, the charts described in this section are analytic functions of two complex variables. However, in practice we fix $(M,N) \in \nn^2$, and for each chart we compute a finite polynomial approximation of order $(M,N)$. Nevertheless, throughout this section we denote these analytic charts and their polynomial approximations using the same notation.
We end this section by outlining methods for reliably, efficiently, and automatically computing these atlases. This includes algorithms for estimating and controlling truncation errors, identifying Taylor series blowup, domain decomposition, and stiffness.

\begin{figure}[!h]\label{fig:subdivision_schematic}
	\centering
\includegraphics[width = .6\textwidth, keepaspectratio=true]{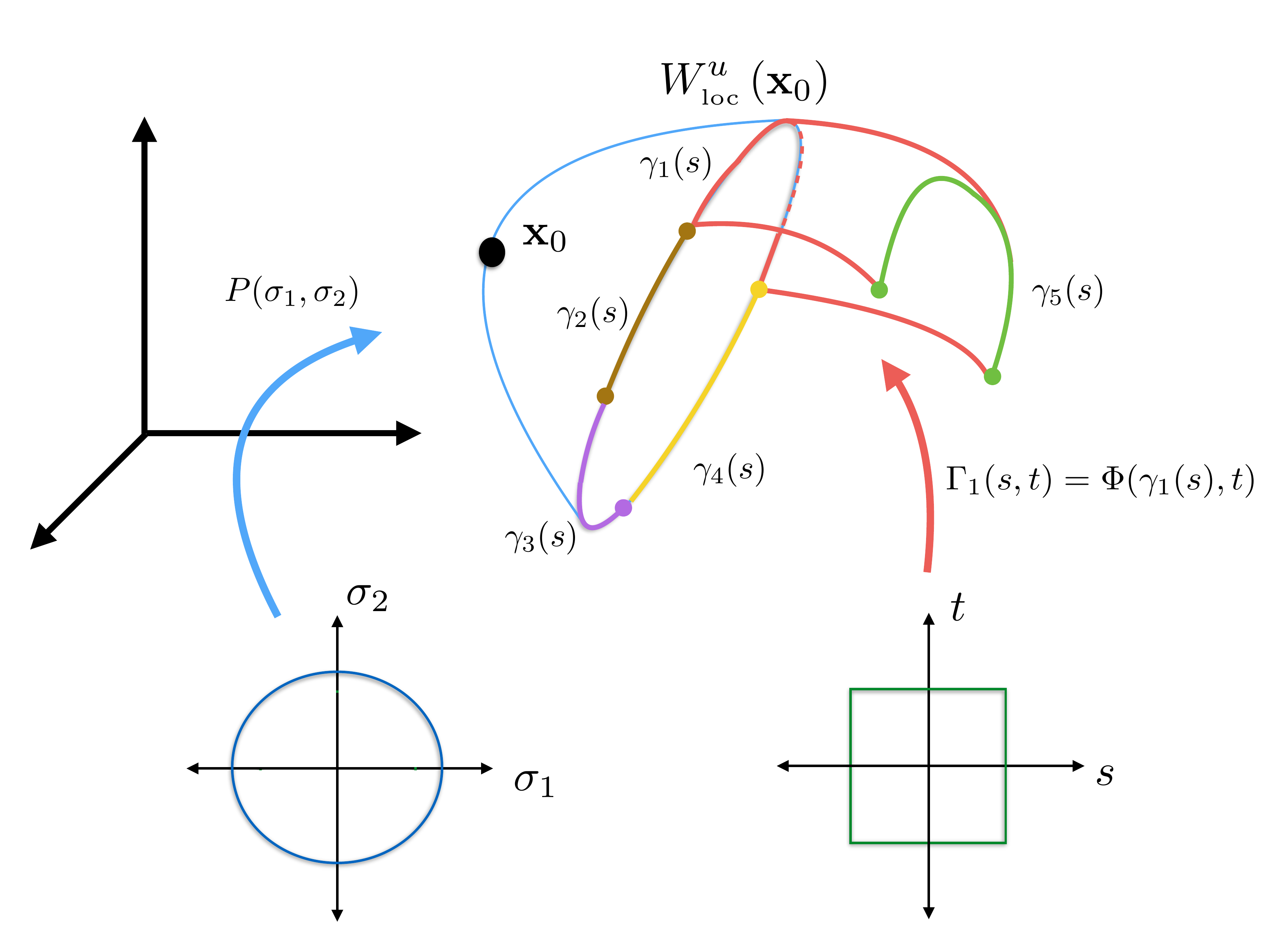}	
	\caption{\textbf{Building an atlas:}  Here $P$ is a chart for a neighborhood of the equilibrium 
	$\mathbf{x}_0$ computed using the parameterization method.  To grow the atlas we 
	mesh the boundary of the $P$ using a collection of analytic arcs $\gamma_j(s)$.  
	Each of these arcs is advected under the flow $\Phi$ to 
	produce a new chart $\Gamma_j(s, t) = \Phi(\gamma_j(s), t)$.  
	The union of $P$ with all the $\Gamma_j$ is an atlas for a larger local stable/unstable 
	manifold.  }
\end{figure}

\subsection{Iterative method for computing charts}
Before elaborating on the technical details of our method we briefly describe the overall strategy.
Starting from the parameterized local invariant manifolds obtained via methods described in 
Section \ref{sec:parmMethod}, we want to build an even larger representation 
of the manifold.  There are many ways to grow such a representation.
We could for example simply integrate a collection of initial 
conditions meshing the boundary of the parameterization.  
However, as is well known, the exponential separation 
of initial conditions will force these orbits apart and eventually 
degrade the description of the manifold.   Instead we mesh the boundary into a collection of 
one dimensional arcs and advect each of these under the flow.  
Propagating these arcs maintains the fidelity of the representation, and 
leads to new ``patches'' of the manifold.  

Since the initial chart is parameterized by a high order polynomial, 
we would like the same representation for new charts.  
To this end we develop a high order Taylor integration scheme which applies
to analytic arcs of initial conditions.  This results in a power series representation of the flow of 
a boundary arc, and we take this as our next chart.  After advecting each 
one of the boundary arcs we have a new and strictly larger representation of
local stable/unstable manifold.
The idea is illustrated in Figure \ref{fig:subdivision_schematic}.

After one step of this procedure we have moved the boundary of the local 
invariant manifold. In some cases, the image of the advected arc undergoes excessive stretching due to the exponential separation of initial conditions. This stretching in phase space is matched by a corresponding blow up in coefficients of the Taylor expansion, and the computations become numerically unstable.  

This problem is overcome by occasionally remeshing the boundary of the atlas. This comes at a cost of increasing the number of charts in the next step of the algorithm. Hence, efficiently computing large atlases while controlling numerical error requires automatic algorithms for managing the growth of the power series coefficients, deciding how long to integrate each individual arc, and deciding when 
and how to subdivide the new boundaries. These topics account for much of the technical details which follow.   

\subsubsection{The initial local manifold}
The first step in our algorithm is to compute a polynomial approximation of the 
local parameterization, either by directly solving the 
homological equations or by iterating the Newton or pseudo-Newton schemes described in 
Section \ref{sec:parmMethod}. Let $\Gamma_0$ be a solution of Equation \eqref{eq:invEq}, 
and $D^2$ denote the unit 
    polydisc in $\cc^2$. 
    Recall that $\Gamma_0: D^2 \to W^*_{\text{loc}}(\mathbf{x}_0)$ is analytic, and 
    that $\Gamma_0(\partial D^2)$ is flow transverse. In particular, $\Gamma_0$ serves 
    as our initial local parameterization, and we refer to it as a $0^{\rm th}$ 
    generation {\em interior} chart and we write 
    $\Gamma_0(D^2) = W^*_0(\mathbf{x}_0)$. 

In practice, we compute $\Gamma_0$ to order $(N,N)$ with $N \in \nn$ chosen by applying the heuristic methods discussed in Section \ref{sec:solveInvEq}.  This chart is represented in the computer as a polynomial in two complex variables of total degree $\deg(\Gamma_0) = (N-1)^2$. The truncation error of this approximation is controlled directly by choosing the eigenvector scaling as described in Remark \ref{rem:eigenvector_scaling}, and in practice, is on the order of machine epsilon.  

\subsubsection{The initial manifold boundary}
With $\Gamma_0$ in hand, we fix $K_0 \in \nn$ and subdivide $\partial D$ into $K_0$-many analytic segments, each of which has the form, $c_j: [-1,1] \to \partial D$, for $1 \leq j \leq K_0$. We parameterize $\partial W^*_0(\mathbf{x}_0)$ by defining $\gamma_j(s)  = \Gamma_0 \circ c_j(s)$ and we refer to $\gamma_j$ as a {\em lifted boundary}. Note that for each $1\leq j\leq K_0$,  $\gamma_j: [-1,1] \to \partial W^*_{\text{loc}}(\mathbf{x}_0)$ and $\gamma_j([-1,1])$ is a flow transverse arc since $\Gamma_0$ is a dynamical conjugacy and the image of  $c_j$ is transverse to the linear flow. Now, we define the $0^{\rm th}$ generation boundary to be 
\[
\partial W^*_0(\mathbf{x}_0) = \bigcup_{j=1}^{K_0} \gamma_j([-1,1]), 
\]
and refer to each $\gamma_j$ as a $0^{\rm th}$ generation {\em boundary} chart.

\subsubsection{The next generation}
Now, we apply the high-order Taylor advection described in Section \ref{sec:advection} to grow a larger local manifold denoted by $W^*_1(\mathbf{x}_0)$. Specifically, for  $1 \leq j \leq K_0$, we choose $|\tau_j| > 0$, and our advection algorithm takes $\gamma_j,\tau_j$ as input and produces a chart, $\Gamma_{1,j}: D \to W^*(\mathbf{x}_0)$ which satisfies
\[
\Gamma_{1,j}(s,t) = \Phi \left( \gamma_j(s),\frac{t}{\tau_j} \right) \qquad \text{for} \quad (s,t) \in [-1,1]^2.
\]
In other words, $\Gamma_{1,j}$ parameterizes the advected image of $\gamma_j$ under the flow over the time interval $[0,\tau_j]$. These new charts are referred to as $1^{\rm st}$ generation interior charts which we add to our atlas to obtain the first generation local parameterization
\[
W^*_1(\mathbf{x}_0) = W^*_0(\mathbf{x}_0) \cup \bigcup_{j=1}^{K_0} \Gamma_{1,j}(D). 
\]
Note that $\tau_j \neq 0$ and since $\gamma_j$ is flow transverse, we have $W^*_0(\mathbf{x}_0) \subsetneq W^*_1(\mathbf{x}_1)$ is a strict subset. In fact,
transversality of $\gamma_j$ implies the stronger condition that $\partial W^*_0(\mathbf{x}_0) \subset \text{Int}(W^*_1(\mathbf{x}_0))$ i.e.\ the manifold has grown through {\em every} point on the previous boundary. 

\begin{remark}[Time rescaling]
	\label{rem:time_rescaling}
In this description, $\tau_j$ serves as a time-rescaling of the flow. This allows direct control over the truncation error (in the time direction) and is analogous to the eigenvector scaling for the initial parameterization described in Remark \ref{rem:eigenvector_scaling}. However, choosing this time-rescaling is typically more difficult than choosing the eigenvector scaling and we postpone the discussion of this problem to Section \ref{sec:time_scaling}.
\end{remark}

Once the $1^{\rm st}$ generation interior charts are computed by advection, the $1^{\rm st}$ generation boundary arcs are now obtained by evaluation of the time variable. In particular, for $1 \leq j \leq K_0$, the evaluation, $\Gamma_{1,j}([-1,1],1) \subset \partial W^*_1(\mathbf{x}_0)$ is a flow transverse arc segment. We perform spatial rescaling as needed (see Remark \ref{rem:spatial_rescaling} below) to obtain the next generation boundary arcs, $\gamma_{1,j} : [-1,1] \to \partial W^*_1(\mathbf{x}_0)$ where $1 \leq j \leq K_1$ for some $K_1 \geq K_0$ and
\[
\gamma_{1,j}([-1,1]) \subset \Gamma_{1,j'}([-1,1],1) \qquad \text{for some} \quad 1\leq j' \leq K_0 
\]
is flow transverse. The advection and evaluation algorithms are then iterated to increase the number of charts in 
the atlas. The $L^{\rm th}$ step in the iteration chain has the form
\[
\xmapsto \cdots  \partial W^*_{L-1}(\mathbf{x}_0) \xmapsto{\text{advection}} W^*_L(\mathbf{x}_0) \xmapsto{\text{evaluation}} \partial W^*_L(\mathbf{x}_0) \xmapsto \cdots 
\]
where $W^*_L(\mathbf{x}_0)$ is parameterized by $K_{L-1}$-many interior charts (polynomials in both the space and time variables), $\partial W^*_L(\mathbf{x}_0)$ is parameterized by $K_{L}$-many boundary charts (polynomials in the space variable only), and $K_{L-1} \leq K_L$. 

If we stop iteration, say at the $L^{\rm th} step$, then the final atlas,
\[
\mathcal{A} = \left\{\Gamma_0, \bigcup_{j=1}^{K_0} \Gamma_{1,j}, \bigcup_{j=1}^{K_1} \Gamma_{2,j}, \dotsc, \bigcup_{j=1}^{K_L} \Gamma_{L,j} \right\},
\]
is a collection of $\abs{\mathcal{A}} =  1 + \sum\limits_{l=1}^{L} K_l$-many analytic charts is a piecewise parameterization 
a portion of the invariant manifold. 

\begin{remark}[Spatial rescaling]
	\label{rem:spatial_rescaling}
	The parameters, $K_0,\dotsc,K_L$, control the number of boundary subdivisions, and therefore, allow direct control over scaling in the spatial direction. As in the time-rescaling problem, choosing these parameters effectively is a nontrivial problem which we take up in Section \ref{sec:subdivision}. 
\end{remark}

\subsection{Convergence, manifold subdivision, and numerical integration}
Thus far, we have ignored the issue of convergence for our formal power series computations. The best method for studying this issue is to combine rigorous numerical computations with a-posteriori analysis and obtain a proof of the existence  of an analytic solution and explicit error bounds on the polynomial approximation. 
Rigorously validated numerical methods for invariant manifold atlases 
are described in detail in \cite{manifoldPaper1,shaneAndJay}.
In the present work we explore the utility of invariant manifold atlases as a purely numerical 
tool, and trade the computer assisted proof of rigorous error bounds for improved runtime performance.

In the absence of a rigorous validation scheme we develop more heuristic checks to insure the reliability of the 
computations.  More precisely, we must automatically identify and fix numerical accuracy issues related to numerical Taylor integration. This amounts to rescaling our Taylor coefficients whenever the decay in either space or time becomes too slow. However, this is less straight-forward than the eigenvector rescaling for the initial local parameterization described in Remark \ref{rem:eigenvector_scaling}. In particular, it is helpful to consider the rescaling in space and time ``directions''  separately.

\subsubsection{Time-stepping}
\label{sec:time_scaling}
Recall that at the saddle-focus equilibrium, the stable/unstable eigenvalues occur in complex conjugate pairs. In particular, both eigenvalues in each pair have equal real parts. It follows that identically re-scaling each pair of eigenvectors is the ideal strategy. In fact, this strategy is also necessary and sufficient to ensure that the initial parameterization is real-valued, see \cite{parmChristian}. Moreover, in the general case of a hyperbolic equilibrium, the real part of each eigenvalue is a measure of the expansion or contraction rate in the direction of its associated eigenvector. Thus, in cases for which they are not equal, the real parts are still explicitly known and the eigenvectors are scaled proportional to these rates. 

On the other hand, all but the initial chart in our atlas is obtained via our advection scheme. In this case, neither the expansion/contraction rates, or their directions are explicitly known. Obtaining these estimates would require solving for the (spatial) derivative of the flow on each chart. For a general vector field defined on $\rr^n$, this amounts to increasing the phase space dimension of our ODE solver from $n$, to $n + n^2$, which would significantly reduce the size of each manifold which is computationally feasible to produce.

Instead, we take an approach similar to \cite{manifoldPaper1}, which describes heuristics for rescaling time and space independent of one another. Specifically, we adopt a time-rescaling which ensures that the norm of the $M^{\rm th}$ ``coefficient'' (with respect to $t$) for each chart, is less than machine epsilon. Note that for a classical IVP this coefficient is of course just a scalar. However, in our case the coefficient is actually an analytic function of the spatial variable, represented as a power series and the norm of this coefficient is measured using the $\ell^1$ norm. This is made more precise in the following section.

This choice is highly conservative, which gives us tight control over the truncation error in the time direction. On the other hand, the spatial rescaling in the present work deviates from the scheme presented in \cite{manifoldPaper1} and is detailed in Section \ref{sec:subdivision}.

\subsubsection{Manifold subdivision}
\label{sec:subdivision}

Next, we describe the spatial-rescaling scheme which we refer to as {\em manifold subdivision}. We assume that the time-rescaling described in the previous section has been carried out on each chart, and our interest is in rescaling each boundary arc to control truncation errors accumulating in the ``space direction''. This is equivalent to subdividing a manifold since it is reasonable to assume the rescaling will always shrink the domain. Thus, a single boundary arc will give rise to multiple subarcs defined on reduced domains. 

To be more precise, we let $C^{\omega}$ denote the collection of real-valued, analytic functions defined on $(-1,1)$, and let $\mathcal{S}$ denote the collection of real-valued sequences. We define the {\em Taylor transform}, $\mathcal{T}: C^{\omega} \to \mathcal{S}$, to be the mapping which sends an analytic function to its sequence of Taylor coefficients centered at $z = 0$. Specifically, if $g \in C^\omega$ has the Taylor expansion, 
\[
g(z) = \sum_{n=0}^{\infty} a_n z^n \qquad a_n \in \rr, \quad z \in (-1,1),
\]
then $\TT{g} = \{a_n\} = a \in \mathcal{S}$. Now, we equip $\mathcal{S}$ with the $\ell_1$-norm defined by 
\[
\norm{a}_{1} = \sum_{n=0}^{\infty} \abs{a_n},
\]
and we note that elements of $\mathcal{S}$ with finite norm form a closed sub-algebra denoted as 
\[
\ell_1 = \{x \in \mathcal{S} \ : \norm{x}_{1} < \infty \},
\]
and we write $\norm{a}_{\ell_1}$ when we want to emphasize that $a \in \ell_1$ (i.e.~we write $\norm{a}_{\ell_1}$ for the norm $\norm{a}_1$ when $\norm{a}_1$ is finite).  

We remark that our error analysis is carried out using the $\ell_1$-norm due to the efficiency of computing this norm for polynomials. However, if $\overbar g \approx g$ is a numerical approximation, then the errors we are interested in are of the form
\[
\norm{\overbar g - g}_{\infty} = \sup_{z \in [-1,1]} \setof{\abs{\overbar g(z) - g(z)}}.
\]
We are justified in using the $\ell_1$ norm
due to the well known result that $\norm{\overbar g - g}_{\infty} \leq \norm{\overbar g - g}_{\ell_1}$.

Now, suppose $\gamma \in C^\omega$ and assume that $\TT{\gamma} = a \in \ell_1$. Since $\Phi$ is a non-linear flow, a typical arc segment undergoes rapid deformation and stretching when advected. This implies that for a single step in our algorithm with the general form,
\[
\xmapsto \cdots  \gamma \xmapsto{\text{advection}} \Gamma \xmapsto{\text{evaluation}} \gamma' \xmapsto \cdots, 
\]
we expect both the arc length and curvature of $\gamma'$ to be larger than for $\gamma$. On the level of Taylor coefficients, this statement about deformation/stretching says that if $b = \TT{\gamma'}$, then in general we expect $\norm{a}_{\ell_1} \leq \norm{b}_{\ell_1}$. The relationship between this norm and the truncation error implies that advecting an arc adversely impacts the propagation error. 

To see this, we recall that in practice our computation stores a truncated polynomial approximation for $\gamma'$ in the form $\overbar{b} = \left(b_0,\dotsc ,b_{N-1}\right)$. In order that $\overbar{b} \approx b$ is a ``good'' approximation (in the $\ell_1$ topology),  $|b_{n}|$ must be ``small'' for each $n \geq N$. These higher order terms correspond to the truncation error for $\gamma'$ and primarily arise from two sources. One source which we can not control (once $N$ is fixed) is the truncation error associated with $\gamma$. However, by inspection of the Cauchy product formula in Equation \eqref{eq:cauchy_product}, it is clear that the polynomial coefficients stored for $\gamma$ also contribute to this truncation error for $\gamma'$ after applying the nonlinearity. We refer to these contributions as {\em spillover} terms.

This observation implies that for $\bar{b} \approx b$ to be a good approximation, we must also require that $\abs{a_n}$ is ``small'' for each $n > N'$ where $N' < N$ depends on the degree of the nonlinearity. This motivates the following heuristic method for controlling truncation error for propagated arcs. We begin by assuming that $a$ has approximately geometric decay. Specifically, we expect that there exists some $r < 1$ such that the tail of the series defined by $\gamma$ decays faster than the geometric series with ratio $r$. In this case, the truncation error is of order $\mathcal{O}(r^N)$. Now, fix $0 < N' < N$, and we define the {tail ratio} for $a$ by
\begin{equation}\label{eq:tailratio}
T_{N'}(a) := \frac{\sum_{n=N'}^{N-1} \abs{a_n}}{\sum_{n=0}^{N-1} \abs{a_n}} = \frac{\norm{a - a^{N'}}_{\ell_1}}{\norm{a}_{\ell_1}}.
\end{equation}
Evidently, $T_{N'}(a)$ is small whenever ``most'' of the $\ell_1$ weight of $a$ is carried in the first $N'$-many coefficients. It follows that if $T_{N'}(a)$ is sufficiently small, then under the action of a nonlinear function, $f : \ell_1 \to \ell_1$, the spillover terms for $f(a)$ remain small. Of course, small is dependent on context and in particular, choices for $N'$ as well as thresholding values for $T_{N'}$ are problem specific. In the present work, we prove it is always possible to control $T_{N'}$. 

\begin{remark}
	\label{rem:product_spaces}
	Strictly speaking, for the CRFBP we have $\gamma = \left(\gamma^{(1)}, \dotsc, \gamma^{(4)}\right)$ where each $\gamma^{(j)} \in C^\omega$ is a coordinate for the boundary chart. Similarly, $\TT{\gamma} = \left(a^{(1)},\dotsc,a^{(4)}\right) \in \ell_1^4$, and thus the discussion in Section \ref{sec:subdivision} thus far is technically not applicable. However, our restriction to scalar valued functions is justified by the fact that if $a \in \ell_1^4$, then defining 
	\[
	\norm{a}_{\ell_1^4} = \max \setof{\norm{a^{(1)}}_{\ell_1},\dotsc,\norm{a^{(4)}}_{\ell_1}}
	\] 
	makes $\ell^4_1$ into a normed vector space. This choice of norm gives us the freedom to restrict the discussion of remeshing and tail ratios to scalar valued functions. 
\end{remark}

Next, we describe our scheme for controlling the tail ratio. This algorithm takes a polynomial representation for $\gamma$, defined on $[-1,1]$ as input, and returns a list of polynomials, $\{\gamma_1,\dotsc,\gamma_K\}$, as outputs. The key point is that these polynomials are also defined on $[-1,1]$, and they can be chosen such that $T_{N'}(\gamma_j)$ is arbitrarily small for $1 \leq j \leq K$. In this work, we assume the output polynomials are specified as coefficient vectors of length $N$ (i.e.\ the same degree as the input), however this is not required. 

This gives rise to an additional {\em remeshing} step in our algorithm which is performed as needed after an evaluation step and prior to an advection step leading to an updated schematic
\[
\xmapsto \cdots  \gamma \xmapsto{\text{remeshing}} \setof{\gamma_j}_{1\leq j \leq K} 
\xmapsto{\text{advection}} \setof{\Gamma_j}_{1 \leq j \leq K} \xmapsto{\text{evaluation}} \setof{\gamma'}_{1 \leq j \leq K} \xmapsto \cdots
\]
In the remeshing step, the tail ratio for each boundary arc from the previous step is computed and checked against a threshold. Boundary arcs which exceed this threshold are flagged as poorly-conditioned, and subdivided into smaller subarcs which satisfy the threshold. The collection of resulting subarcs and well-conditioned arcs from the previous step is passed to the advection step where each results in a separate chart. 

Before proving this threshold can always be satisfied, we describe the subdivision algorithm. As noted in Remark \ref{rem:product_spaces}, it suffices to consider a single coordinate for a parameterized boundary arc. Thus, we assume $\gamma(s): [-1,1] \to \rr$ is analytic with Taylor series
\[
\gamma(s) = \sum_{n=0}^\infty a_n s^n, 
\]
and fix a subinterval, $[s_1,s_2] \subset [-1,1]$. Define the constants
\begin{equation}
\label{eq:new_center_radius}
\hat{s} := \frac{s_1+s_2}{2} \qquad \qquad \delta := \frac{s_2-s_1}{2}
\end{equation}
and define $\hat{\gamma}: [-1,1] \to \rr$ by 
\begin{equation}
\label{eq:def_subarc}
\hat{\gamma}(s) = \sum_{n = 0}^{\infty} c_n s^n \qquad \text{where} \quad c_n = \delta^n \sum_{k = n}^{\infty} a_k \binom{k}{n} \hat{s}^{k-n}.
\end{equation}
Then $\hat{\gamma}$ is a parameterization for the arc segment parameterized by $\gamma$ restricted to $[s_1,s_2]$. In fact,  $\hat{\gamma}$ is the Taylor series for $\gamma$ after re-centering at $\hat{s}$ and re-scaling by $\delta$ which satisfies the functional equation
\begin{equation}\label{eq:subarc_functionaleqn}
\hat{\gamma}(s) = \gamma(\hat{s} + \delta s) \qquad s \in [-1,1].
\end{equation}
Moreover, the mapping $a \mapsto c$ is a linear transformation on $\mathcal{S}$, and in particular, if $a_n = 0$ for all $n \geq N$, then $c_n = 0$ for all $n \geq N$ also. Now, we prove that we have explicit control over the tail ratio for $\hat{\gamma}$. 

\begin{proposition}[Controlling tail ratios]
	\label{prop:controltailratio}
	Suppose $\gamma: [-1,1] \to \rr$ is analytic, fix $\hat{s} \in (-1,1)$, $1 \leq N' \leq N$, and let $\epsilon > 0$. Then there exists $\delta > 0$ such that $T_{N'}(c) < \epsilon$ where $c$ is the truncation to order $N$ for $\hat{\gamma}: [-1,1] \to \rr$ defined by $\hat{s},\delta$ as in Equation \eqref{eq:def_subarc}.
\end{proposition}
\begin{proof}
	Define $\gamma^N: [-1,1] \to \rr$ to be the Taylor polynomial obtained by truncating the Taylor series for $\gamma$ to order $N$. For $k \in \nn$, define the usual $C^k$-norm on $[-1,1]$ to be 
	\[
	\norm{g}_{C^k} = \max\limits_{0 \leq j \leq k} \left\{\norm{g^{(j)}}_{\infty} \right\}.
	\]
	Since $\gamma^N$ is a polynomial, we have the bound
	\[
	\norm{\gamma^N}_{C^k} \leq M := \norm{\gamma^N}_{C^{N-1}} \qquad \text{for all} \quad k \in \nn. 
	\]
	In particular, for any $\hat{s} \in (-1,1)$, we have $\left|\gamma^{(n)}(\hat{s}) \right| \leq M$, for $0 \leq n \leq (N-1)$, and we define  
	\[
	\delta := \min\limits_{N' \leq n \leq N} \left\{ \left(\frac{\epsilon \abs{\gamma(\hat{s})}}{M (N - N')}\right)^{\frac{1}{n}} \right\}.
	\]
	It follows that 
	\[
	\delta^n \abs{\gamma^{(n)}(\hat{s})} \leq \frac{\epsilon\gamma(\hat{s})}{N-N'} \qquad \text{for all} \quad N' \leq n \leq N.
	\]
	
	Now, let $\hat{\gamma}$ be defined as in Equation \eqref{eq:def_subarc}. Recall that $\hat{\gamma}$ is also analytic on $[-1,1]$ and by differentiating Equation \eqref{eq:subarc_functionaleqn} we have the derivative formula, $\hat{\gamma}^{(n)}(s) = \delta^n \gamma^{(n)}\left(\hat{s} + \delta s\right)$, for all $n \in \nn$. By Taylor's theorem, we obtain another explicit formula for $c_n$ given by
	\[
	c_n = \frac{\hat{\gamma}^{(n)}(0)}{n!} = \frac{\delta^n \gamma^{(n)}(\hat{s})}{n!},
	\]
	and we note that $c_0 = \hat{\gamma}(0) = \gamma(\hat{s})$ does not depend on $\delta$. We have the estimate for the tail ratio of $\hat{\gamma}$: 
	\begin{align*}\label{eq:tailratiobound}
	T_{N'}(c) &= \frac{1}{\norm{c}_{\ell_1}} \sum_{n=N'}^{N-1} \abs{c_n} \\
	&= \frac{1}{\norm{c}_{\ell_1}} \sum_{n=N'}^{N-1} \frac{\delta^n \abs{\gamma^{(n)}(\hat{s})}}{n!} \\
	&\leq \frac{1}{\abs{c_0}} \sum_{n=N'}^{N-1} \frac{\epsilon \abs{\gamma(\hat{s})}}{N - N'} \\
	&= \epsilon\\
	\end{align*}
	which completes the proof.
\end{proof}

Proposition \ref{prop:controltailratio} establishes the fact that we may re-parameterize $\gamma$ on subintervals of $[-1,1]$ with width, $2\delta$, and that as $\delta \to 0$ the tail ratio also approaches zero. We note that $\delta$ does not depend on the subinterval, and therefore, for a fixed $\epsilon$ the number of required subarcs is finite. In particular, no more than $K = \lceil \frac{2}{\delta} \rceil$ subarcs are required. To summarize the usefulness of this result, we present the following algorithm for controlling the spatial truncation error which was implemented for the atlases in this work. 
\begin{enumerate}
	\item Fix a threshold $0 < \epsilon \ll 1$, a cutoff $1 \leq N' < N$, and $K \in \nn$. The threshold and cutoff are both chosen based on the alignment of $\gamma$ with the flow, the degree of the non-linearity in $f$, and the truncation size. In practice, these are problem specific choices which require some ad-hoc experimentation in order to balance computational efficiency and truncation error. 
	\item Following each evaluation step in our algorithm, a boundary arc has the form $\gamma: [-1,1] \to \rr$ which is stored in the computer as a polynomial approximation, $\overbar a = \left(a_0,\dotsc,a_{N-1}\right)$. If $T_{N'}(\overbar a) < \epsilon$, continue to the advection step. 
	\item If $T_{N'}(\overbar a) \geq \epsilon$, specify a partition of $[-1,1]$ into $K$-many subintervals by choosing their endpoints,  $\{s_0,s_1,\dotsc,s_K\}$. Apply the formula in Equation \eqref{eq:def_subarc} to obtain $\{\gamma_1,\dotsc,\gamma_K\}$ where for $1 \leq j \leq K$,  $\gamma_j(s) = \gamma(\hat{s}_j + \delta_j s)$ where $\hat{s}_j = \frac{s_j + s_{j-1}}{2}$ and $\delta_j = \frac{s_{j} - s_{j-1}}{2}$. 
	\item Each resulting subarc which satisfies the tail ratio threshold passes to the advection step. Subarcs which violate the threshold are subdivided again by repeating step 3. By Proposition \ref{prop:controltailratio}, this condition is eventually met for every subarc and the algorithm proceeds to the advection step.  
\end{enumerate}

\subsubsection{Stiffness}
\label{sec:stiffness}
The final numerical consideration which we address is the stiffness problem. We recall that the CRFBP vector field is analytic away from the primary masses which correspond to singularities of Equation \eqref{eq:SCRFBP}. Since this system is Hamiltonian, any trajectory which collides with one of these primaries must blow up in finite time. However, smooth trajectories may pass arbitrarily close to these primaries and as they do, the velocity coordinates, $\dot x, \dot y$, become arbitrarily large.  

Recall that a single boundary arc, $\gamma: [-1,1] \to \rr^4$, is a parameterized manifold of initial data. Then its advected image, $\Gamma:[-1,1] \times [0,1] \to \rr^4$, is a parameterized bundle of trajectory segments. For any $s_0 \in [-1,1]$, $\Gamma(s_0,t)$ parameterizes the trajectory passing through $\gamma(s)$ over the (non-scaled) time interval, $[0,\tau]$. 

Now, suppose that for $s_0 \in [-1,1]$, the trajectory through $\gamma(s_0)$ passes ``close'' to a primary at time $t = t_0$. Then, we have 
\[
\norm{f\left(\Gamma(s_0,t_0)\right)}_{\cc^4} \gg 1.  
\]
Recalling our time-rescaling algorithm described in Section \ref{sec:time_scaling}, it is clear that controlling truncation in the time direction will require taking increasingly shorter time-steps. Of course this is not surprising, however, the difficulty arises from the fact that other choices of $s \in [-1,1]$ often correspond to trajectory segments which remain far away from the primary and our time-rescaling is applies uniformly on $[-1,1]$. Hence, the advection of the entire boundary chart is slowed dramatically whenever any portion of its image approaches a primary. We refer to these charts as {\em stiff}. Obviously, this is a major problem for our ``breadth-first'' approach for computing the manifold atlas. Namely, the integrator gets stuck on the stiff charts causing the computation to stall.  

A naive method for dealing with this is to define the {\em speed} for a boundary chart which is a parameterized curve of the form, 
$\gamma(s) =$ $(x(s),\dot x(s)$, $y(s), \dot y(s) )$, by 
\begin{equation}
\label{eq:speed}
	S(\gamma) = \sup_{s \in [-1,1]} \setof{ \sqrt{\dot x(s)^2 + \dot y(s)^2}}, 
\end{equation}
set a threshold, $\kappa$, and cease advection of $\gamma$ whenever $S(\gamma) > \kappa$. While this fixes the problem of computational efficiency, we also lose large portions of the manifold which remain far from the primaries. Instead, we leverage the manifold subdivision procedure which was already introduced in Section \ref{sec:subdivision} to modify the naive algorithm in order to retain these portions of the manifold as follows. 
\begin{enumerate}
	\item Fix a maximum speed threshold, $\kappa >0$. For each boundary chart, $\gamma$, present after the evaluation step, check that $S(\gamma) \leq \kappa$ and if so, continue to the remeshing step.  
	\item If $S(\gamma) > \kappa$, write $\gamma(s) = \left(x(s),\dot x(s), y(s), \dot y(s)\right)$ and compute 
	\[
	\setof{s \in [-1,1] : \dot x(s)^2 + \dot y(s)^2 - \kappa^2 = 0}. 
	\]
	Since $\dot x, \dot y$ are polynomial approximations, this set is a finite collection of roots of a polynomial which we denote by, $\setof{s_0,\dots,s_K}$. 
	\item For $1\leq j \leq K$, check that $\dot x(s)^2 + \dot y(s)^2 - \kappa^2 < 0$ holds on $[s_j,s_{j+1}]$ and if so, compute $\hat \gamma_j$ as in Equation \eqref{eq:def_subarc} and continue to the remeshing step. Subintervals which fail this check are discarded. 
\end{enumerate}

To summarize, our algorithm identifies regions of the manifold boundary which pass close to a primary by checking the maximum speed. Regions which exceed a threshold are cut away while regions of the nearby boundary continue to be advected. The cut regions cause the apparent holes punched out around each primary in the manifold plots, such as Figures \ref{fig:manifoldsTopDownL0} and \ref{fig:manifoldsTopDownL5}.

\begin{figure}[!t]
\centering
\includegraphics[width=6in]{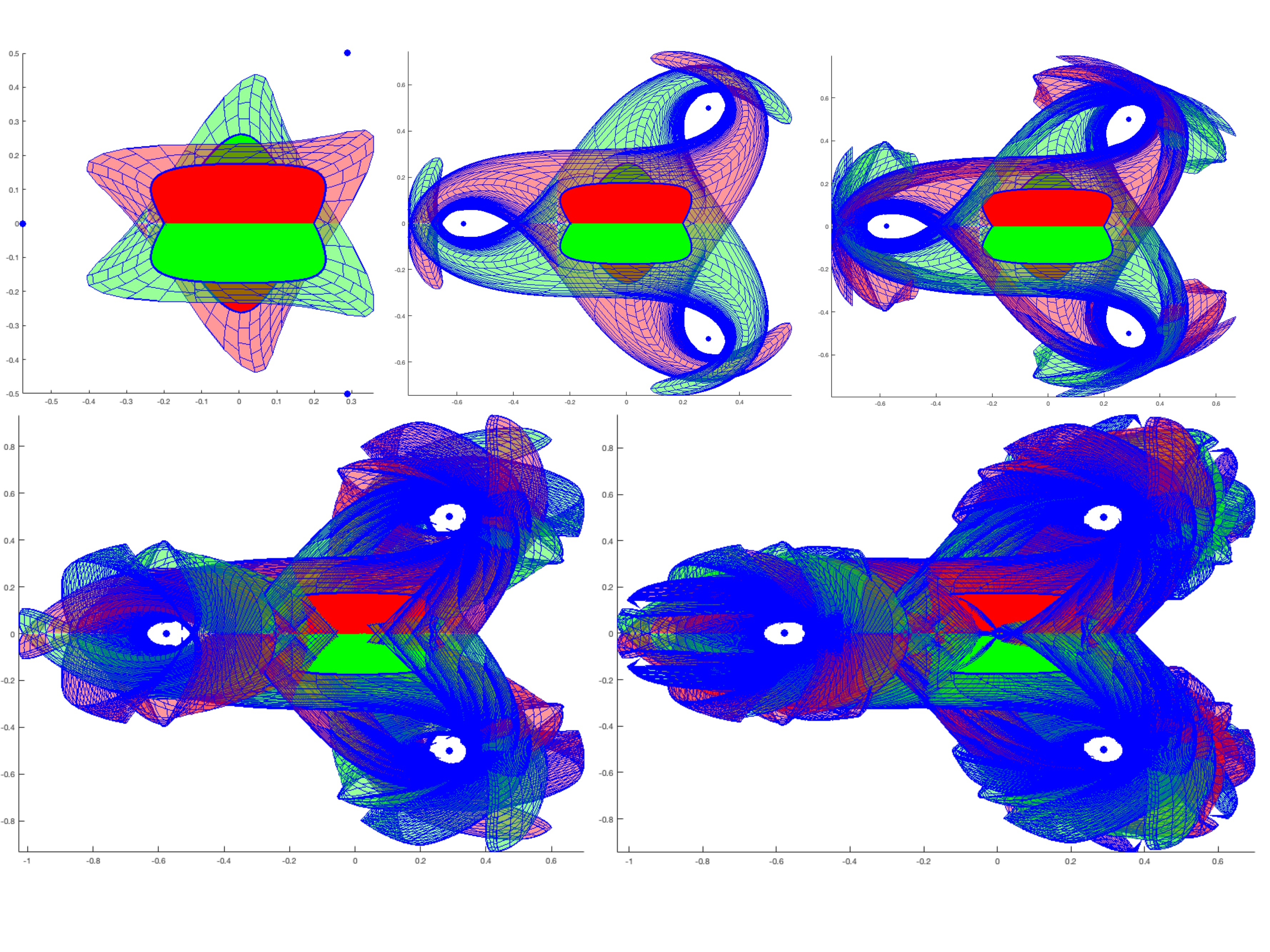}
\caption{\textbf{Atlases at $\mathcal{L}_0$ in the Triple Copenhagen Problem:} 
the center of each frame shows the initial local stable chart (green), and unstable chart (red), computed to order $45$ using the 
parameterization method. The three blue dots in each frame represent the location of 
the primaries. One third of the boundary of each local manifold is meshed into ten analytic arcs and propagated in time with the boundary of each chart illustrated in blue. By Lemma \ref{lem:CRFBP_symmetry}, the rest of the atlas is obtained via $\pm 120$ degree rotations. 
The five frames illustrate the complete atlases obtained after advecting the boundary arcs 
for $\pm 0.25, \pm 1.0, \pm 1.5$ time units (top row)
and $\pm 2.5$ and $\pm 4.0$ time units (bottom row). 
After a fairly short integration time, the resulting 
atlases become complicated enough that visual analysis is difficult or impossible.
This complexity motivates development of the post processing 
schemes described in Section \ref{sec:atlasMining}.
Each chart is approximated using Taylor order $20$ in space and 
$40$ in time.   Runtime and number of charts are as given in 
Table \ref{tab:atlasDataL0}.}\label{fig:manifoldsTopDownL0}
\end{figure}

\subsection{Computational results: manifold atlases for the triple Copenhagen problem} 
\label{sec:manifoldResults}

Performance results for atlas computations at the libration points
$\mathcal{L}_{0}$ and $\mathcal{L}_5$ are given in Tables 
\ref{tab:atlasDataL0} and \ref{tab:atlasDataL5} respectively.  
The computations are performed for the case of equal masses, that is 
for the triple Copenhagen problem.  The tables report the advection time
-- that is the number of time units the boundary of the local parameterizations
are integrated -- as well as the time required to complete the computations
and the number of polynomial charts comprising the atlas. 
All computations were performed on a MacBook Air laptop 
running Sierra version 10.12.6, on a 1.8 GHz Intel Core i5, 
with 8 GB of 1600 MHz DDR3 memory.

\begin{table}[t!]
\caption{\textbf{Atlas Computations at $\mathcal{L}_0$ in the triple Copenhagen problem:} 
each chart is computed to 
polynomial order $20$ in space and order $40$ in time.  Unlike the later computations where we have used a speed threshold of 2, here we set the threshold at 3 to better
illustrate how the atlas size and computation time grow when propagating the stable/unstable manifolds. Additionally, by Lemma \ref{lem:CRFBP_symmetry} we must only consider one third of the boundary of each local manifold, so our initial subdivision into 10 sub arcs actually corresponds to mesh which is 3 times finer. Moreover, the computation time is 3 times faster and the final atlas size is 3 times smaller than for the CRFBP with non-equal masses.}
\label{tab:atlasDataL0}       

\begin{center}
\begin{tabular}{llll}
\hline\noalign{\smallskip}
Integration Time & Run Time (both manifolds) & \# Stable Charts & \# Unstable Charts  \\
\noalign{\smallskip}\hline\noalign{\smallskip}
$\pm 0.25$ & $17.07$ seconds & 39 & 39 \\
$\pm 0.5$ & $37.9$ seconds & 146 & 146 \\
$\pm 0.75$ & $147$ seconds & 497 & 497 \\
$\pm 1.0$ & $4.75$ minutes & 700 & 700 \\
$\pm 1.5$ & $8.3$ minutes & 1579 & 1579 \\
$\pm 2.5$ & $21.8$ minutes &3530& 3493 \\
$\pm 4.0$ & $60.8$ minutes & 9372& 9295 \\
\noalign{\smallskip}\hline
\end{tabular}
\end{center}
\end{table}

\begin{table}[t!]
\caption{\textbf{Atlas Computations at $\mathcal{L}_5$ in the triple Copenhagen problem:} 
each chart is computed to polynomial order $20$ in space and order $40$ in time.
Velocities greater than 2.5 are discarded. 
We consider the entire boundary of the local stable/unstable manifolds, 
and we initially divide into 30 sub arcs.  Because of this 
the computations are roughly $3$ times longer than at $\mathcal{L}_0$.
But we obtain the manifolds at $\mathcal{L}_{4,6}$
by rotational symmetry.}
\label{tab:atlasDataL5}       
\begin{center}
\begin{tabular}{llll}
\hline\noalign{\smallskip}
Integration Time & Run Time (both manifolds) & \# Stable Charts & \# Unstable Charts  \\
\noalign{\smallskip}\hline\noalign{\smallskip}
$\pm 0.5$ & $40.5$ seconds & 124 & 124 \\
$\pm 0.75$ & $57.7$ seconds & 216 & 216 \\
$\pm 1.0$ & $2.3$ minutes & 487 & 487 \\
$\pm 1.5$ & $7.2$ minutes & 634 & 634 \\
$\pm 2.0$ & $15.3$ minutes & 1466 & 1466 \\
$\pm 3.0$ & $32.9$ minutes & 2899& 2899 \\
$\pm 4.0$ & $53$ minutes &4983& 4751 \\
\noalign{\smallskip}\hline
\end{tabular}
\end{center}
\end{table}

\begin{figure}[!t]
\centering
\includegraphics[width=6in]{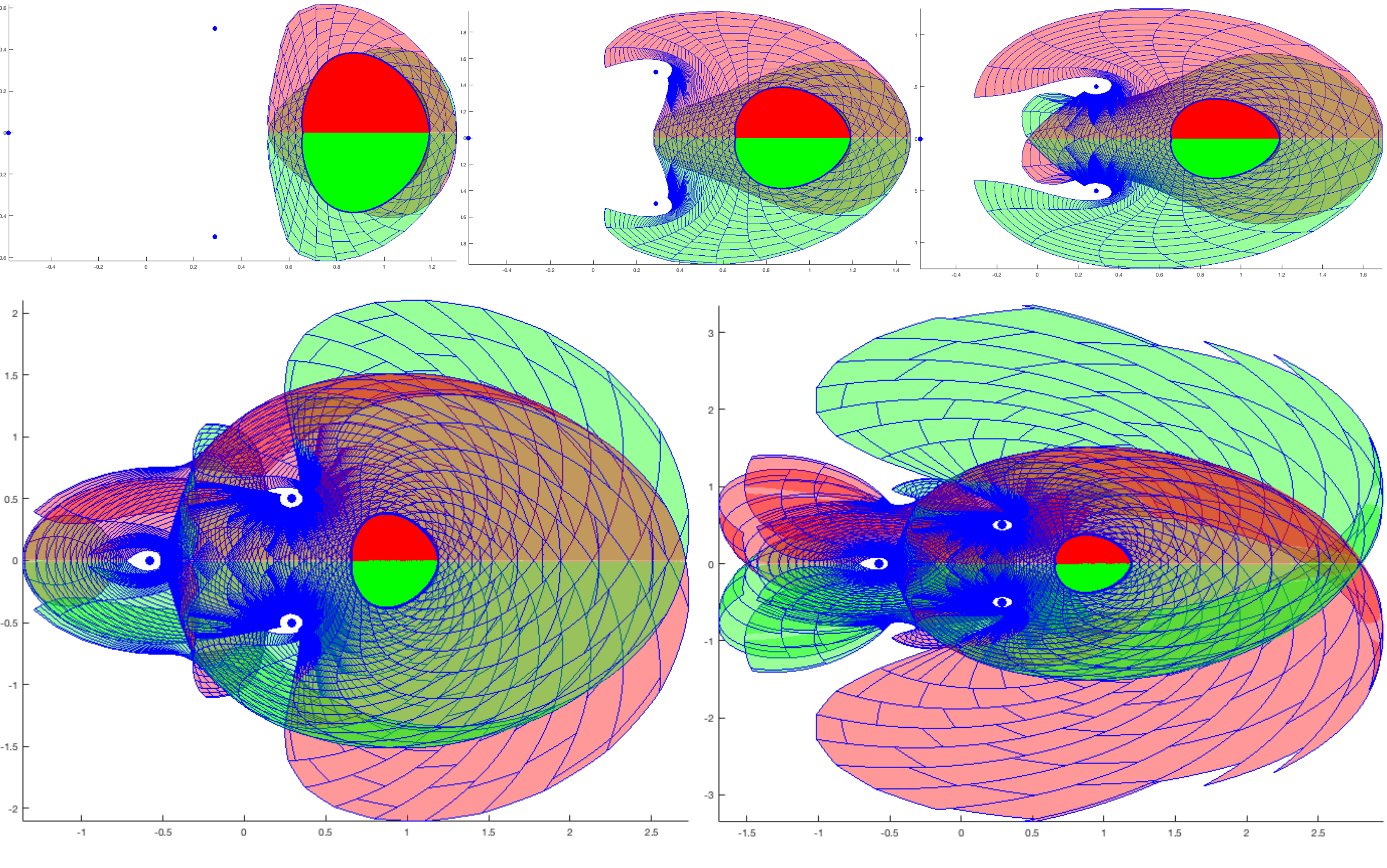}
\caption{\textbf{Atlases at $\mathcal{L}_5$ in the Triple Copenhagen Problem:} 
center of each frame shows the local stable/unstable charts computed using the 
parameterization method (red and green respectively).  
The locations of the primaries are denoted by the blue dots in 
each frame. Parameterizations approximated to polynomial 
order $45$.  The boundary of the local stable/unstable manifold is meshed into thirty
analytic arcs.  The five frames illustrate the atlases obtained by advecting the 
boundary arcs by $\pm 0.75, \pm 1.15, \pm 1.5$ time units (top row)
and by $\pm 3.0$ and $\pm 4.0$ time units (bottom row).
Again it is difficult to analyze the results by eye, and some post processing 
is necessary.  Each chart is approximated using Taylor order $20$ in space and 
$40$ in time.   Runtime and number of charts are as given in 
Table \ref{tab:atlasDataL5}.}\label{fig:manifoldsTopDownL5}
\end{figure}

The resulting atlases for $\mathcal{L}_0$ and $\mathcal{L}_5$
are illustrated in Figures \ref{fig:manifoldsTopDownL0}
and \ref{fig:manifoldsTopDownL5} for various
integration times. The boundaries for the charts are also shown, making 
it clear that the computational effort goes up dramatically 
near the primaries.  Note that the chart boundary lines running out 
of the local parameterizations are actual orbits of the system and 
hence give a sense of the dynamics on the manifold.  
The pictures provide some insight into the dynamics of 
the problem, however, their complexity illustrates the need for more sophisticated search techniques in order to extract further useful qualitative information 
from the atlases.

\section{Homoclinic dynamics in the CRFBP} \label{results}
In this section we discuss connecting orbits found for the 
symmetric $m_1 = m_2 = m_3 = 1/3$ case by searching the 
manifold atlases computed in the previous section.  

%
%
%
%
%


\subsection{Mining the atlases}
\label{sec:mining_the_atlases}
Assume we have computed atlases, $\mathcal{A}^{s,u}$, for the stable/unstable manifolds of $\mathbf{x_0}$.  We are interested in ``mining'' the chart data to find transverse connections. Since each atlas is stored as a collection of polynomial charts, it suffices to identify pairwise intersections between stable and unstable charts. Thus, throughout we assume $\Gamma^{s,u} : [-1,1]^2 \to W^{s,u}(\mathbf{x}_0)$ is a pair of charts which parameterize a portion of the stable/unstable manifold. We write $\Gamma^{s,u}_{1,2,3,4}$ denote the scalar coordinates of each chart.  The following theorem whose proof can be found in \cite{shaneAndJay} provides a computable condition for verifying transverse intersection of a pair of charts. 

\begin{theorem}
\label{thm:transverse_intersection} 
Define $G \colon [-1,1]^3 \to \mathbb{R}^3$
by 
\[
G(s, t, \sigma) := 
\left(
\begin{array}{c}
\Gamma_1^u(s,t) - \Gamma_1^s(\sigma, 0) \\
\Gamma_2^u(s,t) - \Gamma_2^s(\sigma, 0) \\
\Gamma_3^u(s,t) - \Gamma_3^s(\sigma, 0) \\
\end{array}
\right),
\]
and suppose $(\hat s, \hat t, \hat \sigma) \in [-1, 1]^3$ satisfies $G(\hat s, \hat t, \hat \sigma) = 0$.  
If $\Gamma_4^u(\hat s, \hat t)$ and $\Gamma_4^s(\hat \sigma, 0)$ have the same sign, then $\hat{\mathbf{x}} := \Gamma^u(\hat s, \hat t)$ is
homoclinic to $\mathbf{x}_0$. Moreover, if $DG(\hat s, \hat t, \hat \sigma)$ is nonsingular and if $\nabla E(\hat{\mathbf{x}}) \neq 0$
(where $E$ is the CRFBP energy), then the energy level set is a smooth $3$-manifold near $\hat{\mathbf{x}}$ and the stable/unstable 
manifolds of $\mathbf{x}_0$ intersect transversally in the energy manifold. 
\end{theorem}

We emphasize that Theorem \ref{thm:transverse_intersection} provides a computable condition for verifying a transverse intersection using rigorous numerics. However, we will use the same theorem to detect transverse intersections in the purely numerical setting of this paper. This is made explicit in the following algorithm utilized in the mining scheme for all results in the present work. 

Assume, $\Gamma^{s,u}, G$ are as defined in Theorem  \ref{thm:transverse_intersection}.  Apply Newton's method to find an approximate root of $G$. 
Let $\hat{v} = \left( \hat{s}, \hat{t}, \hat{\sigma}\right)$ 
denote an approximate solution with $G(\hat{v}) \approx 0$, and check the following conditions: 
\begin{enumerate}
	\item $\Gamma_4^u(\hat{s},\hat{t})$, and $\Gamma_4^s(\hat{\sigma},0)$ are both ``far'' from 0. 
	\item $\Gamma_4^u(\hat{s},\hat{t})$, and $\Gamma_4^s(\hat{\sigma},0)$ have the same sign.
\end{enumerate}

If condition 1 holds without condition 2, then these charts are non-intersecting. In this case, these charts lie on separated portions of the stable/unstable manifolds which are symmetric with respect to the fourth coordinate. We refer to these as ``pseudo-intersections''. On the other hand, if both conditions hold, then we conclude from Theorem \ref{thm:transverse_intersection} that we have numerically found a transverse homoclinic for $\mathbf{x}_0$ passing through $\Gamma^u(\hat{s}, \hat{t}) = \mathbf{\hat{x}}$.  

Note that condition 1 serves two purposes in this setting. First, it serves as an easily computable condition for checking that $\nabla E(\mathbf{\hat{x}}) \neq 0$ as required in the theorem. This follows by noting that 
\[ 
\pi_4 \circ \nabla E(\mathbf{\hat{x}}) = \mathbf{\hat{x}}_4 = \Gamma_4^u(\hat{s},\hat{t})
\]
so it follows that $\nabla E(\hat{\mathbf{x}}) \neq 0$ is satisfied automatically whenever condition 1 is satisfied. 

In addition, condition 1 gives us some confidence that the sign difference from condition 2 holds due to transversality of the homoclinic, as opposed to numerical error. Indeed, if condition 1 is not satisfied, then  $\Gamma_4^u(\hat{s},\hat{t})$, and $\Gamma_4^s(\hat{\sigma},0)$ take values near zero in which case sign errors for either coordinate are likely due to integration errors. In this case, even if condition 2 is satisfied we are unable to trust the result, hence unable to conclude whether the zero of $G$ corresponds to a transverse intersection or a pseudo-intersection. Fortunately, this situation can be remedied as discussed in Remark \ref{rem:following_connections}. As a result, we are free to  choose our threshold for what is meant by ``far'' in the statement of condition 1 very conservatively which leads to a great deal of confidence that our mining algorithm returns only transverse homoclinic orbits. 

We further increase our confidence in the approximate connection by using it as the input for a BVP solver based on Newton's method, which allows us to refine our approximation to nearly machine precision and it is the BVP formulation to which we then apply continuation methods.
Every connection reported in this section has been so certified and none of the connections identified from the mining algorithm had a BVP which failed to converge. In other words, the mining algorithm did not return any false homoclinics.

\subsection{Efficient atlas mining} \label{sec:atlasMining}
It is not desirable to check every pair of charts from each atlas using the above procedure,
and 
we introduce two methods which significantly reduce the number of chart pairs which must be checked via the Newton intersection scheme based on Theorem \ref{thm:transverse_intersection}.

\subsubsection{The $\ell_1$ box approximation}
The first method for improving the mining efficiency is to apply a coarse preprocessing step to each pair of charts which must be compared. The main idea is based on the fact that for most pairs of charts which do not intersect, these charts will ``obviously'' not intersect in the sense that their images in phase space will be very far apart. We exploit this using a fast algorithm for identifying many such pairs, and in this case skip the slower Newton-based intersection attempt. 

To be more precise, consider an arbitrary polynomial  $P: [-1,1]^2 \to \rr$  defined by 
\[
P(s,t) = \sum_{m = 0}^{M} \sum_{n = 0}^{N} a_{m,n}s^nt^m \quad a_{m,n} \in \rr.
\]
We define the {\em $\ell_1$ box} for $P$ to be 
\[
B_P = [a_{0,0} - r, a_{0,0} + r] \qquad \text{where } \ r =  \sum_{(m,n) \neq (0,0)} \abs{a_{m,n}}.
\]
The significance of $B_P$ is that we have the bound
\[
\abs{P(s,t) - a_{0,0}} \leq r \qquad \text{for all } (s,t) \in [-1,1]^2
\]
or equivalently, $P(s,t) \in B_P$ for all $(s,t) \in [-1,1]^2$. Analogously, we extend this to higher dimensions component-wise and apply this to geometrically rule out pairs of charts which can not intersect because their images are ``well separated''. Specifically, consider a pair of stable/unstable charts
\[
\Gamma^s(s,t) = \sum_{m = 0}^{M} \sum_{n = 0}^{N} a_{m,n}s^nt^m \qquad \Gamma^u(s,t) = \sum_{m = 0}^{M} \sum_{n = 0}^{N} b_{m,n}s^nt^m.
\]
which have $\ell_1$ boxes described by rectangles in $\rr^4$ and satisfying $\Gamma^s(s,t) \in B_{\Gamma^s}$, and $\Gamma^u(s,t) \in B_{\Gamma^u}$. Then, if the set distance, $d(B_{\Gamma^s}, B_{\Gamma^u})$ is large enough, we can conclude that $\Gamma^s,\Gamma^u$ do not intersect. 

Using $\ell_1$ boxes has two advantages. The first is that computing and checking $\ell_1$ boxes for pairwise intersections is much faster than our Newton-like intersection method. This is due to the fact that for each coordinate the box radius, $r$, is equivalently computed as 
\[
r = a_{0,0} + \norm{P}_{\ell_1} - \abs{a_{0,0}}
\]
which is extremely fast to compute using modern implementations. Determining whether two boxes intersect or not is also fast due to efficient interval arithmetic libraries such as the INTLAB library for MATLAB \cite{Ru99a} which was utilized in our implementation.

The second advantage is that an $\ell_1$ box is typically a very coarse enclosure for the true values of $P$. This ``problem'' is often referred to as the data-dependence problem or the wrapping effect. In our situation however, we consider the coarseness to be a feature since it makes our numerical estimates more conservative. Thus, we are able to rule out many pairs of charts which clearly do not intersect without eliminating false negatives.

In practice, a single pairwise $\ell_1$ box intersection check is approximately 1,000 times more efficient than the Newton-based scheme and this method rules out around 90 percent of non-intersecting chart pairs. Moreover, the $\ell_1$ box for each chart can be computed only once during the atlas construction and stored. This leaves the cost of a single box intersection check as the only significant computational operation. 

Finally, we remark that once $\ell_1$ boxes have been computed and stored for each chart in both atlases, one can make careful use of the triangle inequality to reduce the computation even further. This provides roughly an additional order of magnitude improvement in the efficiency of our algorithm which could be crucial to the feasibility of mining extremely large atlases. However, we took limited advantage of this fact in the present work. 

\subsubsection{Fundamental domains}
The other main source of efficiency gain in our algorithm relies on using the dynamics explicitly. Recalling our notation in Section \ref{sec:atlas}, assume $\mathcal{A}^s$ is the stable manifold atlas which we have computed to include the $L_s^{\rm th}$ generation and let $W^{s}_{k}(\mathbf{x}_0)$ denote the $k^{\rm th}$ generation local stable manifold. Then, $W^{s}_k(\mathbf{x}_0)$ is a fundamental domain for $W^{s}(\mathbf{x}_0)$. In other words, if $\mathbf{x}(t)$ is any orbit which satisfies $\lim\limits_{t \to \infty} \mathbf{x}(t) = \mathbf{x}_0$ and if $\mathbf{x}(0) \neq \mathbf{x}_0$, then there exists $t_k \in \rr$ such that $\mathbf{x}(t_k) \in W^{s}_k(\mathbf{x}_0)$. Of course, the same claim holds for the unstable manifold. Taken together, if we assume we have computed the unstable manifold, $\mathcal{A}^u$, up to the $L_u^{\rm th}$ generation, then we have the following observation.
\begin{prop}
	\label{prop:fundamental_domain}
	Let $\mathbf{x}(t)$ be a transverse homoclinic to $\mathbf{x}_0$. Then $\mathbf{x}(t) \in W^s(\mathbf{x}_0) \cap W^u(\mathbf{x}_0)$ for all $t \in \rr$. Let $W^{s,u}_0(\mathbf{x}_0),W^{s,u}_1(\mathbf{x}_0),\dotsc,W^{s,u}_{L_{s,u}}(\mathbf{x}_0)$ denote the generation sequence of local stable/unstable manifolds. Then exactly one of the following is true. 
	\begin{itemize}
		\item There exists $k_s,k_u$ and $t_0 \in \rr$, such that $\mathbf{x}(t_0) \in W^s_{k_s}(\mathbf{x}_0)\bigcap W^u_{k_u}(\mathbf{x}_0)$ and $k_s + k_u$ is constant for all pairs $(k_s,k_u)$ which satisfy this property. 
		\item There exists $t_0\in \rr$ such that for all $0 \leq k_s \leq L_s$, and $0\leq k_u \leq L_u$,  we have 
	\[
	\mathbf{x}((-\infty, t_0)) \bigcap W^s_{k_s}(\mathbf{x}_0) = \emptyset \quad \text{and} \quad \mathbf{x} (( t_0,\infty))\bigcap W^u_{k_u}(\mathbf{x}_0) = \emptyset.
	\]
	\end{itemize}
\end{prop}
Proposition \ref{prop:fundamental_domain}  says that any transverse homoclinic for $\mathbf{x}_0$ satisfying the second condition is a connection which does not intersect in the atlases which we have computed. Restricting to those that do, this proposition says that there is a ``first'' generation for both the stable and unstable atlases for which the connection will appear.  

The significance of this situation is that we need only do pairwise comparisons between stable/unstable charts one generation at a time. Thus, the computational complexity for mining intersections between the two atlases has computational complexity of order $\mathcal{O}(K_sK_u (L_s + L_u))$ where $K_s,K_u$ are the sizes of the largest stable/unstable generations respectively. This is a dramatic improvement over the naive solution of checking every pair in both atlases which has complexity on order $\mathcal{O}(L_sL_uK_sK_u)$. 

\begin{remark}
	We note that often the atlases we compute in practice do not technically satisfy the fundamental domain property. This is due to the fact that sections of manifold boundary which pass near a primary are ``cut out'' as described in Section \ref{sec:stiffness}. Nevertheless, this has no impact on our mining algorithm. Specifically, each generation is still a fundamental domain for the subset of the global manifold which satisfies the speed constraint. Thus, mining for connections via ``leapfrogging'' through pairwise generations is still assured to find all connections which are present in the computed atlases, and therefore, all connections which satisfy the speed constraint. 
\end{remark}

\begin{remark}
	\label{rem:following_connections}
	The result in Proposition \ref{prop:fundamental_domain} gives rise to a natural mining algorithm. Namely, at each generation, all chart pairs are compared and transverse intersections are identified. It follows that once a transverse intersection is identified, then the next/previous generation must also contain an orbit segment corresponding to the same homoclinic. Hence, in addition to gaining a computational speedup, exploiting the fundamental domain property also ensures that all homoclinics identified are distinct. This follows from the existence of the minimum value for $k_s + k_u$ in Proposition \ref{prop:fundamental_domain}. 
	
	Furthermore, this observation yields a method of resolving the ambiguous case in which the Newton-intersection method finds a zero for $G$ but condition 1 from Section \ref{sec:mining_the_atlases} is not satisfied. Specifically, if $G(\hat{s},\hat{t},\hat{\sigma}) \approx 0$ and $\Gamma^u(\hat{s},\hat{t}) \approx \Gamma^s(\hat{\sigma},0) \approx 0$, then we may follow the suspected intersection through earlier/later generation charts until the sign condition can be verified or refuted in appropriate predecessor/successor charts. 
Lastly, we mention that by storing ``parent/child'' information about the charts in the atlas, we can perform the search just 
described in post-processing. 
\end{remark}

\begin{figure}[!t]
\centering
\includegraphics[width=5.5in]{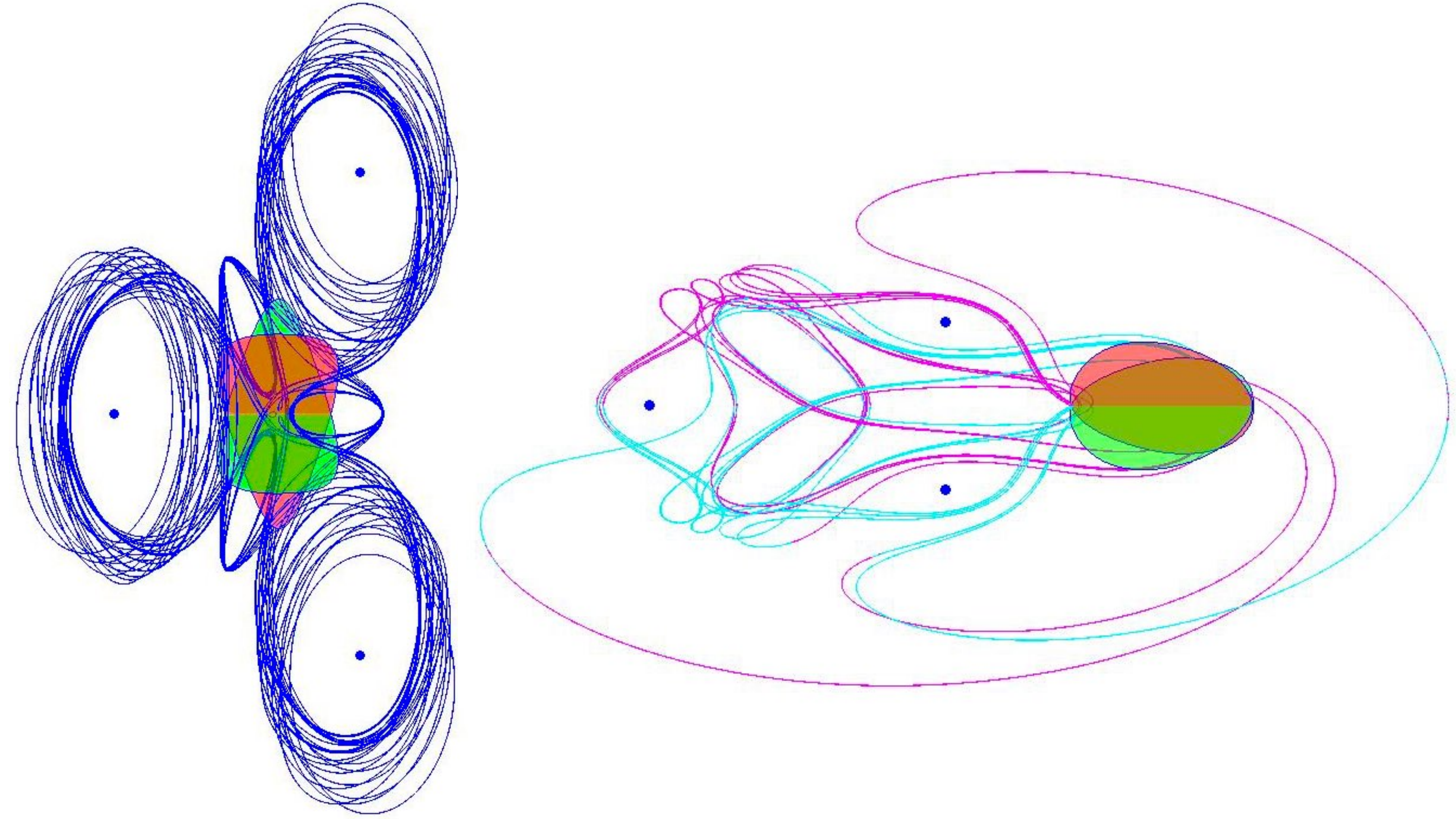}
\caption{\textbf{Asymptotic orbits in configuration space -- top down view of the connections:}
Red and green disks represent the parameterized local stable/unstable manifolds at 
the libration points.  The blue/magenta curves represent the portion of the connecting orbit
off the local invariant manifolds --  that is, the part found by solving the projected boundary 
value problem as discussed in Section \ref{sec:twoViewsOfConnections}.
Left -- 42 shortest homoclinic connecting orbits at $\mathcal{L}_0$ (up to rotational symmetry).
Another $84$ connections are obtained by $\pm 120$ degree rotations. 
The $\pm 120$ rotations are not plotted as they only thicken the blue shaded region.  
Right -- 23 shortest homoclinic connecting orbits at $\mathcal{L}_5$.
Another $46$ connections at $\mathcal{L}_4$ and $\mathcal{L}_6$ are
obtained by $\pm 120$ degree rotations. 
The initial guess for the boundary value solver come from atlases 
obtained by integrating the local unstable/stable manifolds for $\pm T = 5$ time units.
The mining procedure was discussed in Section 5.2.
All reference to color refers to the online version.
}\label{fig:connectionsTopDown}
\end{figure}

\begin{table} 
	\caption{\textbf{Classification of the connecting orbits for $\mathcal{L}_0$:}  
		Advecting one third of the boundary of the local stable/unstable manifolds 
		for $T = \pm 5$ time units and imposing
		a speed threshold of $2$ reveals the 42 homoclinic connections 
		illustrated in the left frame of Figure \ref{fig:connectionsTopDown}.
		The rows of the table give data for the homoclinics, ordered by connection time.
		For each orbit described in the table there are two additional orbits with exactly the same
		connection time, obtained by $\pm 120$ degree rotations.
		Taken with their symmetric counterparts, the orbits given here
		 are are all of the connecting orbits satisfying the connection time and speed constraints. 
		The 3 columns in the table report a connection's order of appearance, its connection time, and our qualitative  
		description of the connection as a word built from the two shortest homoclinics -- 
		the letters shown in Figure \ref{fig:L0_letters}. The connections associated with longer words are 
		separated and illustrated in Figures \ref{fig:connectionsPage1},
		\ref{fig:connectionsPage2},
		\ref{fig:connectionsPage3},
		\ref{fig:connectionsPage4}, and
		\ref{fig:connectionsPage5}.}
	\label{tab:L0connectionClassification}
	\hfill
	\begin{tabular}{lll}
		\hline\noalign{\smallskip}
		Connection & Connection Time &   Letter or Word    \\
		\noalign{\smallskip}\hline\noalign{\smallskip}
		$1$st &  $1.717$  & $L_{0A}$   \\
		$2$nd & 2.331 & $L_{0B}$ \\
		$3$rd  & 4.198 & $L_{0A^+} \cdot L_{0A}$ \\
		$4$th  & 4.520 & $L_{0A^+} \cdot L_{0B}$ \\
		$5$th & 4.715 & $L_{0B}^2$ \\
		$6$th & 5.643 & $L_{0A}^2$ \\
		$7$th & 6.132 & $L_{0A} \cdot L_{0B}$ \\
		$8$th & 6.132 & $L_{0B^-} \cdot L_{0A}$ \\
		$9$th & 6.583 & $L_{0A^-} \cdot L_{0A}$ \\
		$10$th & 6.627 & $L_{0A^+} \cdot L_{0B} \cdot L_{0A}$ \\
		$11$th & 6.628 & $L_{0B^-} \cdot L_{0B}$ \\
		$12$th & 6.684 & $L_{0A^-} \cdot L_{0A^+} \cdot L_{0A}$ \\
		$13$th & 6.760 & $L_{0A^+} \cdot L_{0B}^2$ \\
		$14$th & 6.846 & $L_{0B}^3$ \\
		$15$th & 7.009 & $L_{0B^+} \cdot L_{0A^+} \cdot L_{0A}$ \\
		$16$th & 7.336 & $L_{0B^+} \cdot L_{0A^+} \cdot L_{0B}$ \\
		$17$th & 7.038 & $L_{0B^+} \cdot L_{0A}$ \\
		$18$th & 7.038 & $L_{0A^-} \cdot L_{0B}$ \\
		$19$th & 7.490 & $L_{0B^+} \cdot L_{0B}$ \\
		$20$th & 8.119& $L_{0B} \cdot L_{0A}$ \\
		$21$st & 8.125 & $L_{0A^+}^2 \cdot L_{0A}$ \\
		\noalign{\smallskip}\hline
	\end{tabular}
	\hfill
	\begin{tabular}{lll}
		\hline\noalign{\smallskip}
		Connection & Connection Time &   Letter or Word    \\
		\noalign{\smallskip}\hline\noalign{\smallskip}
		$22$nd & 8.125 & $L_{0A^+} \cdot L_{0A}^2$ \\
		$23$rd & 8.296 & $L_{0A} \cdot L_{0B} \cdot L_{0A}$ \\
		$24$th & 8.448 & $L_{0B} \cdot L_{0A}^2$ \\
		$25$th & 8.453 & $L_{0A} \cdot L_{0B}^2$ \\
		$26$th & 8.453 & $L_{0B^-}^2 \cdot L_{0A}$ \\
		$27$th & 8.499 & $L_{0B}^3 \cdot L_{0A}$ \\
		$28$th & 8.499 & $L_{0A^+} \cdot L_{0B}^3$ \\
		$29$th & 8.583 & $L_{0B}^2$ \\
		$30$th & 8.614 & $L_{0A^+} \cdot L_{0A} \cdot L_{0B}$ \\
		$31$st & 8.732 & $L_{0A^+} \cdot L_{0B}^3$ \\
		$32$nd & 8.794 & $L_{0A} \cdot L_{0B^-} \cdot L_{0B}$ \\
		$33$rd & 8.937 & $L_{0B} \cdot L_{0A} \cdot L_{0B}$ \\
		$34$th & 8.953 & $L_{0B^-} \cdot L_{0B}^2$ \\
		$35$th & 8.953 & $L_{0B^-}^2 \cdot L_{0B}$ \\
		$36$th & 9.065 & $L_{0A} \cdot L_{0A^-} \cdot L_{0A}$ \\
		$37$th & 9.340 & $L_{0A^-} \cdot L_{0B}^2$ \\
		$38$th & 9.88 & $L_{0B^-} \cdot L_{0A^-} \cdot L_{0A}$ \\
		$39$th & 9.444 & $A^- \cdot B^+ \cdot A^+ \cdot B$ \\
		$40$th & 9.495 & $A \cdot A^- \cdot A^+ \cdot B$ \\
		$41$st & 9.579 & $B^+ \cdot A^+ \cdot B^2$ \\
		$42$nd & 9.579 & $(B^+)^2 \cdot A^+ \cdot B$ \\
		\noalign{\smallskip}\hline
	\end{tabular}
	\hfill
\end{table}


\begin{table} 
	\caption{\textbf{Classification of the connecting orbits for $\mathcal{L}_5$:}  
		The 23 homoclinic connections which appear on the right side of Figure \ref{fig:connectionsTopDown} satisfying the same connection time and speed constraints as in the $\mathcal{L}_0$ case. 
	In this case, the 120 degree symmetry does not produce additional connections for $\mathcal{L}_5$ but rather, rotation of each connection produces symmetric homoclinics for both $\mathcal{L}_4$ and $\mathcal{L}_6$. The columns are similar to those in Table \ref{tab:L0connectionClassification} and the longer words associated with $\mathcal{L}_5$ homoclinics illustrated in Figures \ref{fig:connectionsL5_words} and \ref{fig:connectionsL5_long}.} 
	\label{tab:L5connectionClassification}
	\hfill
\begin{tabular}{lll}
\hline\noalign{\smallskip}
Connection & Connection Time &   Letter or Word    \\
\noalign{\smallskip}\hline\noalign{\smallskip}
 $1$st  and $2$nd&  $4.802$  & $L_{5A}$ and $L_{5B}$   \\
 $3$rd and $4$th &  $4.943$  & $L_{5C}$ and $L_{5D}$   \\
 $5$th and $6$th &  $5.261$  & $L_{5E}$ and $L_{5F}$   \\
 $7$th & 6.028 & $L_{5D} \cdot L_{5C}$ \\
 $8$th & 8.204 & $L_{5A} \cdot L_{5B}$ \\
 $9$th & 8.331 & $L_{5A} \cdot L_{5D}$ \\
 $10$th & 8.456 & $L_{5C} \cdot L_{5D}$ \\
 $11$th & 8.917 & $L_{5E} \cdot L_{5F}$ \\
 $12$th & 8.934 & $L_{5A} \cdot L_{5D}$ \\
  $13$th & 9.156 & $L_{5C} \cdot L_{5D}$ \\
 \noalign{\smallskip}\hline
\end{tabular}
\hfill
\quad 
 \begin{tabular}{lll}
\hline\noalign{\smallskip}
Connection & Connection Time &   Letter or Word    \\
\noalign{\smallskip}\hline\noalign{\smallskip}
 $14$th & 9.324 & $L_{5E} \cdot L_{4E} \cdot L_{5C}$ \\
 $15$th & 9.363 & $L_{5A} \cdot L_{5D} \cdot L_{5C}$ \\
 $16$th & 9.387 & $L_{5A} \cdot L_{6A} \cdot L_{5B}$ \\
 $17$th & 9.429 & $L_{5C}^2$ \\
 $18$th & 9.429 & $L_{5D}^2$ \\
 $19$th & 9.487 & $L_{5D} \cdot L_{5C} \cdot L_{5D}$ \\
 $20$th & 9.487 & $L_{5C} \cdot L_{5D} \cdot L_{5C}$ \\
 $21$st & 9.554 & $L_{5A} \cdot L_{6A} \cdot L_{5D}$ \\
  $22$nd & 9.554 & $L_{5C} \cdot L_{6A} \cdot L_{5B}$ \\
 $23$rd & 9.629 & $L_{5C} \cdot L_{6A} \cdot L_{5D}$ \\
\noalign{\smallskip}\hline
\end{tabular}
\hfill\null
\end{table}

\subsection{The symmetric case: locating, refining, and classifying, connections}
We now describe the homoclinic mining procedure  in the case of the triple Copenhagen problem. 
Assuming we have computed stable/unstable atlases denoted by $\mathcal{A}^s,\mathcal{A}^u$ respectively. Each atlas is of the form described in Section \ref{sec:atlas} i.e.~each atlas is a union of chart maps having the form, $\Gamma^{s,u}: D \to \rr^4$ with $\Gamma^{s,u}(D) \subset W^{s,u}(\mathbf{x_0})$. 

We begin with a lemma to motivate the choice to grow each atlas in the symmetric case and then do continuation as opposed to growing the atlas for non-symmetric cases. 

\begin{lemma}
	\label{lem:CRFBP_symmetry}
	Assume $f$ is the symmetric CRFBP vector field i.e.\ $m_1 = m_2 = m_3 = \frac{1}{3}$ and define two linear maps, $\varphi^{\pm}: \rr^4 \to \rr^4$ by $\varphi^{\pm}(x,\dot{x},y,\dot{y}) = \varphi^{\pm}(\mathbf{x}) = R^{\pm}\mathbf{x}$ where $R^{\pm}$ is the matrix given by
\[
R^{\pm} = 
\left(
\begin{array}{cccc}
\cos (\pm \theta) & 0 & - \sin(\pm \theta) & 0 \\
0 & \cos (\pm \theta) & 0 & - \sin(\pm \theta) \\
\sin(\pm \theta) &  0 & \cos(\pm \theta)  & 0 \\
0 & \sin(\pm \theta) &  0 & \cos(\pm \theta)  \\
\end{array}
\right)
\qquad \theta = \frac{2 \pi}{3},
\]
then $\varphi^{\pm}$ is a rotational conjugacy for $f$ and $\varphi^{\pm} \circ f(\mathbf{x}) = f \circ \varphi^{\pm}(\mathbf{x})$ for all $\mathbf{x} \in \rr^4$. In particular, if $\gamma$ parameterizes a homoclinic orbit for $\mathcal{L}_0$, then $\varphi^{\pm} \circ \gamma$ are parameterizations for two additional, distinct ``symmetric'' homoclinic orbits for $\mathcal{L}_0$. Moreover, if $\gamma$ parameterizes a homoclinic orbit for $\mathcal{L}_5$, then $\varphi^{+} \circ \gamma$ and $\varphi^- \circ \gamma$ parameterize symmetric homoclinics for $\mathcal{L}_4$ and $\mathcal{L}_6$. 
\end{lemma}
The proof of Lemma \ref{lem:CRFBP_symmetry}  is included in Appendix \ref{sec:CRFBP_symmetry}. The significance of this symmetry is the fact that global stable/unstable atlases for the triple Copenhagen problem can be separated into 3 distinct equivalence classes where for $\mathbf{x},\mathbf{y} \in W^*({\mathbf{x}_0})$, the equivalence relation $\mathbf{x} \sim \mathbf{y}$ is satisfied if and only if $\mathbf{x} \in \{\mathbf{y}, \varphi^+(\mathbf{y}), \varphi^-(\mathbf{y})\}$. Thus, each atlas is obtained by advection of only a single representative for each class. In other words, in the equal masses case, we only need to advect $\frac{1}{3}$ of each initial parameterization boundary to obtain the entire atlas. Specifically, we define
\[
D' = \{\mathbf{z} \in D : 0\leq \text{Arg}(z_1) < \theta, \ z_2 = \text{conj}(z_1) \}
\]
and we globalize only $\partial D'$ to obtain a partial atlas, $\mathcal{A}'$. We can then access the full global atlas by applying $\varphi^+,\varphi^-$ to each chart in $\mathcal{A}'$ and we set
\[
\mathcal{A} = \mathcal{A}' \cup \varphi^+ \left(\mathcal{A}'\right)  \cup \varphi^- \left(\mathcal{A}'\right).
\]
The advantage is a $9$-fold increase in computational efficiency for the atlas computation, 
and a $3$-fold improvement in efficiency for the atlas mining scheme. 
Applying the procedure for the triple Copenhagen problems results in the 
connecting orbits illustrated in Figure \ref{fig:connectionsTopDown}.
These results are further described and classified in the next section.

\begin{figure}[!t]
\centering
\includegraphics[width=6in]{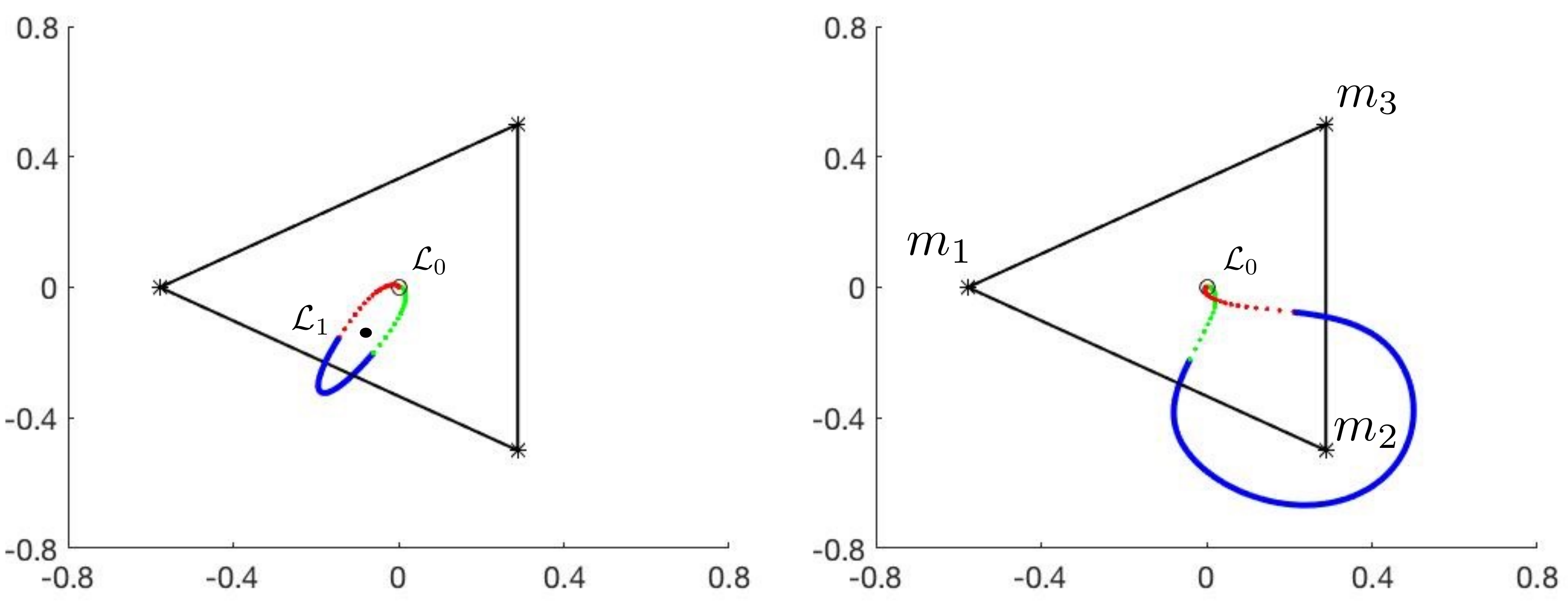}  
\caption{\textbf{Shortest connections at $\mathcal{L}_0$:}  
We distinguish the six shortest connecting orbits at $\mathcal{L}_0$ 
in the triple Copenhagen problem.  Two are shown above. 
The four others are rotations of these by $\pm 120$.
See also the left frames in Figure 20.
The blue portion of the orbit is the segment found by 
solving the boundary value problem.  
The dotted red/green
lines are the asymptotic portions on the parameterized local 
unstable/stable manifolds respectively.  The asymptotic portions 
are recovered using the conjugacy provided by the parameterization method.
We refer to the orbit on the left as $L_{0A}$
and the orbit on the right as $L_{0B}$.  The rotations by $\pm 120$ degrees 
we refer to as $L_{0A^\pm}$ and $L_{0B^\pm}$.
Observe that $L_{0A}$ has winding number $1$ with respect to 
the libration point $\mathcal{L}_1$, while $L_{0A^\pm}$ wind once 
around $\mathcal{L}_{2}$ and $\mathcal{L}_3$ respectively.
$L_{0B}$ on the other hand winds once around $m_2$ while 
$L_{0B^{\pm}}$ wind once around $m_3$ and $m_1$ respectively.
As the next five figures illustrate, the six basic connecting 
orbits organize all the connections we found at $\mathcal{L}_0$.
All references to color refer to the online version.
}
\label{fig:L0_letters}
\end{figure}

\begin{figure}[!h]
\centering
\includegraphics[width=6in]{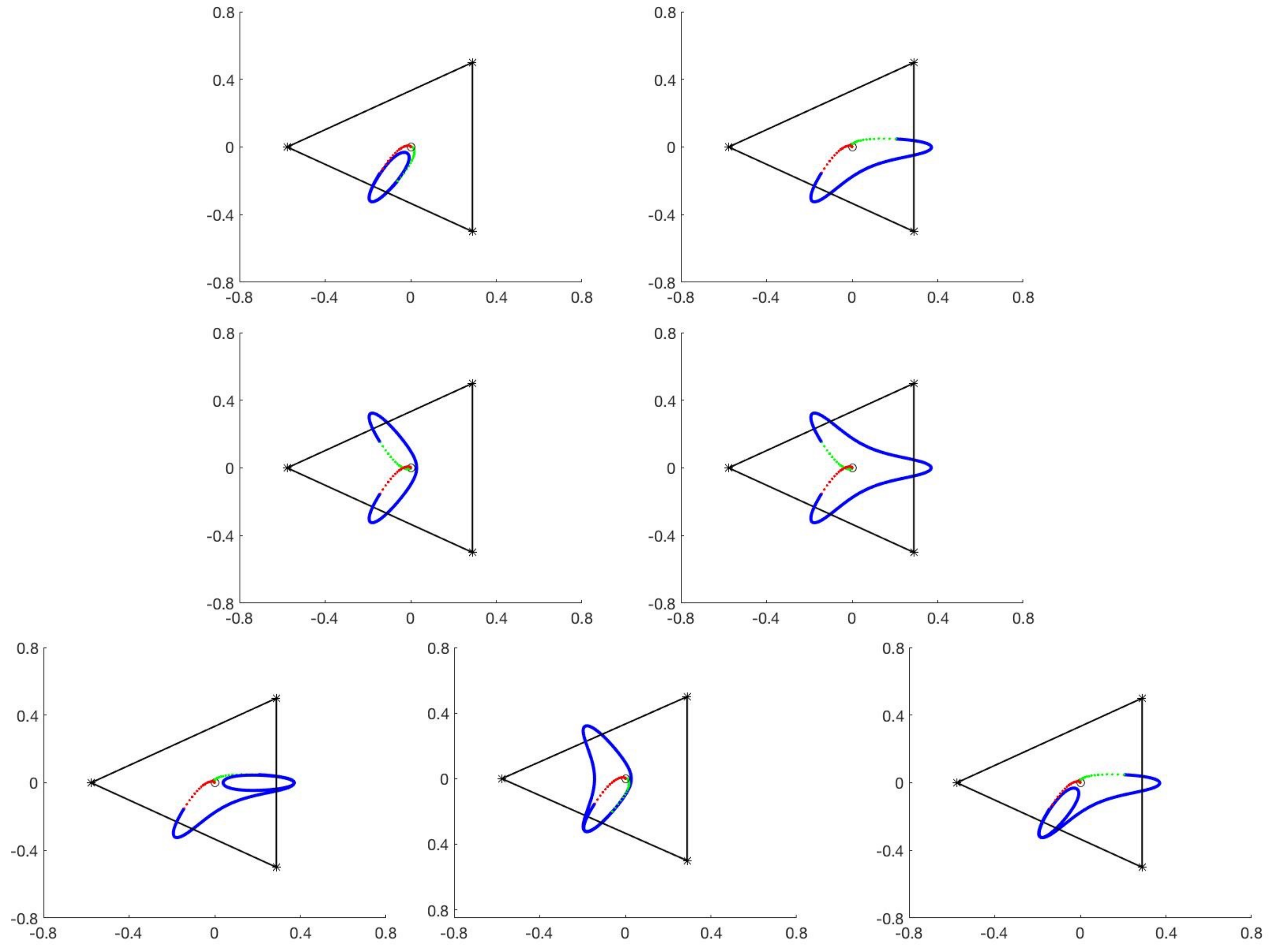}
\caption{\textbf{Homoclinic orbits at $\mathcal{L}_0$ associated with $L_{0A}$:} 
The figure illustrates seven homoclinic connecting orbits all built from 
$L_{0A}$ and its rotations $L_{0A^\pm}$ in the following sense: each of these
orbits has non-trivial winding
number with respect to one or more of the inner libration points, 
but none of these orbits wind around any of the primaries.  That is, these orbits
are ``built from pieces of $L_{0A}$ and its rotations, but have no contribution
from $L_{0B}$ and its rotations. 
(Blue, green and red solid/dotted lines have the same meaning described in 
Figure \ref{fig:L0_letters}). The orbits give the 
following symbol sequences: 
$L_{0A}^2$ -- 6th connection (top left), 
$L_{0A^+} \cdot L_{0A}$ --3rd (top right), 
$L_{0A^-} \cdot L_{0A}$ -- 9th (center left), 
$L_{0A^-} \cdot  L_{0A^+} \cdot  L_{0A}$--12th connection (center right), 
$L_{0A^+}^2 \cdot  L_{0A}$ -- 21st (bottom left), 
$L_{0A} \cdot L_{0A^-} \cdot  L_{0A}$ -- 36th (bottom center), 
 $L_{0A^+} \cdot  L_{0A}^2$ --22nd (bottom right).
In each case rotation of each by $\pm 120$ degrees gives a
distinct connection (not shown).
All references to color refer to the online version.
}
\label{fig:connectionsPage1}
\end{figure}

\subsubsection{Quantitative/qualitative classifications of the homoclinic orbit set
at $\mathcal{L}_{0,5}$}
Suppose $\mathbf{x}_0 \in \mathbb{R}^4$ is an equilibrium solution and
$W^{s,u}_{\mbox{\tiny loc}}(\mathbf{x}_0)$ a local stable/unstable 
manifold.  Let $\gamma$ be an orbit homoclinic to $\mathbf{x}_0$, and suppose 
that $T \in \mathbb{R}$ is the elapsed 
time from when $\gamma$ passes through the boundary of the local unstable 
manifold to when $\gamma$ passes through the boundary of the local stable 
manifold.  Observe that if 
$W^{s}_{\mbox{\tiny loc}}(\mathbf{x}_0) \cap W^{u}_{\mbox{\tiny loc}}(\mathbf{x}_0) = \{\mathbf{x}_0\}$  
and if the vector field is inflowing/outflowing 
on the boundaries of $W^{s,u}_{\mbox{\tiny loc}}(\mathbf{x}_0)$ respectively, then
$T > 0$ is well defined.

When the local parameterizations intersect only at $\mathbf{x}_0$, it makes
sense to talk about the ``shortest'' connection time,''
the ``second shortest'' connection time, and so on.  
This natural ordering on connection times provides a useful observable for classifying
homoclinic connections relative to fixed local stable/unstable manifolds.  
Generically, we expect a one-to-one correspondence between 
connecting orbits and connection times, 
though this expectation will fail in the presence of symmetries
as seen below.

In the CRFBP, when we ``mine'' the stable/unstable atlases for connecting orbits and order
them by connection time we see something interesting.
In each of the cases we studied, the shortest homoclinic orbits appear to 
organize the longer connections.  Informally speaking, we find that
a small number of short homoclinic orbits serve as a sort of  alphabet of ``letters'', and the longer connections can be roughly identified
as ``words'' in this alphabet.

For example, the first 42 homoclinic connecting orbits (up to symmetry) at $\mathcal{L}_0$ 
in the triple Copenhagen problem are classified in Table \ref{tab:L0connectionClassification}. 
These results are obtained by integrating initial local stable/unstable manifolds 
for $\pm 5$ time units subject to the speed constraint, $\dot{x}^2 + \dot{y}^2 \leq 4$.  Our method finds all of the 
connections satisfying these constraints.  The classification is in terms of the 
connection time, the order of appearance, and a geometric description in terms of 
words and letters. 

We give the names
$L_{0A}$ and $L_{0B}$ to the shortest two connections at $\mathcal{L}_0$.  
These orbits are illustrated in Figure \ref{fig:L0_letters},
and have connection times approximately 1.717 and 2.331 respectively.  
Rotating either of these by $\pm 120$ degrees gives another 
connecting orbit with exactly the same shape and connection time.
We refer to these rotations as $L_{0A^{\pm}}$ and $L_{0B^{\pm}}$.
These six shortest connections --$L_{0A}, L_{0B}$ and their 
symmetric counterparts -- organize the rest of the homoclinic behavior seen at 
$\mathcal{L}_0$ as we now describe.

We associate the third shortest connection  
with the word $L_{0A^+} \cdot L_{0A}$ because the 
orbit moves off the unstable manifold appearing to follow 
$L_{0A^+}$, passes near the equilibrium at $\mathcal{L}_0$,
and makes another excursion following $L_{0A}$ before returning 
to the stable manifold.  Similarly we associate to the $5$th longest connecting orbit 
the word $L_{0B}^2$, as this orbit moves off the 
unstable manifold and appears to follow $L_{0B}$, making two loops 
around the second primary before returning to the stable manifold.
Heuristically speaking, $L_{0A}$, $L_{0B}$ and their symmetric counterparts  
comprise a system of ``homoclinic channels'' or simple allowable 
motions and other homoclinic orbits seem to follow in their wake.

Table \ref{tab:L5connectionClassification} records analogous
information for the first 23 connections found at $\mathcal{L}_5$ in the triple 
Copenhagen problem.  In this case there are six basic letters $L_{5A}, L_{5B},
 L_{5D},  L_{5D},  L_{5E},  L_{5F}$.  Words are formed for these letters just as 
 discussed above.   Applying $\pm 120$ degree rotations
 produces connections with the same shapes and connection times at $\mathcal{L}_4$
 and $\mathcal{L}_6$ respectively.  
We stress that this description of the connecting orbits
in terms of words and letters, while intuitively appealing, is based on  
qualitative observations and is subordinate to 
the rigid quantitative classification of the orbits by connection time.

\begin{remark}
Several comments about the results reported in 
Tables \ref{tab:L0connectionClassification} and \ref{tab:L5connectionClassification} are in order.
\begin{itemize}
\item \textbf{Additional symmetries:} Some of the orbits, for example the 27th and 28th shortest orbits 
at $\mathcal{L}_0$ and the 21st and 22nd shortest orbits at $\mathcal{L}_5$ have reported the same connection times to 
three decimal places.  In fact the connection times agree to within numerical errors.  This is because the 
equal mass problem has reversible symmetries that we are not exploiting in our computations. 
Rather these serve as a check on the numerics.
\item \textbf{Connection time versus ordering:}
While the connection times reported in these tables depend on the 
choice of local stable/unstable manifold,   
it  should be remarked that, as long as the parameterization method is used to represent the local 
manifolds as discussed in Section \ref{sec:invariantManifolds}, the ordering of the connections 
does not change.  It is easy to check that the boundary of the parameterized manifolds are
inflowing/outflowing and that the manifolds intersect only at the equilibrium solution.  
Moreover, since the eigenvalues are complex conjugate, the local parameterizations
are unique up to the choice of a single eigenvector scale factor.  
By choosing the unit disk as the domain of the parameterization, the scaling the only 
free parameter in the problem.  Decreasing the scaling by a factor of $\tau > 0$ 
is equivalent to flowing the boundary by the same time $\tau$.
So: rescaling the eigenvectors changes all the times of flight by exactly the same amount, 
hence does not reorder them.  
\item \textbf{Qualitative classification:} the decomposition of the connecting orbits into words
is performed ``by eye'' in the present work.  That is, we simply inspect the connections and 
describe what we see.   We now sketch a method which could be used to formalize our 
qualitative description, and note that the idea is computationally feasible. 
Recall that if $\gamma$ is a 
simple closed rectifiable curve in the plane (with counter clockwise orientation), 
and $z_0 = x_0  + i y_0$ is a point not on $\gamma$, then the 
number of times that $\gamma$ winds around $z_0$ is counted by the integral
\[
\mbox{Ind}_\gamma(z_0) = \frac{1}{2 \pi i} \int_\gamma \frac{1}{z - z_0} \, dz.
\] 

So, observe for example that the curve $L_{0A}$ winds once around 
the libration point $\mathcal{L}_1$, while $L_{0A^+}$ and $L_{0A^-}$ 
each wind once, respectively around $\mathcal{L}_2$ and $\mathcal{L}_3$.  
Similarly the curve $L_{0B}$ winds once around the second primary, while 
$L_{0B^+}$ and $L_{0B^-}$ wind once each around the third and first 
primaries.  
Then for a homoclinic orbit $\gamma$ at $\mathcal{L}_0$ define the integer vector
$(\alpha_1, \alpha_2, \alpha_3, \beta_1, \beta_2, \beta_3) \in \mathbb{Z}^6$, where 
$\alpha_j = \mbox{Ind}_{\gamma}(\mathcal{L}_j)$ for and 
$\beta_j = \mbox{Ind}_{\gamma}(P_j)$ both for $j = 1,2,3$.  (Here $P_j$ are 
the coordinates of the $j$th primary). 
Then $\alpha_1$ counts the contribution of $L_{0A}$ to $\gamma$ while 
$\beta_1$ counts the contribution of $L_{0B}$ and so on. 
This description amounts to an Abelianization 
of the previous notion of words/letters, as the winding vector looses track of the order 
of the letters in the word.   (It is often the case that mechanical calculation of 
topological data is facilitated by passing to an Abeleanization).  
This notion is extended to the homoclinic orbits at $\mathcal{L}_{5}$ in a similar way, 
see Figure \ref{fig:connectionsL5_letters}.
\item \textbf{Blue skies:} the main theorem of Henrard in \cite{MR0365628}, already mentioned in the introduction,
gives that there is a family of periodic orbits accumulating to each of the homoclinic orbits 
found by our procedure.  In some cases we can venture a guess as which families of periodic orbits
accumulate to which homoclinic.  For example Figure \ref{fig:blueSky} illustrates
the orbit $L_{0A}$ and $L_{0A\pm}$ along with the planar Lyapunov families attached to 
the inner libration points $\mathcal{L}_{1,2,3}$.  The results suggest that the planar
families may accumulate at to these homoclinics. 
Similarly, comparing the orbits $L_{5E}$ and $L_{5F}$ in the bottom 
left and right frames of Figure \ref{fig:connectionsL5_letters} with the 
planar Lyapunov families at $\mathcal{L}_{7,9}$ illustrated in the left 
frame of Figure \ref{fig:localManifolds} suggests that these may be 
the families of periodic orbits attached to these homoclinics.
The orbits $L_{0B}$, and $L_{0B^\pm}$, as well as the orbits $L_{5A}$
and $L_{5B}$ must be the limits of families of periodic orbits 
winding around the primary bodies. 
Making a systematic study of the periodic families associated with 
the homoclinic orbits discussed here would make a nice topic for 
a future study.
\end{itemize}
\end{remark}

\begin{figure}[!t]
\centering
\includegraphics[width=6in]{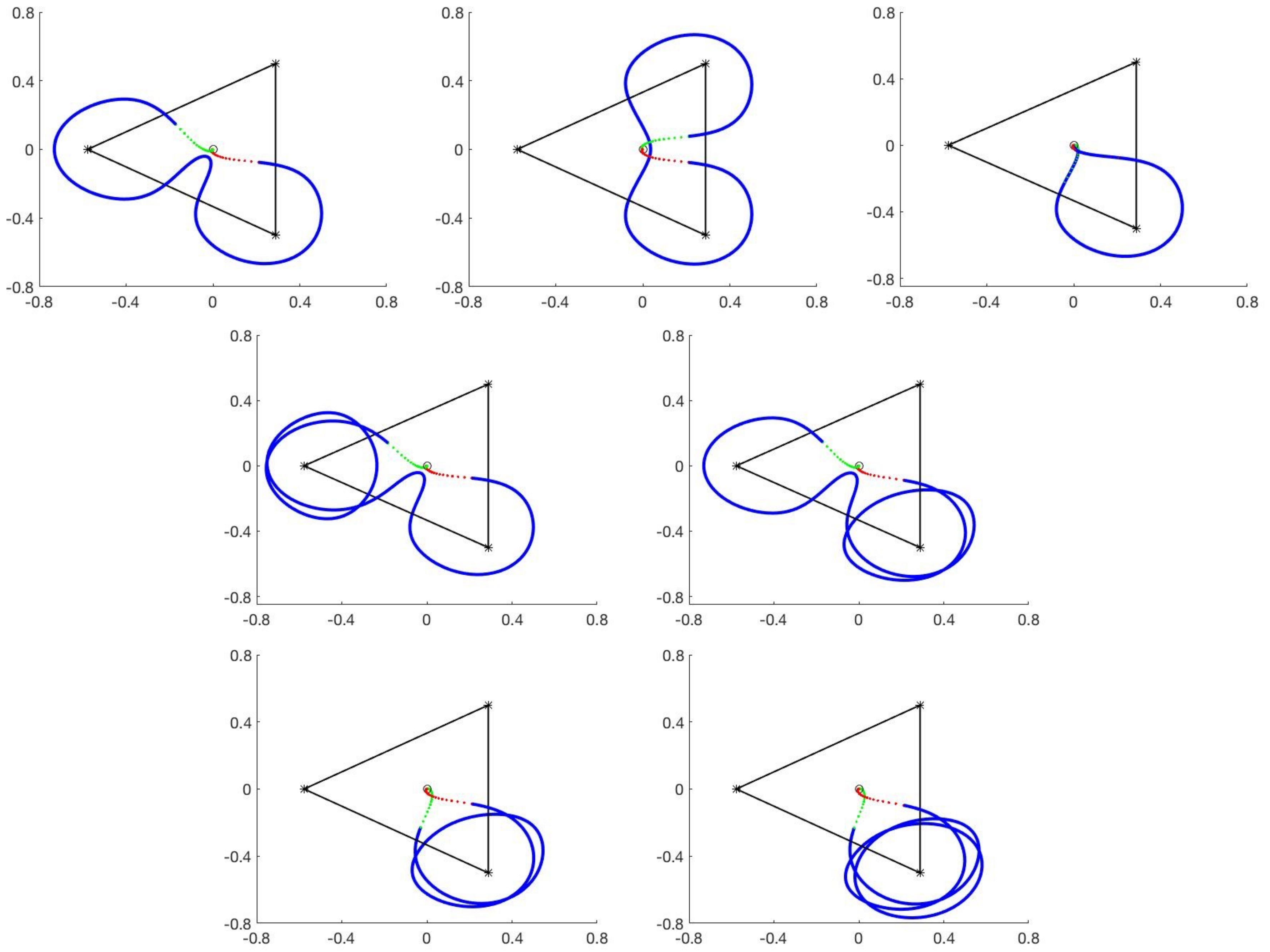}
\caption{\textbf{Homoclinic orbits at $\mathcal{L}_0$ associated with $L_{0B}$}: 
The figure illustrates seven connections built from $L_{0B}$ and its rotations.
So, each of these orbits has non-zero winding about one or more primary 
-- that is $L_{0B}$ behavior -- and none of them have any winding about 
the inner libration points -- that is no $L_{0A}$ behavior. 
(Blue, green and red solid/dotted lines have the same meaning described in 
Figure \ref{fig:L0_letters}). The orbits give the 
following symbol sequences: 
$L_{0B^-} \cdot L_{0B} $ -- 11th connection (top left), 
$L_{0B^+} \cdot L_{0B}$ -- 19th (top center),
$L_{0B}^2$ --29th (top right), 
$L_{0B^-}^2 \cdot L_{0B}$ -- 35th (center left),
$L_{0B^-} \cdot L_{0B}^2$ -- 34th connection (center right),
$L_{0B}^2$ --5th connection (bottom left),
$L_{0B}^3$ --14th (bottom right).
In each case rotation of each by $\pm 120$ degrees gives a
distinct connection (not shown).
Observe that the symbol sequences need not be unique.
For example the top right and bottom left orbits
have the same winding, but different times of flight.
(See also Figure \ref{fig:connectionsDoubling}
and Table \ref{tab:L0connectionClassification}).
All references to color refer to the online version.
}
\label{fig:connectionsPage2}
\end{figure}

\begin{figure}[!t]
\centering
\includegraphics[width=6in]{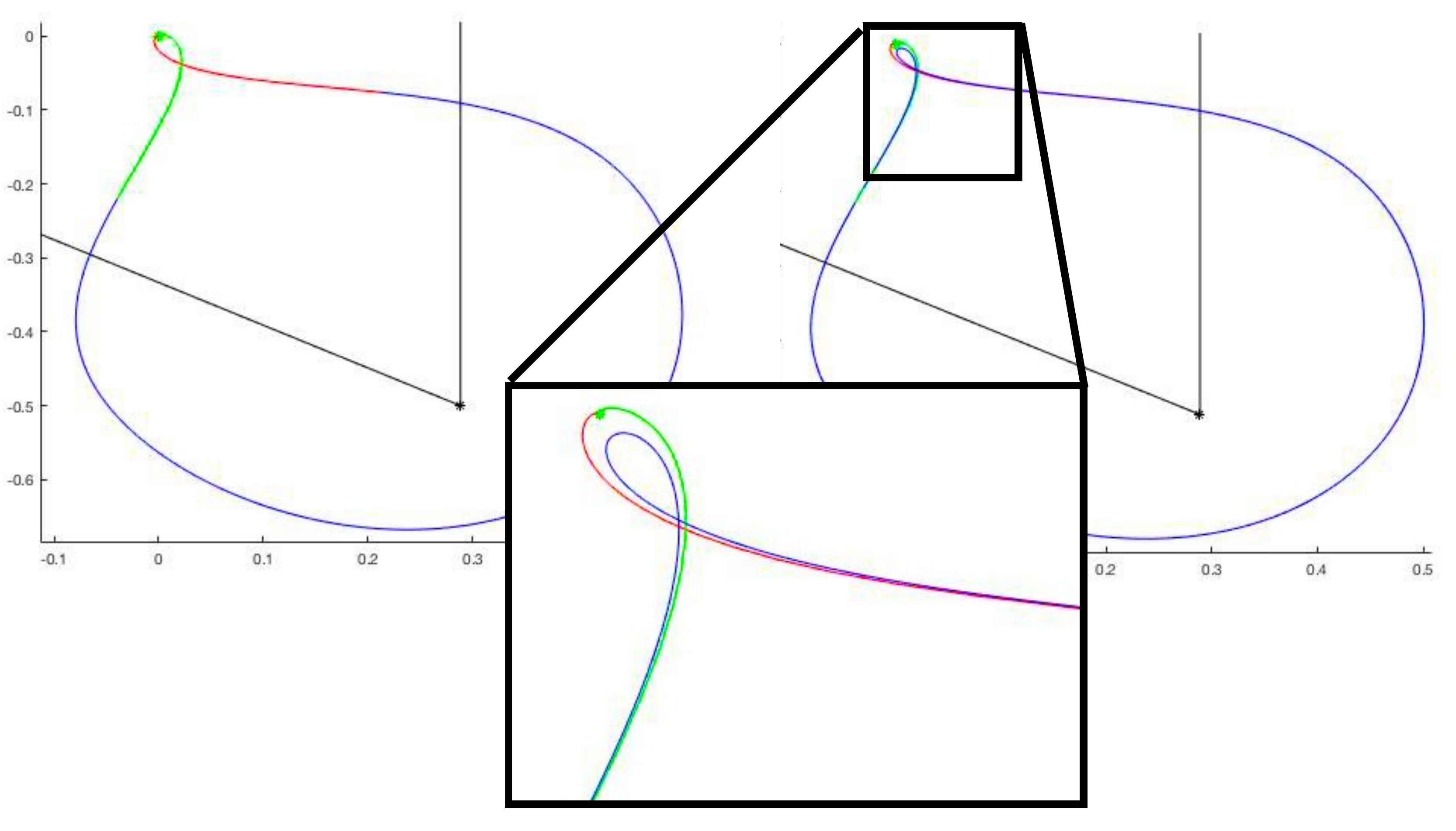}
\caption{\textbf{A close return to $\mathcal{L}_0$:}
While the top right frame in Figure \ref{fig:connectionsPage2}
-- illustrating the $L_{0B}^2$ orbit --
looks like a repeat of $L_{0B}$, further examination 
reveals that the orbits are distinct.  
The $L_{0B}$ and $L_{0B}^2$ orbits
are here shown side-by-side.
The longer $L_{0B}^2$ orbit on the right
follows  $L_{0B}$, makes are ``flyby'' of the libration point, 
then a second excursion, finally returning to the stable
manifold.  A close up of a neighborhood of $\mathcal{L}_{0}$ -- shown in 
the inlay -- illustrates the flyby.  These orbits illustrate the fact that 
two very similar looking orbits can have very different connection times.
Indeed, very simple looking orbits can have long connection times if they 
spend a long time in the neighborhood of a libration point where dynamics move
slow.
}\label{fig:connectionsDoubling}
\end{figure}

\begin{figure}[!t]
\centering
\includegraphics[width=6in]{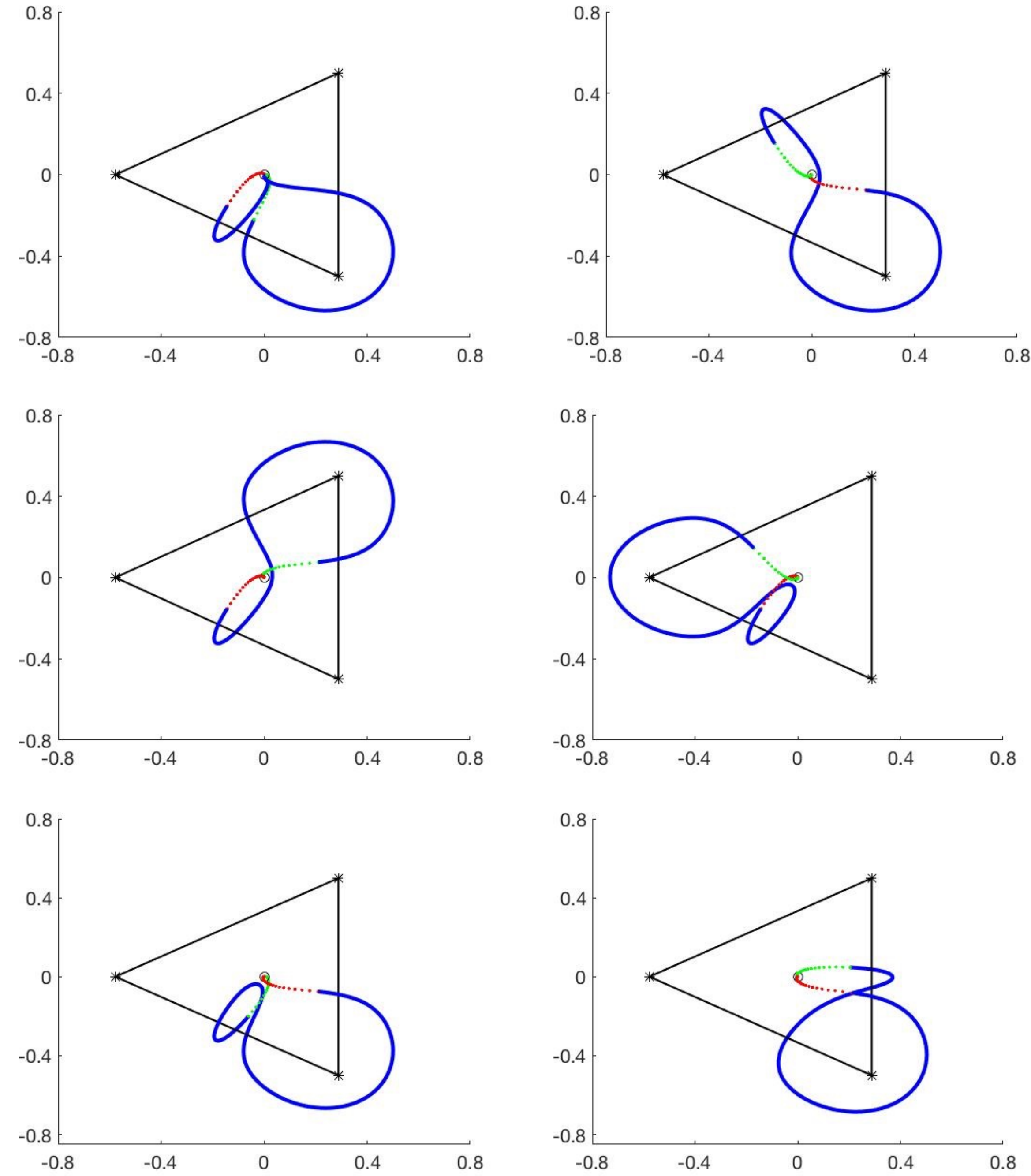}
\caption{\textbf{Homoclinic orbits at $\mathcal{L}_0$ with $AB$ words: } six connections
built from exactly one occurrence of $L_{0A}$ (or a symmetric counterpart) and 
one occurrence of $L_{0B}$ (or one of its symmetric counter parts).
So, each of these orbits has winding number one about exactly one of the 
primaries -- giving one instance of $L_{0B}$ behavior -- and winding number one 
one about exactly one of the inner libration points -- giving one instance of $L_{0A}$ 
behavior.  We refer to these as two letter words.
(Blue, green and red solid/dotted lines have the same meaning described in 
Figure \ref{fig:L0_letters}). We see the 
following symbol sequences: 
$L_{0B} \cdot L_{0A}$ --20th connection (top left), 
$L_{0A^-} \cdot L_{0B}$ -- 18th (top right), 
$L_{0B^+} \cdot L_{0A}$  -- 17th (center left),
$L_{0B^-} \cdot L_{0A}$ -- 8th (center right),
$L_{0A} \cdot L_{0B} $ -- 7th  (bottom left),
$L_{0A^+} \cdot L_{0B}$ -- 4th  (bottom right).
In each case rotation of each by $\pm 120$ degrees gives a
distinct connection (not shown).
All references to color refer to the online version.
}\label{fig:connectionsPage3}
\end{figure}

\begin{figure}[!t]
\centering
\includegraphics[width=6in]{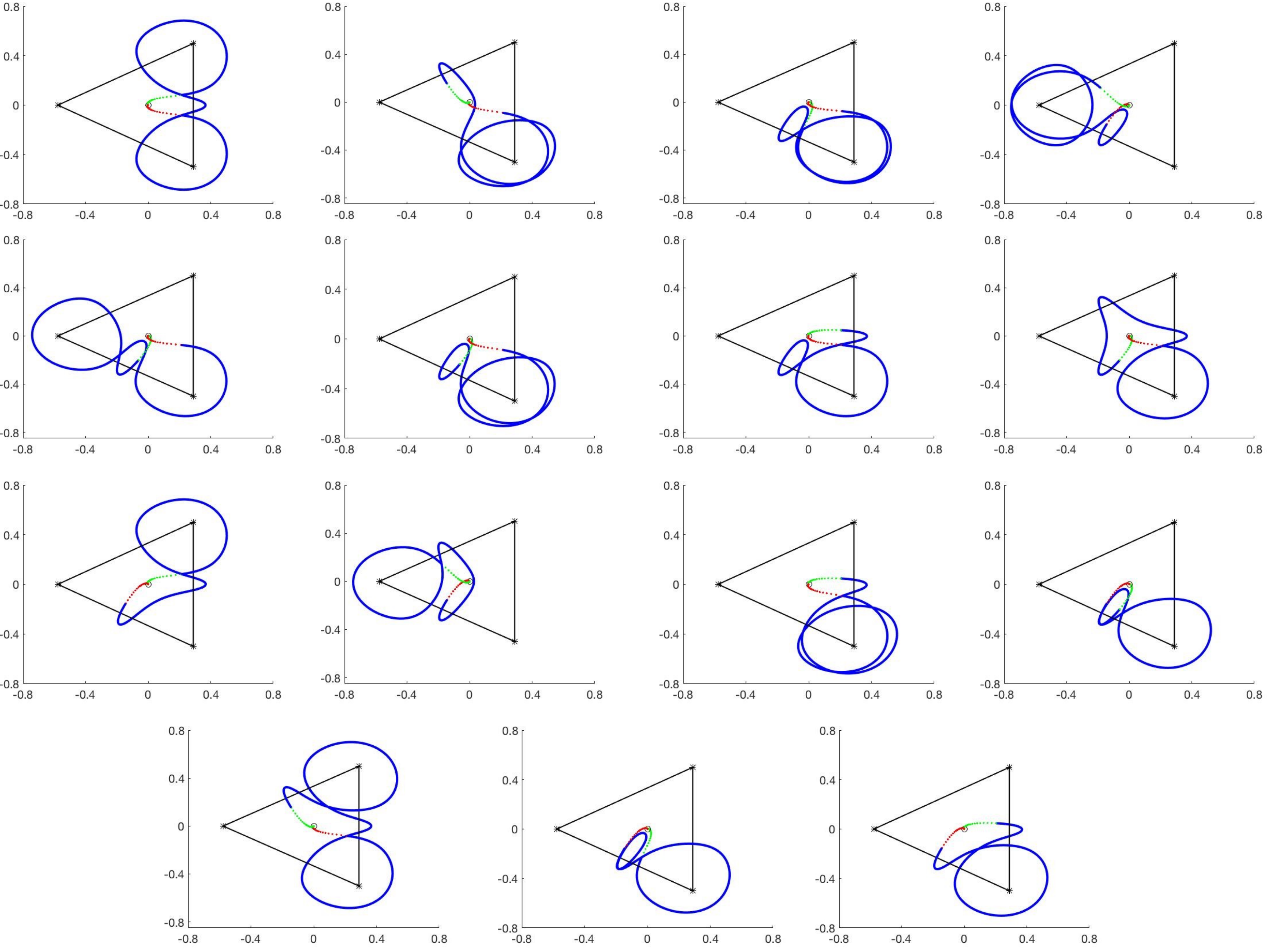}
\caption{\textbf{Fifteen connections at $\mathcal{L}_0$
whose words have three or more symbols:}
orbits with multiple winding about one or more primaries as well as
one or more inner libration point, and hence multiple $L_{0A}$ and $L_{0B}$
behaviors.  We refer to these as three or four letter 
words.  (Blue, green and red solid/dotted lines have the same meaning described in 
Figure \ref{fig:L0_letters}). We see the 
following symbol sequences: 
$L_{0B^+} \cdot L_{0A^+} \cdot L_{0B}$ --16th connection (first row, first column),
$L_{0A^-} \cdot L_{0B}^2$ -- 37th (first row, second column),
$L_{0B} \cdot L_{0A} \cdot L_{0B}$-- 33rd (first row, third column),
$L_{0B^-}^2 \cdot L_{0A}$ -- 26th (first row, fourth column),
$L_{0A} \cdot L_{0B^-} \cdot L_{0B}$ --32nd (second row, first column),
$L_{0A} \cdot L_{0B}^2$ --25th (second row, second column),
$L_{0A^+} \cdot L_{0A} \cdot L_{0B}$ --30th (second row, third column),
$L_{0A} \cdot L_{0A^-} \cdot L_{0A^+} \cdot L_{0B}$ -- 40th (second row, fourth column),
$L_{0B^+} \cdot L_{0A^+} \cdot L_{0A}$ -- 15th (third row, first column),
$L_{0B^-} \cdot L_{0A^-} \cdot L_{0A}$ -- 38th (third row, second column),
$L_{0A^+} \cdot L_{0B}^2$ --13th (third row, third column),
$L_{0A} \cdot L_{0B} \cdot L_{0A} \cdot$ --23rd (third row, fourth column),
$L_{0A^-} \cdot L_{0B}^+  \cdot L_{0A}^+ \cdot L_{0B}$ --39th (fourth row, first column),
$L_{0B} \cdot L_{0A}^2$ -- 24th (fourth row, second column),
$L_{0A}^+ \cdot L_{0B} \cdot L_{0A}$-- 10th (fourth row, third column),
In each case rotation of each by $\pm 120$ degrees gives a
distinct connection (not shown).
}\label{fig:connectionsPage4}
\end{figure}

\begin{figure}[!t]
\centering
\includegraphics[width=6in]{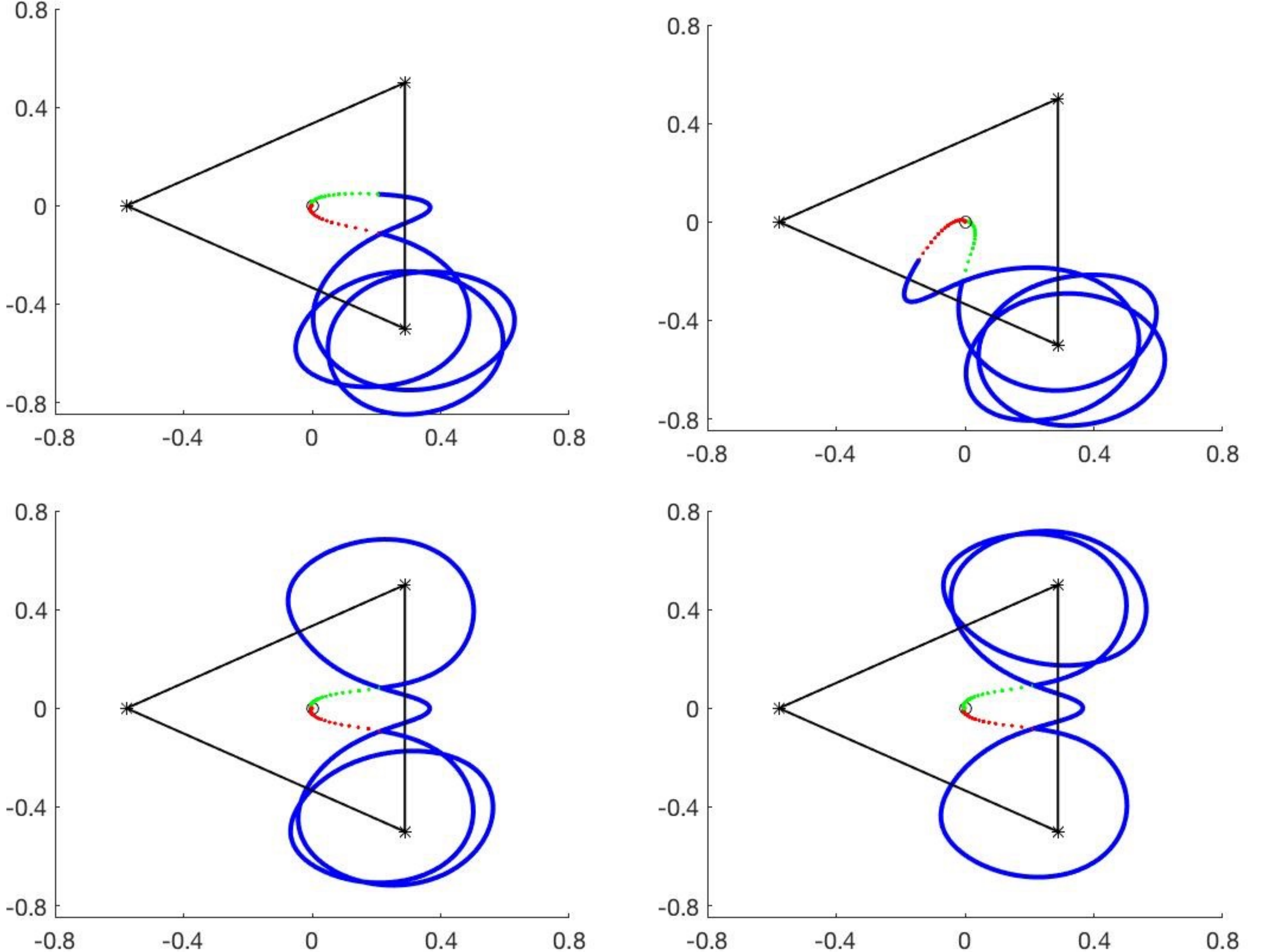}
\caption{\textbf{Longest words the $\mathcal{L}_0$ atlas data:} 
the four most complicated words
in our $\mathcal{L}_0$ search --
orbits whose words contain four letters.
Each has triple winding about the primaries -- either 
winding three about one primary mass or winding two about one followed by 
winding one about another.  Each also has winding one about an inner libration 
point.  It is interesting to note that these are not necessarily the 
longest orbits, in the sense of connection time.  While the bottom left and 
right are the two longest orbits, the top left and right are only 27th and 28th 
respectively.  (Blue, green and red solid/dotted lines have the same meaning described in 
Figure \ref{fig:L0_letters}). We see the 
following symbol sequences: 
$L_{0A}^+ \cdot L_{0B}^3$ --28th connection (top left),
$L_{0B}^3 \cdot L_{0A}$ -- 27th (top right),
$L_{0B}^+ \cdot L_{0A^+} \cdot L_{0B}^2$ -- 41st (bottom left),
$(L_{0B}^+)^2 \cdot L_{0A} \cdot L_{0B}$ 42nd (bottom right).
In each case rotation of each by $\pm 120$ degrees gives a
distinct connection (not shown).
}\label{fig:connectionsPage5}
\end{figure}

\begin{figure}[!h]
\centering
\includegraphics[width=6in]{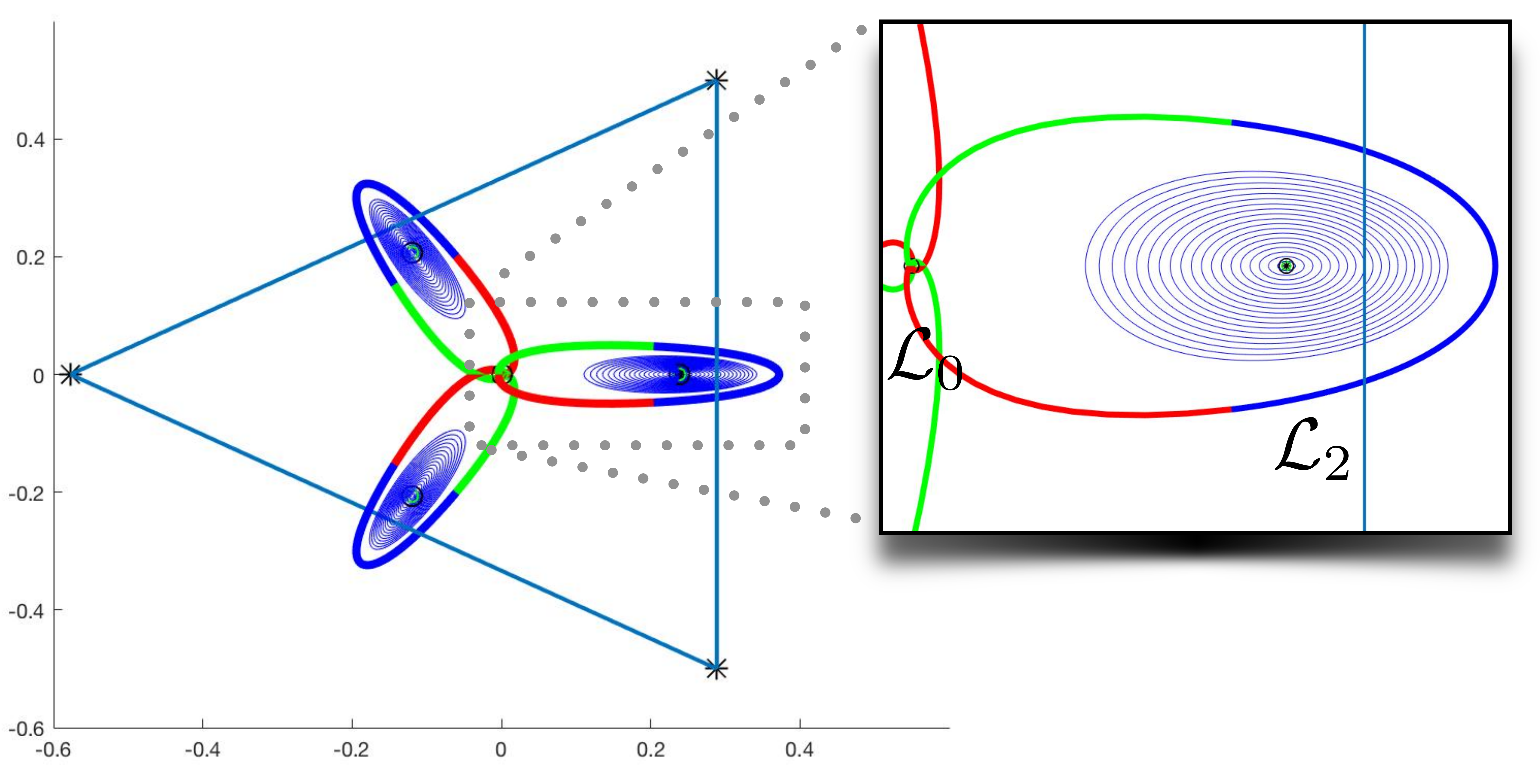}
\caption{\textbf{Blue sky catastrophes at $\mathcal{L}_0$:}
the thick blue lines correspond to the portion of 
the homoclinic orbit computed using the BVP approach. 
The green and red lines correspond to the 
asymptotic portion of the orbit on the parameterized 
local unstable and stable manifolds respectively.
Left: the three shortest homoclinic orbits at $\mathcal{L}_0$.
Note that the orbits are rotations by $120$ degrees of one another.
The figure also includes the planar Lyapunov families of 
orbits about the inner libration points $\mathcal{L}_{1,2,3}$.
These periodic orbits were computed using a center 
manifold reduction, and we have not applied numerical 
continuation to the boundary.  Nevertheless the images
suggest that the planar Lyapunov families may accumulate 
that the $L_{0A}$ homoclinic orbits.
Right: close up on a neighborhood of $\mathcal{L}_0$ and 
$\mathcal{L}_2$.  The homoclinic orbit $L_{0A^+}$ clearly 
has winding number one about $\mathcal{L}_2$.
All reference to color refers to the online version.
}\label{fig:blueSky}
\end{figure}

\begin{figure}[!t]
\centering
\includegraphics[width=5.5in]{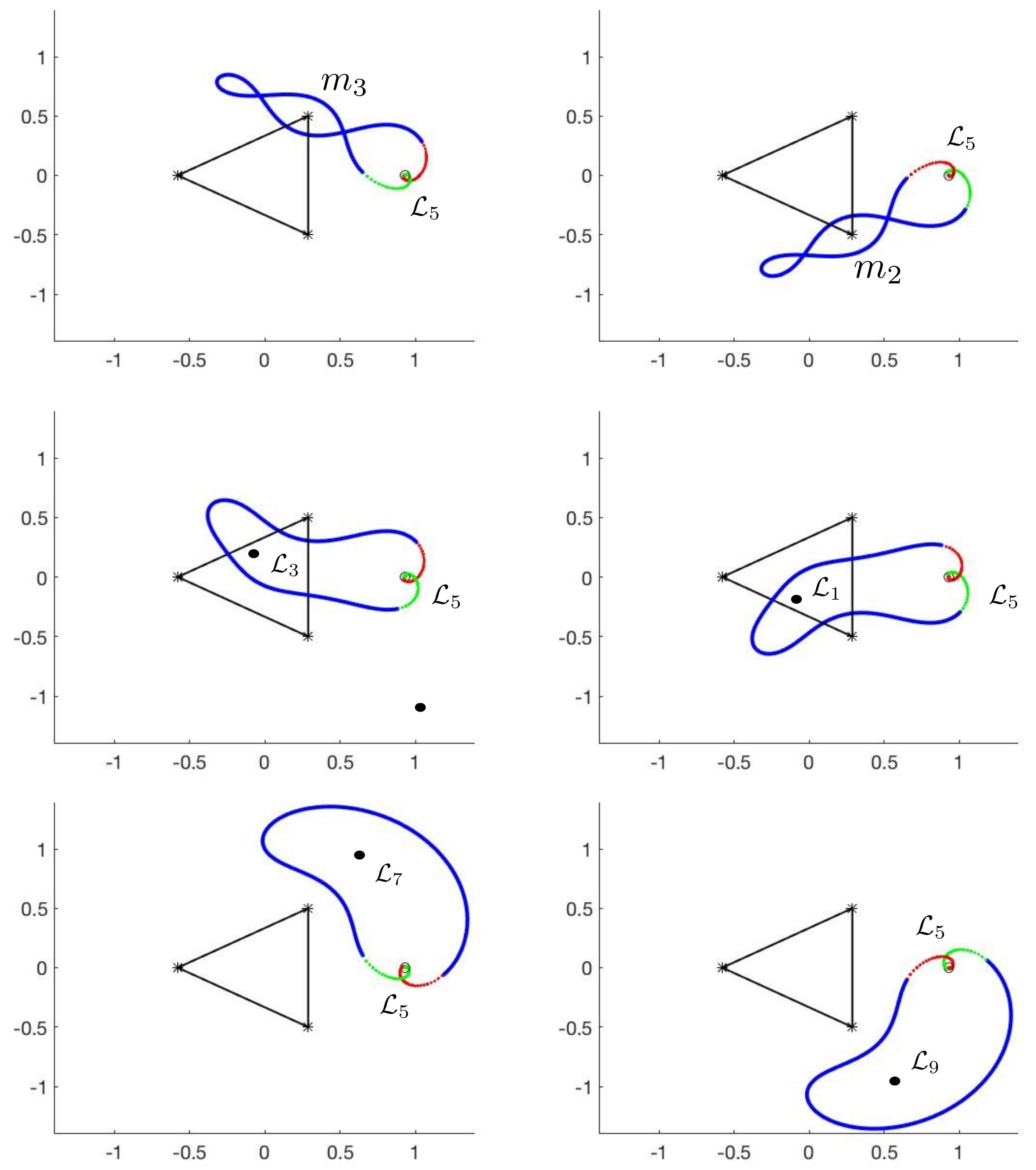}
\caption{
\textbf{Fundamental connections at $\mathcal{L}_5$:}  The six 
shortest connecting orbits at $\mathcal{L}_5$ for the triple Copenhagen
problem.  Shortest connections at $L_{4,6}$ are obtained by 
rotation by $\pm 120$.  These six orbits and their rotations 
organize all the connections we find at $\mathcal{L}_5$,
as illustrated in the next two figures.
(Blue, green and red solid/dotted lines have the same meaning described in 
Figure \ref{fig:L0_letters})
We refer to these orbits as 
$L_{5 A}$, $L_{5B}$ (top left and right),
$L_{5 C}$, $L_{5D}$ (middle left and right),
and $L_{5E}$, $L_{5F}$ (bottom left and right).
Moreover, by rotating the pictures by $\pm 120$ degrees we
obtain orbits which we refer to as $L_{4 ABCDEF}$ and 
$L_{6ABCDEF}$.  Observe that $L_{5A}$ and $L_{5B}$
have winding number 1 with respect to the second and third
primaries.  Moreover they are the only orbits of the six basic words
which wind around any primary.  Similarly $L_{5C}$ and $L_{5D}$
wind once around the inner libration points
$\mathcal{L}_3$ and $\mathcal{L}_1$ once respectively.  
They are the only basic words at $\mathcal{L}_5$ with this property.
Finally, $L_{5E}$ and $L_{5F}$ are distinguished by the fact that they have
winding number one with respect to the outer libration points $\mathcal{L}_7$
and $\mathcal{L}_{9}$ respectively.  
All references to color refer to the online version.  
}\label{fig:connectionsL5_letters}
\end{figure}

\begin{figure}[!t]
\centering
\includegraphics[width=6in]{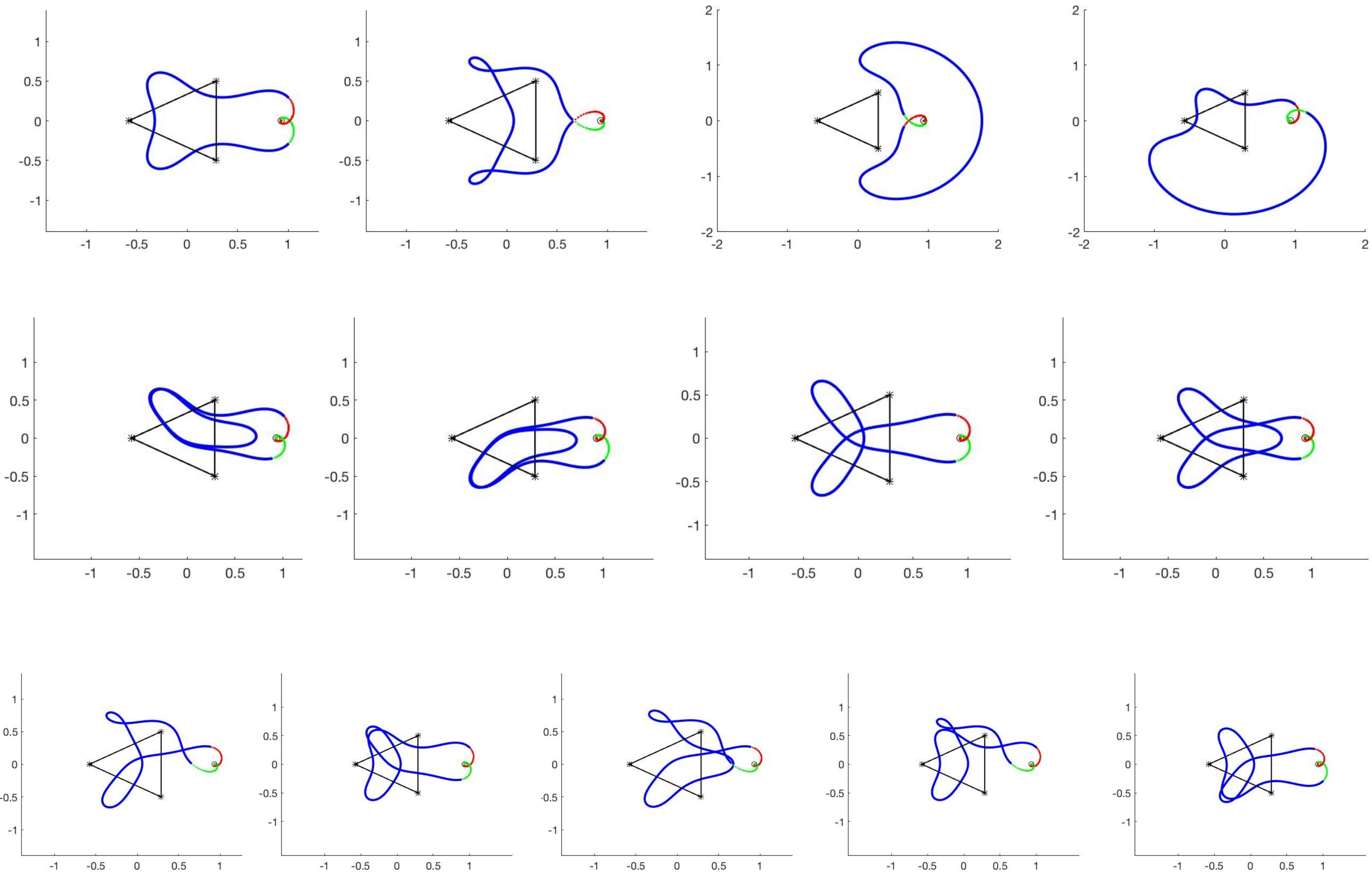}
\caption{\textbf{Homoclinic orbits with two or three letter words at $\mathcal{L}_5$}: 
each of the orbits shadows two or three of the six basic shapes shown in 
Figure \ref{fig:connectionsL5_letters}
(or one of their rotations).  Rotation of each by $\pm 120$ degrees give a 
connection at $L_{4,6}$. 
The symbol sequences for these orbits are
$L_{5D} \cdot L_{5C} $ -- 7th connection (first row, first column),
$L_{5A} \cdot L_{5B}$ -- 8th (first row, second column), 
$L_{5E} \cdot L_{5F} $ -- 11th (first row, third column),
$L_{5F} \cdot L_{4E} \cdot L_{5C} $ --14th (first row, fourth column),
$L_{5C}^2 $ -- 17th (second row, first column),
$L_{5D}^2$ -- 18th (second row, second column), 
$L_{5C}  \cdot L_{5D} $ -- 10th  (second row, third column),
$L_{5C} \cdot L_{5D} $ -- 13th (second row, fourth column),
$L_{5A} \cdot L_{5D} $ --9th (third row, first column),
$L_{5C} \cdot L_{5D} \cdot L_{5C} $ --20th (third row, second column), 
$L_{5A} \cdot L_{5D}$ -- 12th (third row, third column),
$L_{5A} \cdot L_{5D} \cdot L_{5C}$ -- 15th (third row, fourth column),
$L_{5D} \cdot L_{5C} \cdot L_{5D}$-- 19th (third row, fifth column).
(Blue, green and red solid/dotted lines have the same meaning described in 
Figure \ref{fig:L0_letters}).
Again we see that a symbol sequence can appear more than once.
For example 
the frames in the second row third and fourth column
both have $L_{5C}  \cdot L_{5D}$.  However, the orbits are
distinguished by their connection time (See Table \ref{tab:L5connectionClassification}).
All references to color refer to the online version.
}
\label{fig:connectionsL5_words}
\end{figure}

\begin{figure}[!t]
\centering
\includegraphics[width=6in]{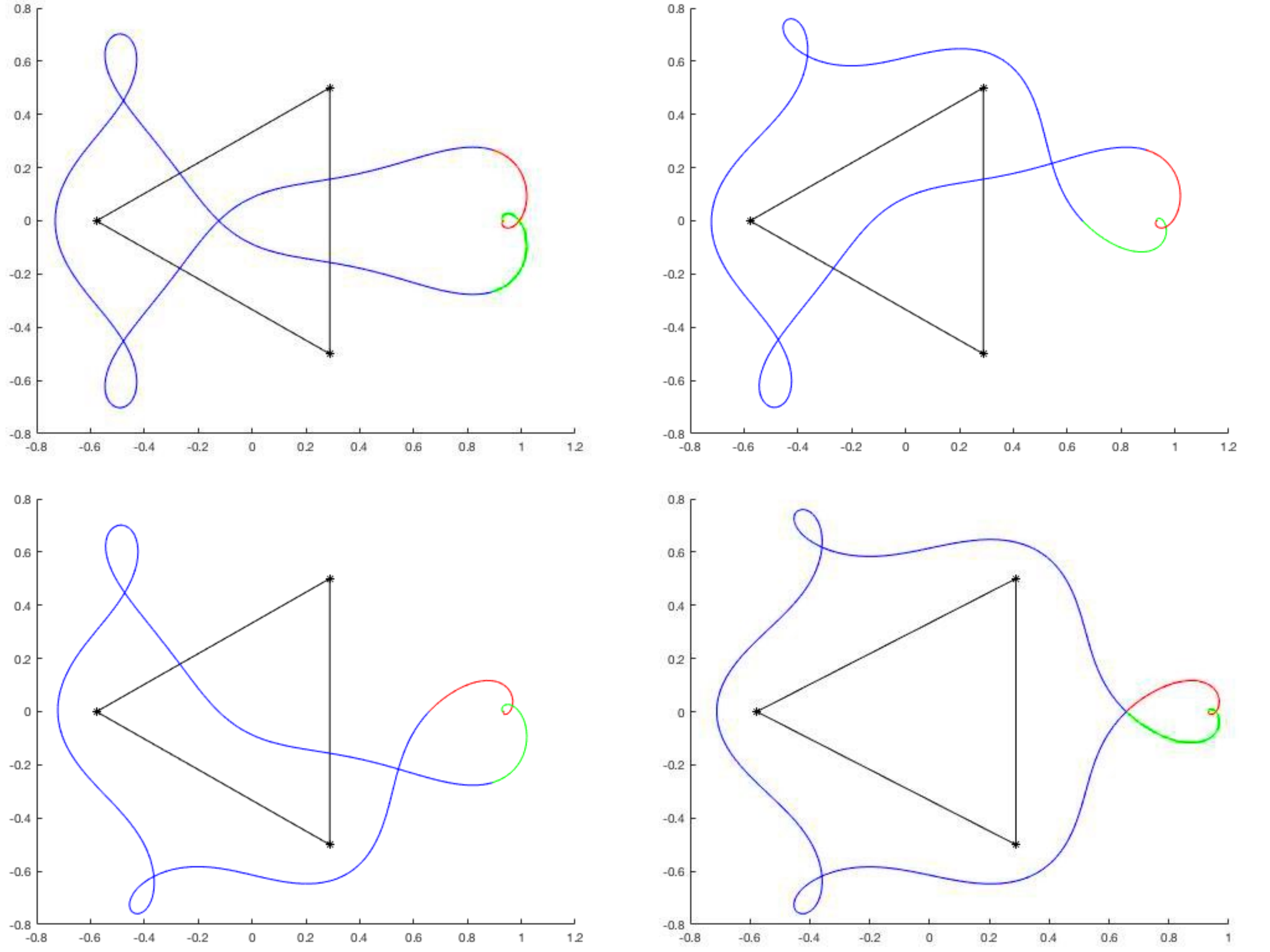}
\caption{\textbf{Orbits at $\mathcal{L}_5$ which wind around the first 
primary}: The orbit on the top left has
nontrivial winding about the largest primary and no winding about the 
smaller primaries.  The orbits in the top right and bottom left frames
wind around the largest primary and one or the other smaller primaries.
Finally the orbit in the bottom right frame winds around all three primaries.
The symbol sequences for these orbits are 
$L_{5C} \cdot L_{6A} \cdot L_{5D}$ --23rd and longest connection (top left),
$L_{5A} \cdot L_{6A} \cdot L_{5D}$ -- 21st (top right),
$L_{5C} \cdot L_{6A} \cdot L_{5B}$ -- 22nd (bottom right)
$L_{5A} \cdot L_{6A} \cdot L_{5B}$ -- 16th connection (bottom right), 
(Blue, green and red solid/dotted lines have the same meaning described in 
Figure \ref{fig:L0_letters}).  Note that while orbit $L_{5A} \cdot L_{6A} \cdot L_{5B}$
is one of the most geometrically complicated in our search, it has only the 
$16$th longest connection time. 
All references to color refer to the online version.
}\label{fig:connectionsL5_long}
\end{figure}

\begin{remark}[Some related work on asymptotic orbits] \label{rem:similarOrbits}
There are interesting similarities between some of the 
orbits discussed above, and some asymptotic orbits already 
discovered in  \cite{papadakisPO_likeUs}.  
The interested reader might for example 
compare the homoclinic orbit on the bottom right frame of our Figure
\ref{fig:connectionsL5_long} with the heteroclinic termination orbit 
illustrated in Figure 5 of \cite{papadakisPO_likeUs}. 
(To make such a comparison one has to ``flip''  Figure 
5 of \cite{papadakisPO_likeUs} $180$
degrees about the $y$-axis as the two papers use different 
normalizations of the four body problem.  Also, their $L_3$ is our 
$\mathcal{L}_5$).
In that study
the heteroclinic is discovered by numerical continuation of the
author's $f_{10}$ family of periodic orbits: 
a family of orbits with winding number one about all three of the 
primary masses.  We note that our homoclinic of Figure
\ref{fig:connectionsL5_long} is similar, but that the 
$\pm 120$ degree rotational symmetry broken.  We conjecture that there are
three families of periodic orbits bifurcating from the $f_{10}$ family after a 
symmetry breaking, and that these families 
terminate on the homoclinic of Figure \ref{fig:connectionsL5_long} (bottom 
right) and its rotation by $\pm 120$ degree counterparts.

Similarly, the heteroclinic orbit illustrated in Figure 4 of
\cite{papadakisPO_likeUs} -- which is the termination of the 
author's $f_5$ family -- is related the pair of homoclinic orbits $L_{E,F}$ 
illustrated in our Figure \ref{fig:connectionsL5_letters}.
To see this, imagine an orbit obtained by combining our $L_{5F}$
with the orbit $L_{4E}$, that is our $L_{5E}$ rotated by $-120$ degrees
so that it is based at $\mathcal{L}_4$.  The resulting union of curves 
has the same shape as the heteroclinic illustrated in 
Figure 4 of \cite{papadakisPO_likeUs}.
This  suggests that the families of periodic orbits which terminate at 
our $L_{E,F}$ could emerge from the planar Lyapunov families 
after symmetry breaking.

In general we note that the homoclinic orbits tend to have less symmetry than the 
heteroclinic, so that studying the periodic orbits terminating at the homoclinics 
is a good way to obtain asymmetric periodic orbits -- even in the symmetric 
versions of the problem.  We also note that changing the mass parameters will 
tend to destroy heteroclinic connections, as the libration points will move into
distinct energy levels.  Homoclinic orbits on the other hand
persist under generic Hamiltonian perturbations of the vector field.
In particular they persist after a small change in the mass ratios,
facilitating numerical continuation as discussed below.    
\end{remark}

\subsection{Numerical continuation of ensembles of connections}
The fact that the homoclinic connecting orbits are formulated as
solutions of boundary value problems makes parameter continuation 
natural.  We give only an outline of our continuation algorithm, as 
numerical continuation of homoclinic orbits for Hamiltonian systems
is described in great detail in the literature.  References are discussed in 
the introduction.  

Begin with an ensemble of connecting orbits 
for a libration point $\mathcal{L}$ at the mass parameters $m_1, m_2, m_3$
(initially we have $m_1 = m_2 = m_3 = 1/3$).  
\begin{itemize}
\item We choose a new parameter set $\bar m_1  = m_1 + \delta_1$, 
$\bar m_3 = m_3 + \delta_3$.  Then we compute $\bar m_2 = 1 - \bar m_1 - \bar m_3$, 
and apply a first order predictor corrector to find the libration point at the
 new parameter values.  We numerically compute the 
 eigenvalues and eigenvectors of the new libration 
 point, and if it remains a saddle-focus (i.e. if there has been no bifurcation) 
 we proceed.  
\item We recompute the local invariant manifolds at the new
parameter set.  A good strategy is to compute the coefficients to order $N_0$
by recursively solving the homological equations.  Initially we take the eigenvector scaling 
from the previous step and rescale if needed.  
For the higher order coefficients we use the coefficients from the 
previous step.  This gives as an initial guess for the Newton or 
 pseudo-Newton method which usually converges very fast.  
 \item The new local parameterizations provide the boundary 
 conditions for the multiple shooting scheme for the homoclinic orbits.
We take the connecting orbits from the previous step as the initial
guesses for the Newton method at the current mass parameters.
If necessary we can apply a first order predictor corrector, but this is often 
unnecessary, due to the fact that the boundary value problem formulated
with the high order parameterizations of the local manifold is very well 
conditioned.   Note that in a given continuation step, the same local 
parameterizations serve 
as the boundary conditions for the entire ensemble of connecting orbits.
This justifies the cost of computing high order representations of the manifolds.  
\item Once we have applied Newton to all the connections in the ensemble
we are ready to take a new step.  If Newton fails to converge for any of the 
connecting orbits we have to decide if we throw the orbit away, or if we recompute
with smaller $\delta_1, \delta_2, \delta_3$. 
\end{itemize}

We also remark that the atlas is not recomputed at the new mass parameter set.
That is, we continue only the connecting orbits -- the intersections of the 
stable unstable manifolds -- not the manifolds themselves.  Continuation of 
ensembles of connections is much cheaper than recomputing the atlas
each time we change parameters.  

Results of several numerical continuations are illustrated in 
Figures \ref{fig:continuationAtL0} and 
\ref{fig:continuationAtL5}.  As we change the masses we break the 
rotational symmetry of the Triple Copenhagen problem and the symmetric
counterparts resolve into distinct connection, no longer obtainable by rotations
of a single representative.
During the numerical continuation we sometimes encounter bifurcations of 
the connecting orbits themselves, which involve no bifurcation of the underlying 
equilibrium.  Figure \ref{fig:homoclinicDoubling} illustrates a common scenario
where a family of homoclinic orbits undergoes a doubling bifurcation.  These 
bifurcations seem very common and we have not made a systematic 
effort to track them.  This would make an interesting topic for a future study.

\begin{figure}[!t]
\centering
\includegraphics[width=6in]{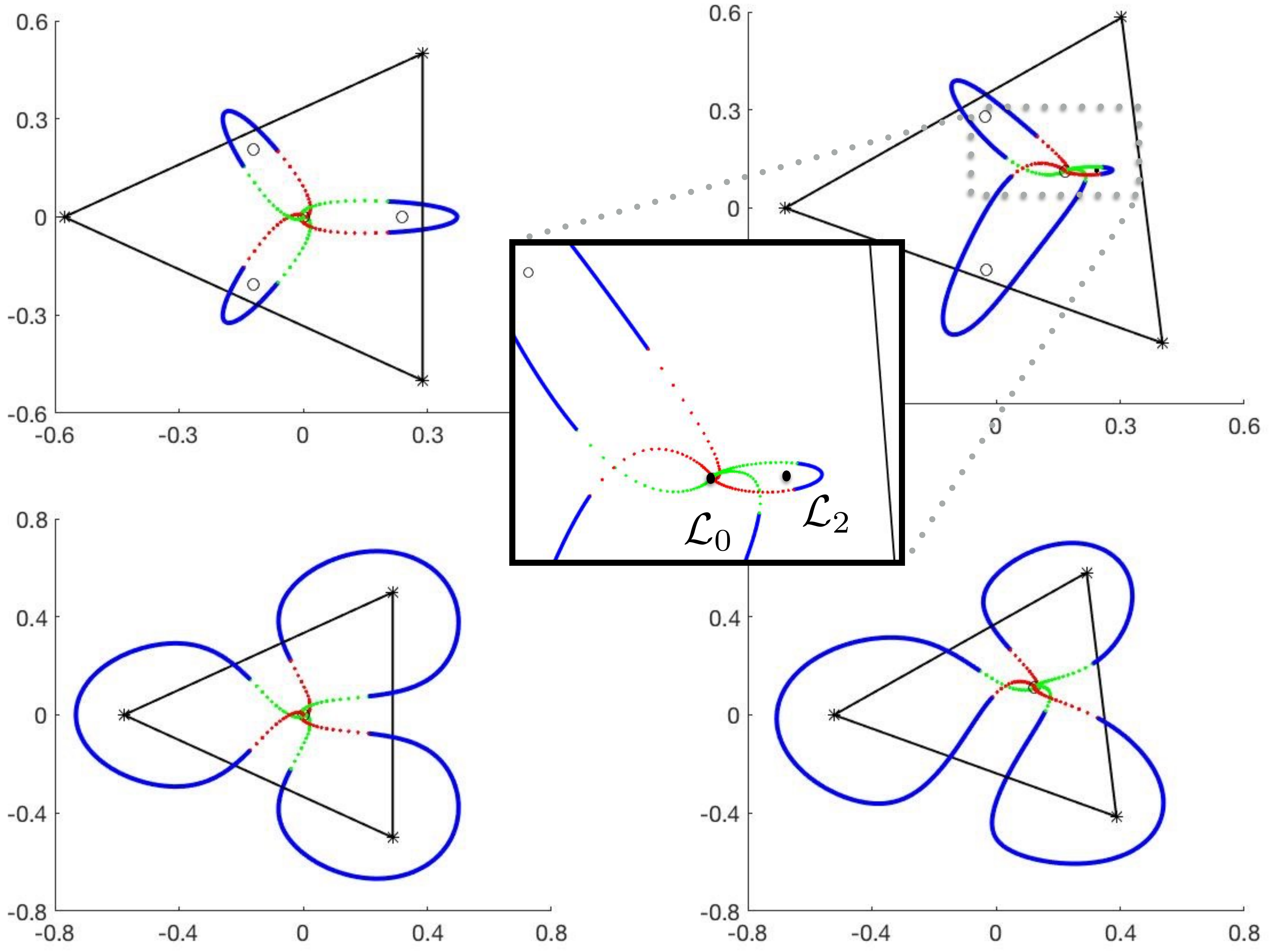}
\caption{\textbf{Continuation at $\mathcal{L}_0$:} of all six fundamental connecting orbits at
$\mathcal{L}_0$ and their symmetric counterparts.  Open circles in the top right frame
depict the locations of the inner libration points $\mathcal{L}_{1,2,3}$.
(Blue, green and red solid/dotted lines have the same meaning described in 
Figure \ref{fig:L0_letters}). Top/bottom left: fundamental connections in the 
triple Copenhagen problem -- equal masses. Top right: final result of numerically continuing
$L_{0A}$ and its symmetric counterparts
along the line in parameter space beginning at $m_1 = m_2 = m_3 = 1/3$
and ending at  $m_1 = 0.415,  m_2 = 0.3425, m_3 = 0.2425$--
close to where $\mathcal{L}_0$ loses  saddle-focus stability near $m_1 \approx 0.42$.
During the continuation $L_{0A^+}$ shrank substantially.  This is due to the 
fact that $L_{0A^+}$ winds around $\mathcal{L}_2$, which collides with $\mathcal{L}_0$
when $m_1 \approx 4.2$.  A close-up of the situation is illustrated in the 
center frame.  We observe that the connections $L_{0A}$ and $L_{0A^-}$ are deformed
much less dramatically.
Bottom right: the result of numerically continuing
$L_{0B}$ and its symmetric counterparts along the line in parameter space beginning at
$m_1 = m_2 = m_3 = 1/3$, and ending at $m_1 = 0.4$, $m_2 = 0.35$, $m_3 = 0.25$. 
These orbits are also deformed
less dramatically, though the loops do seem to decrease in size according to 
the loss of mass in the respective primary, with the largest loop around $m_1$ and 
the smallest around $m_3$.  All references to color refer to the online version.
}\label{fig:continuationAtL0}
\end{figure}

\begin{figure}[!t]
\centering
\includegraphics[width=6in]{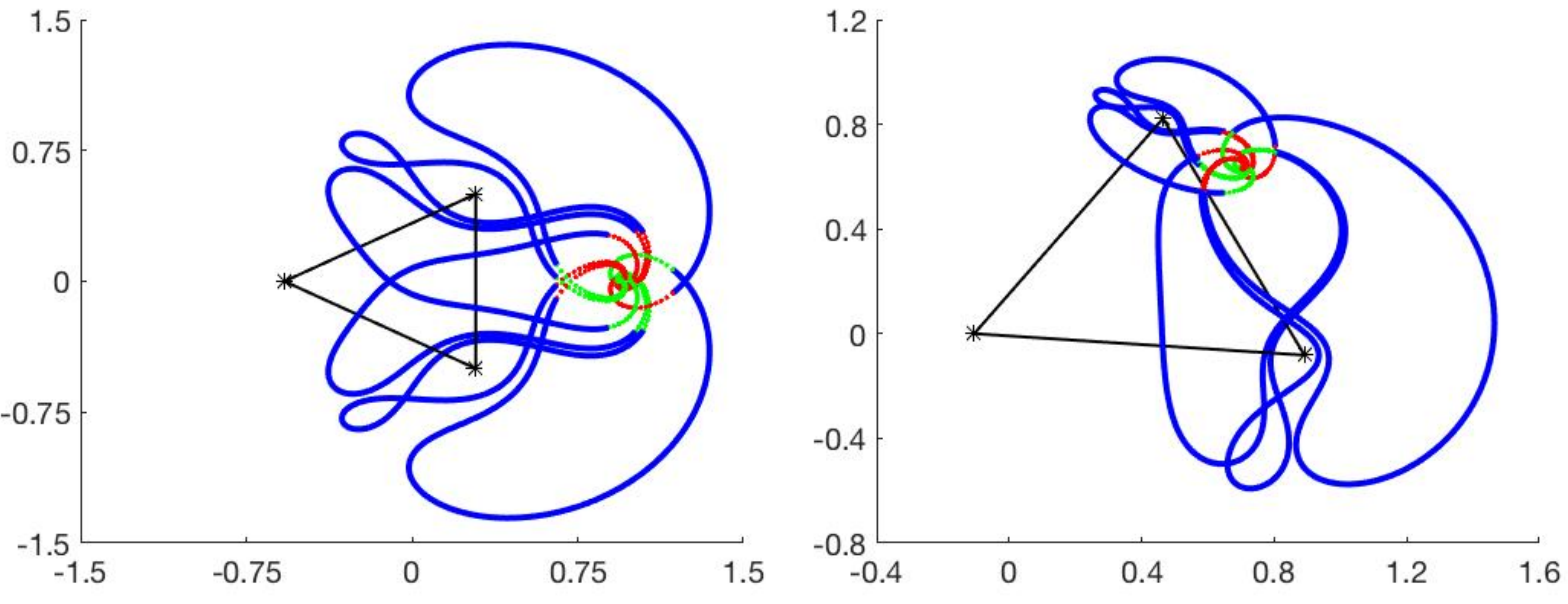}
\caption{\textbf{Continuation at $\mathcal{L}_5$}: continuation of the 
six basic homoclinic motions at $\mathcal{L}_5$, along the line in parameter
space beginning at $m_1 = m_2 = m_3 = 1/3$ and ending at
$m_1 = 0.89, m_2 =  0.1, m_3 =  0.01$.  (Blue, green and red solid/dotted 
lines have the same meaning described in Figure \ref{fig:L0_letters})  
So, starting from the triple Copenhagen problem, we deform until almost ninety 
percent of the mass is in the first primary body -- near the 
loss of saddle-focus stability of $\mathcal{L}_5$.  
Observe that the libration point $\mathcal{L}_5$ moves closer to the 
smallest primary $m_3$, and that the loops contract around the smallest
primary, a similar situation to that discussed in the caption of Figure \ref{fig:continuationAtL0}.
All references to color refer to the online version.
}\label{fig:continuationAtL5}
\end{figure}

\begin{figure}[!h]
\centering
\includegraphics[width=6in]{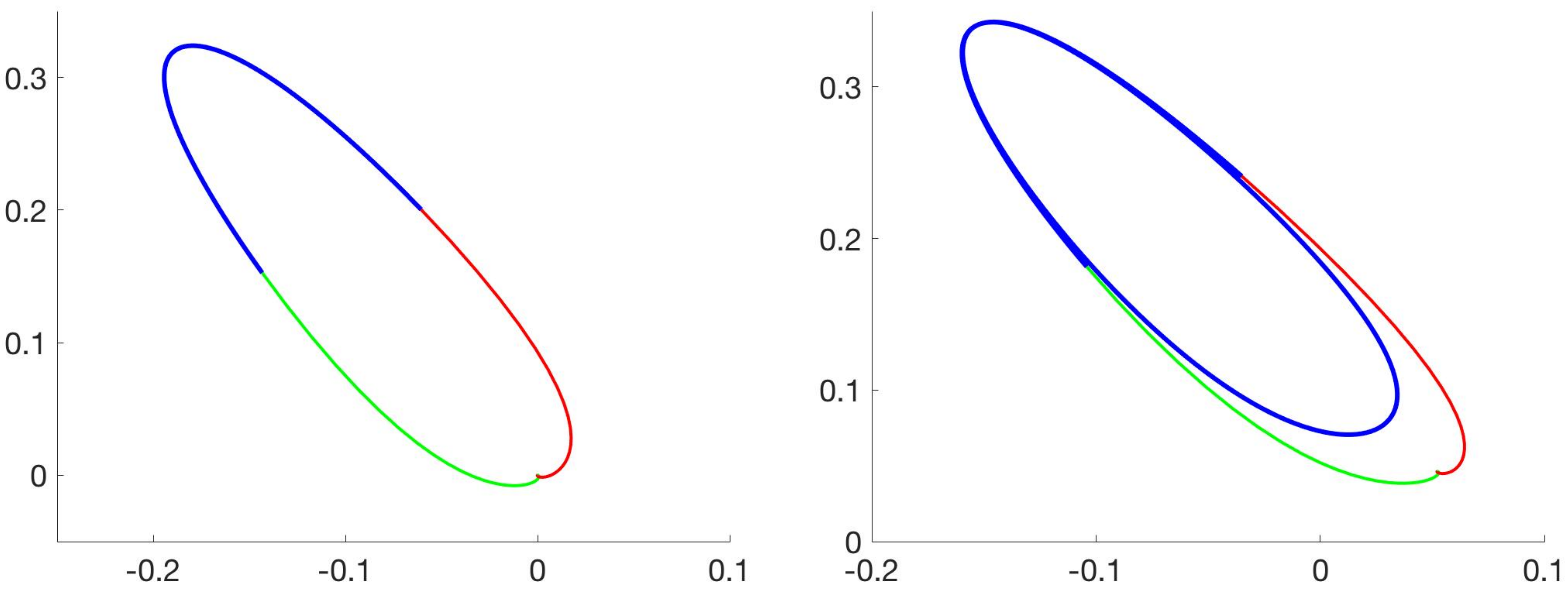}
\caption{\textbf{Illustration of a typical bifurcation:}
On the left is a homoclinic orbit for $\mathcal{L}_0$ in the 
$m_1 = m_2 = m_3 = 1/3$ case.  On the right is a connecting orbit
for the parameter values $m_1 = 0.3617$, $m_2 = 0.34042$ and 
$m_3 = 0.298$, just after a homoclinic doubling bifurcation
of the orbit on the left.  Both orbits persist after the bifurcation, 
that is it seems to be a pitchfork. The new homoclinic 
has a close ``flyby'' of $\mathcal{L}_0$ before making a second excursion 
and finally landing on the stable manifold.  In both frames the solid blue lines
represent the portion of the orbit represented by the boundary value problem, 
while the green and red are portions on the unstable and stable manifolds 
respectively.  All references to color refer to the online version.
}\label{fig:homoclinicDoubling}
\end{figure}


\section{Conclusions}
In this work we implemented a numerical method for computing an atlas for the stable/unstable 
manifold attached to a libration point in the CRFBP.  We consider saddle-focus equilibria,
as in this case topological theorems give rich dynamical structure near a transverse homoclinic. 
We then developed algorithms for searching or ``mining'' the atlas for approximate connections.
After an approximate connection is found we sharpen it using a Newton scheme for an appropriate
boundary value problem.  
The procedure is completely automated, and locates all the homoclinic orbits up to a 
specified integration time.  (To make the calculations less stiff we ignore orbits passing 
too close to the primaries).
The resulting collection of homoclinic orbits is ordered by connection time.
A different choice of local parameterization may yield different connection times,
but the ordering of these connections is universal. 
This last comment requires that the local stable/unstable manifolds are computed using the 
parameterizaiton method.

Our method locates dozens of distinct homoclinic connections and we consider their
qualitative properties in addition to their connection times.  The geometry of the homoclinic orbit set is 
organized by the several shortest connections, in the sense that they form 
a system of channels that other connections appear to follow.   
We decompose the motions of the longer connections into words built from these simple letters,
and discuss briefly how this decomposition could be calculated in an automated way 
using the formulation of the winding number as a complex line integral.  

Finally, we continue some of the orbits found in the equal mass case  
to other non-symmetric mass values using a predictor corrector scheme for 
the boundary value problem.  Rather than recomputing the entire atlas the continuation 
scheme only recomputes the initial parameterization at the new masses, and this 
can be done via a Newton scheme.  

We remark that it would be easy to intersect that atlas data computed here with any
desired surface of section.  We have not used sections in the present work
because (a) we wanted to find all the connections up to a certain integration time 
and a given section may find some orbits and not others, and (b)
projecting to a section may not preserve the ordering of the homoclinics by connection time. 
Moreover,  while the present project focuses on the CRFBP -- a two degree of freedom
Hamiltonian system -- the scheme described here could be extended to higher dimensional 
systems were sections are four or more dimensional and hence less useful for 
visualization purposes.  In such a situation, for example the spatial 
CRFBP, it is desirable to have an automated procedure.

Other interesting topics of future research would be to combine our methods with more sophisticated
continuation and branch following algorithms. 
It would also be nice to return to the 
ideas of Str\"{o}mgren, and examine the ``tubes'' of periodic orbits attached to each of our 
homoclinic connections.  These periodic families would themselves undergo bifurcations
which one could try to follow numerically.  

Another improvement to our method would be to remove the speed constraints
on our manifold computations.  This could be done by regularizing binary collisions.
The idea would be that whenever a chart gets too close to a primary, 
then instead of subdividing we would change to the regularized coordinates where
computations are less stiff.  This idea of using such regularizations 
to improve numerics goes back at least to the work of Thiele.  
This would also provide a natural way for computing collision orbits between $\mathcal{L}_{0, 5}$
and each of the primaries.  A topic we have avoided via our imposed speed constraints.
A modern implementation combined with our approach to computing atlases would be valuable,
and is the subject of ongoing work.

If such advancements let us compute larger and more complete atlases, a very interesting 
question is to see if other ``fundamental'' connecting orbits appear.  For example at $\mathcal{L}_0$
all the connections we find shadow two basic orbits $L_{0A}$, $L_{0B}$ and their symmetric counterparts.  
Is this true of all the connections?  Or is this simply an artifact of the fact that we only consider
connections whose velocity is never too large?  Will performing longer searches yield more 
fundamental letters for the alphabets at $\mathcal{L}_{0,5}$?

Of course with more computing power one could perform the atlas computations at more values of 
the mass parameters, say for a mesh of ten or twenty different points in the 
simplex $m_1 + m_2 + m_3 = 1$.  This would provide a more complete 
picture of the global orbit structure.  Such a project would greatly benefit form a cluster computing 
implementation exploiting the data independence of the computations at different parameter sets, and 
indeed the independence of different portions of the atlas at a given parameter set.  
Numerical continuation could then be applied to ``fill in the gaps'' between the 
mesh points.

\begin{acknowledgements}
The authors would like to sincerely thank two anonymous referees who carefully 
read the submitted version of the manuscript.  Their suggestions 
greatly improved the final version.
The second author was partially supported by NSF grant 
DMS-1813501.
Both authors were partially supported by NSF grant DMS-1700154 
and by  the Alfred P. Sloan Foundation grant G-2016-7320.
\end{acknowledgements}

\clearpage
\appendix

\section{Rotational symmetry for the equal mass case}
\label{sec:CRFBP_symmetry}
Let $m_1 = m_2 = m_3 = 1/3$ and 
$\theta = \frac{2 \pi }{3}$. Define the linear map, $\varphi : \rr^4 \to \rr^4$, by
\[
\varphi(x,\dot{x},y,\dot{y}) = 
\left(
\begin{array}{cccc}
\cos (\theta) & 0 & - \sin(\theta) & 0 \\
0 & \cos (\theta) & 0 & - \sin(\theta) \\
\sin(\theta) &  0 & \cos(\theta)  & 0 \\
0 & \sin(\theta) &  0 & \cos(\theta)  \\
\end{array}
\right)
\left(
\begin{array}{c}
x \\
\dot{x} \\
y \\
\dot{y}
\end{array}
\right)
=(\varphi_1, \varphi_2, \varphi_3, \varphi_4)^T.
\]
Note that $\varphi$ acts as a rotation by $\theta$ in the $(x,y)$ and $(\dot{x},\dot{y})$ coordinate planes independently. Now, suppose that $\mathbf{x}: \rr \to \rr^4$ is a trajectory for $f$, then $\tilde{\mathbf{x}} = \varphi \circ \mathbf{x}$ is also a trajectory for $f$. Moreover, 
if $\mathbf{x} \subset W^{s,u}(\mathcal{L}_i)$ for $i \in \{0,4,5,6\}$, then $\tilde{\mathbf{x}} \subset W^{s,u}(L_{\sigma(i)})$,  where $\sigma$ is the permutation given by $\sigma = (0)(4,5,6)$.
\begin{proof}
	Let $\hat x = (x,  \dot{x},y,\dot{y}) \in \rr^4$ and suppose $\mathbf{x}$ is the trajectory through $\hat x$ satisfying $\mathbf{x}(0) = \hat x$. By definition, $\tilde{\mathbf{x}}(0) = \varphi(\mathbf{x}(0)) = \varphi(\hat x)$, and we note that $\tilde{\mathbf{x}}$ will parameterize a trajectory for $f$ if and only if $\tilde{\mathbf{x}}(t)$ is tangent to $f(\tilde{\mathbf{x}}(t))$ for all $t \in \rr$. Thus, it clearly suffices to prove that $f \circ \varphi = \varphi \circ f$ holds for any $\hat x$ on $\rr^4$. 
	
	With this in mind, define the planar rotation $\eta: \rr^2 \to \rr^2$ by 
	\[
	\eta(x,y) =
	\left(
	\begin{array}{cc}
	\cos (\theta) & - \sin(\theta) \\
	\sin(\theta) & \cos(\theta)  \\
	\end{array}
	\right)
	\left(
	\begin{array}{c}
	x \\
	y 
	\end{array}
	\right)
	= 
	\left(
	\begin{array}{c}
	\eta_1(x,y) \\
	\eta_2(x,y) 
	\end{array}
	\right)
	\]
	Recall that for the symmetric mass case, we have equal masses given by $m_1 = m_2 = m_3 = \frac{1}{3}$. Set $m = \frac{1}{3}$, then the primaries are located at $P_1,P_2,P_3$ given by 
	\[
	P_1 = \left(-\frac{\sqrt{3}}{3}, 0 \right) \quad P_2 = \left(\frac{\sqrt{3}}{6}, -\frac{1}{2} \right) \quad P_3 = \left(\frac{\sqrt{3}}{6}, \frac{1}{2} \right)
	\]
	and note that $\norm{P_1} = \norm{P_2} = \norm{P_3} = \frac{1}{\sqrt{3}}$. Moreover, $P_1,P_2,P_3$ are vertices of an equilateral triangle and a direct computation shows that 
	$\eta$ acts as a cyclic permutation on the primary bodies in configuration space given by the cycle $\pi = (1,2,3)$.  Recalling that $r_i(x,y) = \sqrt{(x-x_i)^2 + (y-y_i)^2} = \norm{(x,y) - P_i}$, it follows from this symmetry that for $i \in \{1,2,3\}$ we have
	\begin{equation}
		\label{eq:r_compose_eta}
	r_i \circ \eta(x,y) = \norm{\eta(x,y) - P_i} = \norm{(x,y) - P_{\pi^{-1}(i)}} = r_{\pi^{-1}(i)}.
	\end{equation}
	Now, we recall that in the symmetric case, the CRFBP vector field is given by
	\[
	f(x,\dot{x},y,\dot{y}) = 
	\left(
	\begin{array}{c}
	\dot{x} \\
	2\dot{y} + x - \frac{1}{3}\sum_{i=1}^{3} \frac{x - x_i}{r_i} \\
	\dot{y} \\
	-2\dot{x} + y - \frac{1}{3}\sum_{i=1}^{3} \frac{y - y_i}{r_i} 
	\end{array}
	\right).
	\]
	which we write in scalar coordinates as $f = \left(f_1,f_2,f_3,f_4\right)$. Similarly, write $\varphi = \left(\varphi_1,\varphi_2,\varphi_3,\varphi_4 \right)$ and we note that $(\varphi_1(\hat x),\varphi_3(\hat x)) = \eta(x,y)$. Now, we check that $f_i \circ \varphi = \varphi_i \circ f$ holds for each $i \in \{1,2,3,4\}$. For $i =1$, we have the direction computation
	\[
	\varphi_1 \circ f \left(\hat x\right) = \dot{x} \cos(\theta) - \dot{y} \sin(\theta) = f_1 \circ \varphi \left(\hat x\right).
	\]
	Now, for $i = 2$ we first compute each expression
	\begin{align*}
	\varphi_2 \circ f \left(\hat x\right) & = \left(2 \dot{y} + x - \frac{1}{3}\sum_{i=1}^{3} \frac{x - x_i}{r_i(x,y)}\right) \cos(\theta) 
	- \left(-2\dot{x} + y - \frac{1}{3} \sum_{i=1}^{3} \frac{y - y_i}{r_i(x,y)}\right) \sin(\theta)   \\
	f_2 \circ \varphi \left(\hat x\right) & = 2(\dot{x} \sin(\theta) + \dot{y} \cos(\theta)) + x \cos(\theta) - y \sin(\theta) - \frac{1}{3} \sum_{i = 1}^{3} \frac{\eta_1(x,y) - x_i}{r_i \circ \eta (x,y)}.
	\end{align*}
	After canceling like terms in each expression, we are left to prove the following equality
	\begin{equation}
	\label{eq:reduced_equality}
	\sum_{i = 1}^{3} \frac{\eta_1(x,y) - x_i}{r_i \circ \eta(x,y)} = \cos(\theta) \sum_{i = 1}^{3} \frac{x - x_i}{r_i(x,y)} - \sin(\theta) \sum_{i=1}^{3} \frac{y - y_i}{r_i(x,y)}.
	\end{equation}
	Applying the result from \eqref{eq:r_compose_eta} to the left side we have 
	\[
	\sum_{i = 1}^{3} \frac{\eta_1(x,y) - x_i}{r_{\pi^{-1}(i)}} = \frac{\eta_1(x,y) - x_1}{r_3(x,y)} + \frac{\eta_1(x,y) - x_2}{r_1(x,y)} + \frac{\eta_1(x,y) - x_3}{r_2(x,y)}
	\]
	so that for each $i \in \{1,2,3\}$, the numerator for $r_i$ is given by $\eta_1(x,y) - x_{\pi(i)}$. Now, we compute the numerators for $r_i(x,y)$ on the right hand side as 
	\[
	\cos(\theta)(x-x_i) - \sin(\theta)(y-y_i) = \eta_1(x,y) - \eta_1(x_i,y_i) = \eta_1(x,y) - x_{\pi(i)}.
	\]
	We conclude that the numerators for each $r_i$ are equal, and therefore,  the equality in \eqref{eq:reduced_equality} holds which proves that $\varphi_2 \circ f = f_2 \circ \varphi$. The proofs for the $i = 3,4$ cases are computationally similar to the corresponding proofs for $i = 1,2$ which concludes the proof that $f \circ \varphi = \varphi \circ f$, or equivalently, $\tilde{\mathbf{x}}$ is a trajectory for $f$. 
	
	To prove the second claim, fix $i \in \{0,4,5,6\}$ and suppose $\mathbf{x}(t) \to L_i$ as $t \to \infty$ implying that  $\mathbf{x} \subset W^s(L_i)$. Let $\tilde{\mathbf{x}} = \varphi(\mathbf{x})$, and note that $L_i$ is an equilibrium solution for $f$ implying that $\mathbf{x}_{2,4}(t) \to 0$. Noting that $\eta$ is a unitary operator, it follows that $\tilde{\mathbf{x}}_{2,4}(t) \to 0$ as well.  
	 Moreover, $\varphi$ is a dynamical conjugacy implying that in configuration space we have  
	\[
	\lim\limits_{t \to \infty} \tilde{\mathbf{x}}_{1,3}(t) = \lim\limits_{t \to \infty} \eta\left(x(t),y(t)\right) = \eta(L_i).
	\]
	Taken together it follows that $\eta(L_i)$ is again an equilibrium solution for $f$. Thus, $\eta$ acts as a permutation on equilibria. A direct computation shows that $\eta(L_i) = L_{\sigma(i)}$ where $\sigma$ is the permutation given by $\sigma = (0)(4,5,6)$. The preceding argument applies equally well to the unstable manifold of each equilibrium with $t \to -\infty$ which completes the proof of the second claim. 
\end{proof}

\section{Power series manipulation, automatic differentiation, and the radial gradient} \label{sec:formalSeries}
Our local invariant manifold computations are based on formal power series manipulations. 
The main technical challenge is to 
compute $f \circ P$ with $P$ an arbitrary power series 
and $f$ the vector field for the CRFBP.  
As usual in gravitational $N$ body problems, the nonlinearity contains 
terms raised to the minus three halves power. 

Consider two formal power series $P, Q \colon \mathbb{C}^2 \to \mathbb{C}$
given by 
\[
P(z_1, z_2) = \sum_{m=0}^\infty \sum_{n=0}^\infty 
a_{m,n} z_1^m z_2^n, 
\quad \quad \mbox{and} \quad \quad 
Q(z_1, z_2) = \sum_{m=0}^\infty \sum_{n=0}^\infty 
b_{m,n} z_1^m z_2^n,
\]
where $a_{m,n}, b_{m,n} \in \mathbb{C}$ for all $(m,n) \in \mathbb{N}^2$.
The collection of all formal power series forms a complex vector space, so that 
for any $\alpha, \beta \in \mathbb{C}$ we have that 
\[
(\alpha P+ \beta Q)(z_1, z_2) = \sum_{m=0}^\infty \sum_{n=0}^\infty
\left(\alpha a_{m,n} + \beta b_{m,n}\right)  z_1^m z_2^n.
\]
The collection becomes an algebra when endowed with 
the Cauchy product 

\begin{equation}
\label{eq:cauchy_product}
(P \cdot Q)(z_1, z_2) = 
 \sum_{m=0}^\infty \sum_{n=0}^\infty
 \left(\sum_{j=0}^m \sum_{k=0}^n
a_{m-j,n-k} b_{j k} \right) \, z_1^m z_2^n.
\end{equation}

We evaluate elementary functions of formal power series 
using a technique called \textit{automatic differentiation} by many authors.  
Suppose for example we are given a formal series 
\[
P(z_1, z_2) = \sum_{m=0}^\infty \sum_{n=0}^\infty
p_{m,n}
z_1^m z_2^n,
\]
with $p_{0,0} \neq 0$.  We seek the formal series coefficients $q_{m,n}$ of the function 
\[
Q(z_1, z_2) = 
\sum_{m=0}^\infty \sum_{n=0}^\infty
q_{m,n}
z_1^m z_2^n = 
P(z_1, z_2)^{\alpha}, \quad \quad \quad \alpha \in \mathbb{R}.
\]
%

Our approach follows the discussion given by Alex Haro in 
\cite{mamotreto}. Consider the 
first order partial differential operator 
\[
\nabla_{\mbox{\tiny rad}} P(z_1, z_2) = \nabla P(z_1, z_2) \left(
\begin{array}{c}
z_1 \\
z_2
\end{array}
\right) = 
z_1 \frac{\partial}{\partial z_1} P(z_1, z_2) + z_2 \frac{\partial}{\partial z_2} P(z_1, z_2),
\]
which is referred to as \textit{the radial gradient} of $P$.
Evaluating on the level of formal power series leads to 
\[
\nabla_{\mbox{\tiny rad}} P(z_1, z_2) 
= \sum_{m=0}^\infty \sum_{n=0}^\infty (m+n) p_{m,n} z_1^m z_2^n.
\]
Observe that 
\begin{align*}
\nabla_{\mbox{\tiny rad}} Q(z_1, z_2) &= 
\nabla  Q(z_1, z_2) \left(
\begin{array}{c}
z_1 \\
z_2
\end{array}
\right) \\
&= \nabla P(z_1, z_2)^\alpha \left(
\begin{array}{c}
z_1 \\
z_2
\end{array}
\right) \\
&= \alpha P(z_1, z_2)^{\alpha - 1} \nabla P(z_1, z_2) \left(
\begin{array}{c}
z_1 \\
z_2
\end{array}
\right).
\end{align*}
Multiplying both sides of the equation by $P$ we obtain
\begin{equation} \label{eq:radialGradInvEq}
P(z_1, z_2) \nabla
 Q(z_1, z_2) \left(
\begin{array}{c}
z_1 \\
z_2
\end{array}
\right) = \alpha Q(z_1, z_2) \nabla
P(z_1, z_2) \left(
\begin{array}{c}
z_1 \\
z_2
\end{array}
\right).
\end{equation}
Here the fractional power is replaced by 
operations involving only differentiation and multiplication.  
This is the virtue of the radial gradient in automatic differentiation schemes.
Plugging the power series expansions into Equation \eqref{eq:radialGradInvEq}
leads to 
\[
\left(  
\sum_{m=0}^\infty \sum_{n=0}^\infty p_{m,n} z_1^m z_2^n
\right)
\left(
\sum_{m=0}^\infty \sum_{n=0}^\infty 	(m+n) q_{m,n} z_1^m z_2^n
\right) =
\]
\[
\left(  
\sum_{m=0}^\infty \sum_{n=0}^\infty \alpha q_{m,n} z_1^m z_2^n
\right)
\left(
\sum_{m=0}^\infty \sum_{n=0}^\infty 	(m+n) p_{m,n} z_1^m z_2^n
\right),
\]
and taking Cauchy products gives 
\[
\sum_{m=0}^\infty \sum_{n=0}^\infty
\sum_{j=0}^m \sum_{k=0}^n  (j+k) p_{m-j, n-k} q_{j,k} z_1^m z_2^n
\]
\[
 =
 \sum_{m=0}^\infty \sum_{n=0}^\infty
\sum_{j=0}^m \sum_{k=0}^n  \alpha (j+k) q_{m-j, n-k} p_{j,k} z_1^m z_2^n.
\]
Match like powers to get
\[
\sum_{j=0}^m \sum_{k=0}^n  (j+k) p_{m-j, n-k} q_{j,k} = 
\sum_{j=0}^m \sum_{k=0}^n  \alpha (j+k) q_{m-j, n-k} p_{j,k},
\]
or 
\[
(m+n) p_{0,0} q_{m,n} + \sum_{j=0}^m \sum_{k=0}^n  \hat{\delta}_{j,k}^{m,n} (j+k) p_{m-j, n-k} q_{j,k}
\]
\[
= \alpha (m+n) q_{0,0} p_{m,n} + \sum_{j=0}^m \sum_{k=0}^n \hat{\delta}_{j,k}^{m,n} \alpha (j+k) q_{m-j, n-k} p_{j,k},
\]
for $m + n \geq 1$.  Here 
\[
\hat{\delta}_{j,k}^{m,n} := \begin{cases}
0 & \mbox{if } j = m   \mbox{ and } k = n \\
0 & \mbox{if } j = 0  \mbox{ and } k = 0 \\
1 & \mbox{otherwise}
\end{cases}.
\]
The $\hat \delta$ appears to remind us that terms of 
order $(m,n)$ are extracted from the sum.   
Isolating $q_{m,n}$ gives 
\begin{equation} \label{eq:recForAlphaPower}
q_{m,n} = \alpha p_{0,0}^{\alpha-1} p_{m,n} + \frac{1}{(m+n) p_{0,0}} 
\sum_{j=0}^m \sum_{k=0}^n \hat{\delta}_{j,k}^{m,n} (j+k)
\left(
\alpha q_{m-j, n-k} p_{j,k} - p_{m-j, n-k} q_{j,k}
\right), 
\end{equation}
for $m + n \geq 1$.  Note that $q_{0,0} = p_{0,0}^{\alpha} \neq 0$ by 
hypothesis, so that the coefficients $q_{m,n}$ are formally well defined to all orders.    
Using the recursion given in Equation \eqref{eq:recForAlphaPower}
we can compute the formal series coefficients for $Q$ for the cost of a 
Cauchy product.  
This allows us to compute power series representations for the nonlinear terms in 
$f(P)$ and $Df(P)$ in the CRFBP.  
Another approach which converts the CRFB field to a higher dimensional polynomial 
filed in discussed in \cite{shaneAndJay}.

\section{Conflict of interest statement}
The authors of this manuscript certify that they have no affiliations with or involvement in any
organization or entity with any financial interest (such as honoraria; educational grants,
 participation in speaker's bureaus,
membership, employment, consultancies, stock ownership, 
or other equity interest, and expert testimony or patent-licensing
arrangements), or non-financial interest 
(such as personal or professional relationships, affiliations, knowledge or beliefs) in
the subject matter or materials discussed in this manuscript.


\end{document}